\pdfoutput=1
\documentclass[10pt,a4paper]{amsart}

\usepackage{comment}
\usepackage[bookmarksdepth=3]{hyperref}

\def\equationautorefname~#1\null{Equation~(#1)\null}

\usepackage[marginpar=2cm,ignoremp,margin=3cm]{geometry}

\usepackage{todonotes}
\usepackage{xfrac}
\usepackage{ textcomp }
\usepackage{shuffle}

\usepackage[normalem]{ulem}
\usepackage{amssymb}

\usepackage{braket}

\usepackage{mathtools}
\usepackage{thmtools}

\declaretheorem[
style=plain,
name=Theorem,
numberwithin=section,
refname={Theorem,Theorems},
Refname={Theorem,Theorems}
]{Thm}
\declaretheorem[
style=plain,
name=Proposition,
numberlike=Thm,
refname={Proposition,Propositions},
Refname={Proposition,Propositions}
]{Prop}
\declaretheorem[
style=plain,
name=Lemma,
numberlike=Thm,
refname={Lemma,Lemmas},
Refname={Lemma,Lemmas}
]{Lem}
\declaretheorem[
style=plain,
name=Corollary,
numberlike=Thm,
refname={Corollary,Corollaries},
Refname={Corollary,Corollaries}
]{Cor}
\declaretheorem[
style=definition,
name=Definition,
numberlike=Thm,
refname={Definition,Definitions},
Refname={Definition,Definitions},
]{Def}

\declaretheorem[
style=definition,
name=Example,
numberlike=Thm,
refname={Example,Examples},
Refname={Example,Examples},
]{Eg}
\declaretheorem[
style=plain,
name=Conjecture,
numberlike=Thm,
refname={Conjecture,Conjectures},
Refname={Conjecture,Conjectures},
]{Conj}
\declaretheorem[
style=definition,
name=Remark,
numberlike=Thm,
refname={Remark,Remarks},
Refname={Remark,Remarks},
]{Rem}

\renewcommand{\vec}[1]{\uline{{\boldsymbol #1}}}

\newcommand{\Z}{\mathbb{Z}}
\newcommand{\Q}{\mathbb{Q}}
\newcommand{\C}{\mathbb{C}}
\newcommand{\R}{\mathbb{R}}
\newcommand{\wt}{{\mathrm{wt}}}

\DeclareMathOperator{\Li}{Li}
\DeclareMathOperator{\QS}{QSym}

\newcommand{\abs}[1]{\left\lvert #1 \right\rvert}
\renewcommand{\Im}{\operatorname{Im}}

\renewcommand{\r}{\overset{r}{*}}
\newcommand{\zt}{\zeta}

\renewcommand{\>}{\rangle}

\newcommand{\ii}{\mathrm{i}}

\DeclareMathOperator{\csch}{csch}
\DeclareMathOperator{\arctanh}{arctanh}

\renewcommand{\H}{\mathfrak{H}}

\newcommand{\dd}{\mathrm{d}}
\newcommand{\mot}{\mathfrak{m}}
\DeclareMathOperator{\reg}{reg}

\newcommand{\half}{{\sfrac{1\!}{2}}}
\newcommand{\ee}{\uline{e}}
\newcommand{\Lic}{\operatorname{\mathcal{L}i}}
\newcommand{\Tic}{\operatorname{\mathcal{L}i}^t}

\let\eps\varepsilon
\let\eps\eps

\newcommand{\asym}{\mathrm{asym}}

\newcommand{\sgnarg}[2]{
	\!\left( \genfrac{}{}{0pt}{0}{#1}{#2} \right)
}
\newcommand{\sgnargsm}[2]{
	\left( \textstyle\genfrac{}{}{0pt}{1}{#1}{#2} \right)
}

\newcommand{\poch}[2]{\left\{ #1 \right\}_{#2}}
\newmuskip\pFqmuskip

\newcommand*\pFq[6][8]{%
	\begingroup 
	\pFqmuskip=#1mu\relax
	\mathcode`\,=\string"8000
	\begingroup\lccode`\~=`\,
	\lowercase{\endgroup\let~}\pFqcomma
	{}_{#2}F_{#3}{\left[\genfrac..{0pt}{}{#4}{#5};#6\right]}%
	\endgroup
}
\newcommand{\pFqcomma}{\mskip\pFqmuskip}

\makeatletter
\let\@@pmod\pmod
\DeclareRobustCommand{\pmod}{\@ifstar\@pmods\@@pmod}
\def\@pmods#1{\mkern4mu({\operator@font mod}\mkern 6mu#1)}
\makeatother

\makeatletter
\@namedef{subjclassname@2020}{%
	\textup{2020} Mathematics Subject Classification}
\makeatother

\begin{document}
	
	\title{Symmetry results for multiple $t$-values}
	\date{April 29, 2022}
	
	\author{Steven Charlton}
	\address{Fachbereich Mathematik (AZ), Universit\"at Hamburg, Bundesstra\textup{\ss}e 55, 20146 Hamburg, Germany}
	\email{steven.charlton@uni-hamburg.de}
	
	\author{Michael E. Hoffman}
	\address{Department of Mathematics, U.S. Naval Academy, Annapolis, MD 21402, USA}
	\email{meh@usna.edu}
	
	\keywords{Multiple zeta values, multiple $t$ values, alternating MZV's, alternating M$t$V's, parity theorem, Hopf algebra, stuffle product, motivic MZV's, special values, generating series}
	\subjclass[2020]{Primary 11M32, 11G55; Secondary 33B30, 33B15}

	\begin{abstract}
		For a composition $I$ whose first part exceeds 1, we can define the multiple $t$-value $t(I)$ as the sum of all the terms in the series for the multiple zeta value $\zeta(I)$ whose denominators are odd.  In this paper we show that if $I$ is composition of $n\ge 3$, then $t(I)=(-1)^{n-1}t(\bar I)$ mod products, where $\bar I$ is the reverse of $I$, and
		both sides are suitably regularized when $I$ ends in 1.  This result is not true for multiple zeta values, though there is an argument-reversal result that does hold for them (and for multiple $t$-values as well).  We actually prove a more general version of this result, and then use it to establish explicit formulas for several classes of
		multiple $t$-values and interpolated multiple $t$-values.
	\end{abstract}
	
	\maketitle
	
	\section{Introduction}
	We define multiple zeta values $\zeta(n_1,\dots,n_\ell)$
	and multiple $t$-values $t(n_1,\dots,n_\ell)$ by
	\[
	\zeta(n_1,n_2\dots,n_\ell)=
	\sum_{1\le k_1<k_2<\dots<k_\ell}\frac1{k_1^{n_1} k_2^{n_2}\cdots k_\ell^{n_\ell}}
	\]
	and
	\[
	t(n_1,n_2\dots,n_\ell)=
	\sum_{1\le k_1<k_2<\dots<k_\ell}\frac1{(2k_1-1)^{n_1} 
		(2k_2-1)^{n_2}\cdots (2k_\ell-1)^{n_\ell}}
	\]
	respectively.  The series converge provided $n_\ell>1$.
	We note that the set of multiple zeta values and the
	set of multiple $t$-values are both algebras under the
	``stuffle" product, e.g.,
	\[
	t(2)t(3,2)=t(2,3,2)+t(5,2)+2t(3,2,2)+t(3,4).
	\]
	More formally, both the multiple zeta values and
	the multiple $t$-values are homomorphic images of
	a subalgebra of the quasi-symmetric functions
	$\QS$.  Now $\QS$ can be regarded as the vector
	space on words in noncommuting variables 
	$z_1,z_2,\dots$, with a commutative product $*$
	defined inductively by 
	\[
	z_iu*z_jv=z_i(u*z_jv)+z_j(z_iu*v)+z_{i+j}(u*v) .
	\]
	Then $\QS$ has a subalgebra $\QS^0$ generated by
	1 and all words ending in $z_i$, $i>1$; there
	are homomorphisms $\zt,t:\QS^0\to\mathbb{R}$ given by
	\begin{align*}
	\zt(z_{i_1}\cdots z_{i_k})&=\zt(i_1,\dots,i_k)\\
	t(z_{i_1}\cdots z_{i_k})&=t(i_1,\dots,i_k) .
	\end{align*}
	By the result of Malvenuto and Reutenauer
	\cite{malvenuto-reutenauer}, $\QS$ is a polynomial
	algebra on Lyndon words in the $z_i$, and the only 
	Lyndon word ending in $z_1$ is $z_1$ itself.  Thus
	$\QS=\QS^0[z_1]$, and we can extend the homomorphisms
	above to homomorphisms $\QS\to\mathbb{R}[T]$ by sending $z_1$ to $T$.  We denote these by $\reg^\ast_{T}\zt$ and $\reg^\ast_{T}t$ respectively.
	It is convenient to set $T=0$ in the first case
	and $T=\log 2$ in the second.  We will typically work with the stuffle regularization, and so for notational simplicity we shall suppress the asterisk from the notation and write \( \reg_{T} = \reg_{T}^\ast \), unless we need to clarify which type of regularization is in use.\medskip
	
	Our principal result is as follows.
	\begin{Thm}[Symmetry Theorem]
		\label{thm:scon}
		If $I$ is a composition of $n\ge 3$, then
		\[
		\reg^\ast_{T=\log2}t(I)=(-1)^{n-1}
		\reg^\ast_{T=\log2}t(\bar I) \pmod{\text{\rm products}},
		\]
		where $\bar I$ is the reverse of $I$.
	\end{Thm}
	We will in fact prove a stronger version of this result in \autoref{thm:symgsfull}, which holds for M$t$V's at any roots of unity, and provides the neglected product terms via a generating series identity involving M$t$V's and MZV's.  (We will use a result of Murakami \cite{murakami} to replace MZV's with M$t$V's to establish our claim in the case of classical M$t$V's.) \medskip
	
	We can define interpolated multiple $t$-values $t^r$ in the same way as S. Yamamoto \cite{yamamoto} defined interpolated multiple zeta values, e.g.,
	\[
	t^r(2,1,3)=t(2,1,3)+rt(3,3)+rt(2,4)+r^2t(6).
	\]
	Then $t^0=t$, and we write $t^{\star}$ for
	$t^1$.  The preceding result has the following corollary.
	\begin{Cor}\label{cor:symtr} For any composition $I$ of $n\ge 3$, 
		\[
		\reg^\ast_{T=\log2}t^r(I)=(-1)^{n-1}\reg^\ast_{T=\log2}t^r(\bar I) \pmod{\text{\rm products}}.
		\]
	\end{Cor}
	\begin{proof}
		Induct on the length of $I$, using the definition of the interpolated multiple $t$-value; to start the induction, note that $t(n)$ is decomposable for $n\ge 4$ even. 
	\end{proof}
	Using the Hopf algebra structure on the
	interpolated multiple $t$-values (see
	\cite{hoffman20}, or alternatively \cite[Lemma 4.2.2]{glanoisTh},\cite[Lemma 3.3]{glanoisUnramified}, at least for \( r \in \{ 0, \half, 1 \} \)), we can prove the following
	result.  
	\begin{Thm}[Stuffle antipode]\label{thm:dep}
		For any composition $I$,
		\[
		\reg^\ast_{T=\log2}t^r(I)=
		(-1)^{\ell(I)-1}\reg^\ast_{T=\log2}t^{1-r}(\bar I)
		\pmod{\text{\rm products}},
		\]
		where $\ell(I)$ is the length (number of parts) of $I$.
	\end{Thm}
	\begin{proof}
		Let $(\H^1,\r,\Delta)$ be the Hopf algebra with underlying vector space $\H^1=\Q\<z_1,z_2,\dots\>$, the interpolated product $\r$, and the deconcatenation coproduct $\Delta$.
		As shown in \cite{hoffman20}, this Hopf algebra has antipode
		$S=\Sigma^{1-2r}TR$, where $R$ reverses words, $T(w)=(-1)^{\ell(w)}w$, and 
		\[
		\Sigma^p(z_I)=\sum_{I_1\sqcup\cdots\sqcup I_k=I}p^kz_{|I_1|}\cdots z_{|I_k|},
		\]
		where for a composition $I=(i_1,\dots,i_n)$, $z_I$ denotes the word $z_{i_1}\cdots z_{i_n}$,  $|I|\coloneqq i_1+\cdots+i_n$ and $\sqcup$ is juxtaposition of compositions.   By induction one can show that $S$ has the alternative formula
		\begin{equation}
		\label{ant}
		S(z_I)=\sum_{I_1\sqcup\cdots\sqcup I_k=I}(-1)^kz_{I_1}\r\cdots\r z_{I_k} .
		\end{equation}
		Since $(\H^1,\r,\Delta)$ is commutative, $S$ is an algebra homomorphism and an involution.  Apply $S$ to both sides of
		\autoref{ant} to get
		\begin{equation}\label{eqn:stuffleantipode}
		z_I=\sum_{I_1\sqcup\cdots\sqcup I_k=I}(-1)^{\ell(I)-k}\Sigma^{1-2r}R(z_{I_1})\r\cdots\r\Sigma^{1-2r}R(z_{I_k}),
		\end{equation}
		and then apply $\reg^\ast_{T=\log2}t^r=\reg^\ast_{T=\log2}t\circ\Sigma^r$ to both sides of the latter equation to get
		\begin{equation*}
		\reg^\ast_{T=\log2}t^r(I)=\sum_{I_1\sqcup\cdots\sqcup I_k=I}(-1)^{\ell(I)-k}
		\reg^\ast_{T=\log2}t^{1-r}(\bar{I_1})\cdots \reg^\ast_{T=\log2}t^{1-r}(\bar{I_k}),
		\end{equation*}
		from which the conclusion follows.
	\end{proof}
	
	\autoref{thm:scon} and \autoref{thm:dep} imply the following.
	\begin{Cor}
		For any composition $I$, with sum $|I|\ge 3$, 
		\[
		\reg^\ast_{T=\log2}t^r(I)=
		(-1)^{|I|-1}\reg^\ast_{T=\log2}t^r(\bar I)=
		(-1)^{|I|-\ell(I)}\reg^\ast_{T=\log2}t^{1-r}(I)\ \pmod{\mathrm{products}}\,.
		\]
	\end{Cor}
	\noindent In particular, if $|I|\ge 3$ then
	\[
	\reg^\ast_{T=\log2}t(I)=
	(-1)^{|I|-\ell(I)}\reg^\ast_{T=\log2}t^{\star}(I)\ \pmod{\text{products}}
	\]
	and 
	\[
	\reg^\ast_{T=\log2}t^{\half}(I)=0 \pmod{\text{products}}\,,
	\]
	if $|I|$ and $\ell(I)$ have opposite parity.  A parity theorem for MZV's, which reduces \( \zeta(I) \) to lower depth \( \pmod*{\text{products}} \), if $|I|$ and $\ell(I)$ have opposite parity, is well-known.  Extensions to arbitrary roots of unity and to multiple polylogarithm functions are established in \cite{panzer}.  (This is also related to the result from \cite[\S2.6]{goncharov}, holding on the torus, which we adapt and utilize for our result.)
	
	\begin{Rem} Since both multiple $t$-values and multiple zeta values are images of homomorphisms from $(\H^1,\r,\Delta)$ to the reals, \autoref{thm:dep} holds for multiple zeta values.  (In fact, \autoref{thm:dep} can also be deduced from \cite[Theorem 1.2]{bachmann}.)  But \autoref{thm:scon} fails for multiple zeta values.  For example,
		\[
		\zeta(2,3)+\zeta^{\star}(3,2)=\zt(2)\zt(3)
		\]
		but
		\[
		\zeta(2,3)-\zeta(3,2)=-10\zeta(5)+5\zeta(2)\zeta(3) .
		\]
	\end{Rem}
	
	\medskip
	\paragraph{\bf Outline.} The remainder of this paper is organized as follows.
	In Section 2 we give a proof of \autoref{thm:symgsfull},
	which as indicated above implies \autoref{thm:scon}.  For this we first establish an identity on MZV's and M$t$V's truncated to order \( M \) (\autoref{prop:truncgs}), then using some analytic results from \autoref{sec:appendix} we pass to the limit \( M \to \infty \) in order to obtain an identity among MZV's and M$t$V's at roots of unity (\autoref{thm:symgsrestricted}), with `non-degenerate' angles \( \phi_1,\ldots,\phi_m \), such that \( \phi_1 , \phi_m, \phi_1 + \cdots + \phi_m \neq 0 \).  Using the asymptotic expansion recalled in \autoref{sec:asymp} and \autoref{sec:asympolylog}, we extend this identity to all angles \( \phi_1,\ldots,\phi_m \) by considering how \( \phi_i \to 0 \), to establish the Symmetry Theorem in generating series form.  After application of some motivic results expressing MZV's via M$t$V's \cite{murakami} we obtain the Symmetry Theorem as stated above (\autoref{thm:scon}).
	
	In \autoref{sec:applications} we give three applications of
	these results.  First in \autoref{sec:t3223andstar}, to the computation of $t(3,\{2\}^n,3)$, where \( \{a\}^n \) denotes the string \( a, \ldots, a \) with \( n \) repetitions.  This requires Zagier's evaluation of \( \zeta(\{2\}^a,3,\{2\}^b) \) \cite{zagier2232}, Murakami's evaluation of \( t(\{2\}^a,3,\{2\}^b) \) \cite{murakami}, and the Ohno-Zagier Theorem \cite{ohno-zagier} to evaluate certain MZV combinations.  Second in \autoref{sec:t1221andt1mm1}, to formulas for $t(1,\{\overline{1}\}^n,1)$ and $t(1,\{2\}^n, 1)$.  This requires the evaluations of \( \reg_{T=\log2} t(\{2\}^a,1,\{2\}^b) \) and \( \reg_{T=\log2} t(\{\overline{1}\}^a,1,\{\overline{1}\}^b) \) given in \cite{charltont2212,charltontmm1m}, and an apparently new evaluation (\autoref{prop:zetamm1}) for \( \reg_{T=0} \zeta(\{\overline{1}\}^m,1) \).  Finally in \autoref{sec:th111ev}, to generating series formulas for the interpolated multiple $t$-values of the form $t^\half(\{1\}^n,2\ell+2)$ and $t^\half(2\ell+2, \{1\}^{2n}, 2\ell+2) $.  This requires solving a pair of simultaneous generating series relations, one obtained from the stuffle antipode (\autoref{thm:dep}) and one obtained from the Symmetry Theorem (\autoref{thm:symgsfull}) by a certain infinite series of differentials.  We also explicitly treat the case \( t^\half(2,\{1\}^n,2) \), for odd and even \( n \), in particular giving a conjectural evaluation in odd weight (\autoref{conj:thalf21od2}).
	
	\medskip
	\paragraph{\bf Acknowledgements.}  The authors thank the directors of the Max-Planck-Institut f\"ur Mathematik in Bonn for their support during the unusual circumstances of 2020, when this project was conceived.  
	The first author is grateful to Oleksiy Klurman for some helpful discussions on the analytic technicalities needed 
	in \autoref{sec:appendix}.
	He was supported by DFG Eigene Stelle grant CH 2561/1-1, for Projektnummer 442093436.  
	The second author received partial support from the Naval Academy Research Council.
	
	\section{Regularized version of the Symmetry Theorem}

	\subsection{Polylogarithms and regularization}
	
	We recall the definition of the multiple polylogarithm functions (MPL's) in several variables, whose asymptotic expansion will be important in the sequel.
	
	\begin{Def}[Multiple polylogarithm]
		For \( \abs{x_i} < 1 \), \( i = 1, \ldots, d \), the multiple polylogarithm is defined by
		\[
		\Li_{n_1,\ldots,n_d}(x_1,\ldots,x_d) = \sum_{1 \leq k_1 < k_2 < \cdots < k_d} \frac{x_1^{k_1} \cdots x_d^{k_d}}{k_1^{n_1} \cdots k_d^{n_d}} \,.
		\]
	\end{Def}
	Corollary 2.3.10 in \cite{zhao} shows that the series defining \( \Li_{n_1,\ldots,n_d}(x_1,\ldots,x_d) \) in fact converges (but perhaps only conditionally), for \( \abs{x_i} \leq 1 \), \( i = 1,\ldots,d \) if and only if \( (x_d,n_d) \neq (1,1) \). \medskip
	
	The behavior of \( \Li_{n_1,\ldots,n_{d-1},1}(x_1,\ldots,x_{d-1}, x_d) \) as \( x_d \to 1 \) (which is actually dependant on how \( x_d \to 1 \), and on whether \( n_{d-1} = 1 \) and then on how \( x_{d-1} \to 1 \)) can be used to define various notions of regularization.  This allows us to make sense of identities and results even in cases where not all of the objects of interest actually converge; in order to apply the symmetry result to the convergent value \( t(1,1,2) \), one necessarily needs to make sense of the divergent value \( t(2,1,1) \) somehow.  One can then show that
	\begin{align}\label{eqn:t112eg}
	t(1,1,2) + t(2,1,1) = -t(2) t(1,1) + t(1) t(1,2) + \frac{1}{2} t(2) \zeta(1,1) \,,
	\end{align}
	where \( t \) is stuffle-regularized with \( t(1) = \log2 \), and \( \zeta \) is stuffle-regularized with \( \zeta(1) = 0 \). 
	
	\medskip
	To this end, we more formally introduce the stuffle regularization of multiple polylogarithms, and related objects, and indicate how one computes it.  Polylogarithms can be multiplied with the stuffle product, generalizing the formula which holds for multiple zeta values and multiple \( t \) values.  For example
	\[
	\Li_{2,3}(a,b) \Li_1(c) = \begin{aligned}[t]
	& \Li_{1,2,3}(c,a,b) + \Li_{2,1,3}(a,c,b) + \Li_{2,3,1}(a,b,c) \\
	& + \Li_{2,4}(a,bc) + \Li_{3,3}(ac,b)
	\end{aligned}
	\]
	Ignoring convergence issues for the moment and viewing \( \Li_1(1) \) as a formal object (this is made rigorous by considering truncated versions, and allowing the summation bound to tend to infinity), any multiple polylogarithm of the form
	\[
	\Li_{n_1,\ldots,n_\ell,\{1\}^k}(x_1,\ldots,x_\ell,\{1\}^k) \,,
	\]
	with \( (n_\ell,x_\ell) \neq (1,1) \), where \( \{a\}^n \) denotes the string \( \overbrace{a, \ldots, a}^{\text{$n$ times}} \) with \( n \) repetitions, can be written as a polynomial in \( \Li_{1}(1) \) with convergent polylogarithm coefficients.  This is accomplished by considering
	\begin{align*}
	\Li_{n_1,\ldots,n_\ell,\{1\}^k}(x_1,\ldots,x_\ell,\{1\}^k) = {} 
	& \frac{1}{k} \Li_{n_1,\ldots,n_\ell,\{1\}^{k-1}}(x_1,\ldots,x_\ell,\{1\}^{k-1}) \Li_1(1) \\
	& + (\text{terms with \( < k \) trailing \( (n_i,x_i) = (1,1) \) entries})
	\end{align*}
	and recursively applying the process to all of the terms on the right-hand side.
	
	We shall write 
	\[
	\reg^\ast_T \Li_{n_1,\ldots,n_\ell,\{1\}^k}(x_1,\ldots,x_\ell,\{1\}^k)
	\]
	to denote this stuffle-regularization polynomial, where \( \Li_1(1) \) is replaced by the indeterminate \( T \).  This is called the stuffle regularization with parameter \( T \), and also applies to MZV's and M$t$V's. As in the introduction, we shall suppress the asterisk, and write \( \reg_T= \reg_T^\ast \), unless we need to clarify which type of regularization is in use.\medskip
	
	With this formalized, the identity in \autoref{eqn:t112eg} is then written as
	\begin{align*}
	& t(1,1,2) + \reg^\ast_{T=\log2} t(2,1,1) = \\
	& -t(2) \reg^\ast_{T=\log2} t(1,1) + \reg_{T=\log2}^\ast t(1) t(1,2) + \frac{1}{2} t(2) \reg^\ast_{T=0} \zeta(1,1)
	\end{align*}
	After regularization, this states
	\begin{align*}
	& 2 t(1,1,2) + t(1,3) - 2 t(1,2) \log2 - t(3)\log2 + t(2) \log^2 2 + \frac{1}{2} t(4) - \frac{1}{2} t(2)^2 + \frac{1}{4} t(2) \zeta(2) = 0 \,,
	\end{align*}
	which can be directly checked using the tables in  \cite[Appendix A]{hoffman19}.

	\subsection{Asymptotic expansions of polylogarithms}\label{sec:asympolylog}
	
	We recall the setup of the asymptotic expansion of polylogarithms introduced in \cite[\S{}2.10]{goncharov}.  Lemma 2.18 \cite{goncharov} establishes that, for \( \abs{x_i} \leq 1 \), with \( (n_\ell, x_\ell) \neq (1,1) \),  the power series 
	\[
	f(\eps) = \Li_{n_1,\ldots,n_\ell,\{1\}^k}(x_1,\ldots,x_\ell, \{1-\eps\}^k)
	\]
	has an asymptotic expansion (as \( \eps \to 0^{+} \)), which is a polynomial in \( \log(\eps) \), whose coefficients are explicitly computable \( \Q \)-linear combinations of lower depth MPL's.  Moreover the polynomial has degree \( k \), and weight \( w = n_1 + \cdots + n_\ell + k \) counting \( \log(\eps) \) as weight 1.  It is instructive to review the proof of this claim, by way of an example, as we will utilize a similar setup with \( x_d \to 1 \) through roots of unity in order to establish the regularized version of the Symmetry Theorem. \medskip
	
	More formally, we refer to \cite[\S{}3.7.4]{burgos-fresan}  for the preliminaries about logarithmic asymptotic expansions of continuous functions.  In particular, we make the following definition.
	
	\begin{Def}[{\cite[\S{}3.7.4]{burgos-fresan}}]
		Let \( f \colon (0, \tau) \to \mathbb{C} \) be a continuous function, \( 0 < \tau \leq 1 \).  We say that \( f \) admits a logarithmic asymptotic expansion of degree \( r \) if it can be written
		\[
		f(\eps) = f_0(\eps) + \sum_{k=0}^r a_k \log^k\eps \,,
		\]
		with \( \abs{f_0(\eps)} = O(\eps^{1-\delta}) \) for some \( \delta < 1 \), \( a_k \in \mathbb{C} \).
	\end{Def}
	
	We shall then write
	\[
	\reg^\asym f(\eps) \coloneqq \sum_{k=0}^r a_k \log^k \eps
	\]
	to denote this logarithmic asymptotic expansion. \medskip
	
	If it exists, this logarithmic asymptotic expansion is unique (Lemma 3.237 \cite{burgos-fresan}), as one can recover
	\[
	a_r = \lim_{t \to 0} \frac{f(t)}{\log^r t} \,.
	\]
	Upon knowing \( a_{s+1},\ldots,a_r \), one can then find
	\[
	a_s = \lim_{t\to0} \frac{f(t) - \sum_{k=s+1}^r a_k \log^k t}{\log^s t} \,,
	\]
	to determine the entire logarithmic asymptotic expansion. \medskip
	
	In order to compute the asymptotic expansion of 
	\[
	f(\eps) = \Li_{n_1,\ldots,n_d,\{1\}^k}(x_1,\ldots,x_d,\{1-\eps\}^k) \,,
	\]
	with \( (n_d, x_d) \neq (1,1) \), we apply the stuffle product to see
	\begin{align*}
	& \Li_{n_1,\ldots,n_d,\{1\}^{k-1}}(x_1,\ldots,x_d, \{1-\eps\}^{k-1}) \cdot \Li_{1}(1-\eps)  \\ & \quad {} = \begin{aligned}[t] & k \Li_{n_1,\ldots,n_d,\{1\}^k}(x_1,\ldots,x_d,\{1-\eps\}^k) \\
	& + \text{(terms with $<k$ trailing $(n_i,x_i) = (1-\eps,1)$ entries)} \,.
	\end{aligned}
	\end{align*}
	By recursion, we obtain that \( f(\eps) \) is a sum of products of terms of the form
	\begin{align*}
	\Li_{1}(1-\eps) 
	\quad\text{ and }\quad \Li_{m_1,\ldots,m_e}(y_1 (1-\eps)^{p_1}, \ldots, y_e (1-\eps)^{p_e}) \,, \quad (y_e, m_e) \neq (1, 1) 
	\,.
	\end{align*}
	Since
	\[
	\lim_{\eps\to0} \Li_{m_1,\ldots,m_e}(y_1 (1-\eps)^{p_1}, \ldots, y_e (1-\eps)^{p_e}) = \Li_{m_1,\ldots,m_e}(y_1, \ldots, y_e) 
	\]
	exists (as a convergent polylog), we claim that this value at \( \eps = 0 \) is the entire asymptotic expansion.  Indeed it amounts to the differentiability of the multiple polylogarithm at this point, as we can see that for any \( 0 < \delta < 1\),
	\begin{align*}
	& \lim_{\eps\to0} \frac{\Li_{m_1,\ldots,m_e}(y_1 (1-\eps)^{p_1}, \ldots, y_e (1-\eps)^{p_e}) - \Li_{m_1,\ldots,m_e}(y_1, \ldots, y_e) }{\eps^{1 - \delta}} \\ 
	& =  
	\lim_{\eps\to0} \eps^\delta \frac{\Li_{m_1,\ldots,m_e}(y_1 (1-\eps)^{p_1}, \ldots, y_e (1-\eps)^{p_e}) - \Li_{m_1,\ldots,m_e}(y_1, \ldots, y_e) }{\eps} \\
	& = 0 \cdot \frac{\dd }{\dd \eps} \bigg\rvert_{\eps=0} \Li_{m_1,\ldots,m_e}(y_1 (1-\eps)^{p_1}, \ldots, y_e (1-\eps)^{p_e}) \\
	&= 0 .
	\end{align*}
	In particular, we have that 
	\begin{align*}
	& \Li_{m_1,\ldots,m_e}(y_1 (1-\eps)^{p_1}, \ldots, y_e (1-\eps)^{p_e}) \\
	& - \reg^\asym \Li_{m_1,\ldots,m_e}(y_1 (1-\eps)^{p_1}, \ldots, y_e (1-\eps)^{p_e}) = O(\eps^{1-\delta}) \,, \text{any \( 1 - \delta < 1 \)} .
	\end{align*}
	Finally, as
	\[
	\Li_1(1-\eps) = -\log\eps
	\]
	exactly, we therefore have that the asymptotic expansion of \( f(\eps) \) is given just as the sum of products of the individual asymptotic expansions.
	
	\begin{Eg}
		For example with \( x \neq 1 \), we compute using the stuffle product that
		\begin{align*}
		& \Li_{1,1,1}(x,1-\eps,1-\eps) = \\
		& 
		\Li_{1,1,1}(1-\eps ,1-\eps ,x)
		+\frac{1}{2} \Li_3(x (1-\eps )^2)
		+\frac{1}{2} \Li_{2,1}((1-\eps )^2,x)
		+\Li_{1,2}(1-\eps ,x (1-\eps )) \\
		& 
		-\frac{1}{2} \Li_{1,2}(x,(1-\eps )^2)
		-\Li_1(1-\eps ) \Li_{1,1}(1-\eps ,x)
		-\Li_1(1-\eps ) \Li_2(x (1-\eps )) \\
		& +\frac{1}{2} \Li_1(x) \Li_1(1-\eps )^2 
		\end{align*}
		The asymptotic expansions of each term are now readily found by the above analysis, and we have
		\begin{align*}
		& \reg^\asym \Li_{1,1,1}(x,1-\eps,1-\eps) = \\
		& \Li_{1,1,1}(1,1,x)
		+ \frac{1}{2}\Li_3(x) 
		+ \Li_{1,2}(1,x)
		- \frac{1}{2} \Li_{1,2}(x,1)
		+ \frac{1}{2} \Li_{2,1}(1,x) \\
		& + \Li_{1,1}(1,x) \log\eps
		+ \Li_2(x) \log\eps
		+\frac{1}{2}  \Li_1(x) \log ^2\eps \,.
		\end{align*}
	\end{Eg}
	
	Of significant interest and use will be the constant term of such an asymptotic expansion; write \( [\log^0\eps] \reg^\asym f(\eps) \eqqcolon \reg^\asym_0 f(\eps) \) to denote this constant term.  In particular, we will use this to make sense of identities in a limiting case, in terms of regularized values of multiple zeta values and multiple $t$-values.
	
	By extending the argument above, we notice that the asymptotic expansion of
	\[
	\Li_{n_1,\ldots,n_d,\{1\}^k}(x_1,\ldots,x_d, \{\alpha(\eps)\}^k) \,,
	\]
	with \( (n_d, x_d) \neq (1,1) \) and \( \alpha(\eps) \) with \( \alpha(\eps) \xrightarrow{\eps\to0} 1 \) (for sufficiently nice \( \alpha \)) depends only on the asymptotic expansion of \( \Li_1(\alpha(\eps)) \).  The constant term of
	\[
	\reg^\asym \Li_{n_1,\ldots,n_d,\{1\}^k}(x_1,\ldots,x_d, \{\alpha(\eps)\}^k)
	\]
	therefore only depends on the constant term of the asymptotic expansion of \( \Li_1(\alpha(\eps)) \).  The stuffle product structure used to obtain this asymptotic expansion means that the constant term is obtained by the stuffle regularization, with regularization parameter given by 
	\[ T_0 = [\log^0\eps] \reg^\asym \Li_1(\alpha(\eps)) \,.
	\]
	We have established the following lemma.
	
	\begin{Lem}(Asymptotic expansion and stuffle regularization)\label{lem:asympviastuff}
		For \( \alpha(\eps) \to 1 \) as \( \eps \to 0^+ \), the constant term \( [\log^0\eps] A \) in the asymptotic expansion
		\[
		A = \reg^\asym \Li_{n_1,\ldots,n_d,\{1\}^k}(x_1,\ldots,x_d, \{\alpha(\eps)\}^k)
		\]
		is given by 
		\[
		[\log^0\eps] A = \reg^\ast_{T=T_0} \Li_{n_1,\ldots,n_d,\{1\}^k}(x_1,\ldots,x_d, \{1\}^k) \,.
		\]
		Here \( \reg^\ast_{T=T_0} \) denotes the stuffle regularized version of \( \Li_{n_1,\ldots,n_d,\{1\}^k}(x_1,\ldots,x_d, \{1\}^k) \), with regularization parameter
		\[
		T_0 \coloneqq \reg^\ast_{T=T_0} \Li_1(1) 
		\]
		given by
		\[
		T_0 = [\log^0\eps] \reg^\asym \Li_1(\alpha(\eps)) \,.
		\]
	\end{Lem}
	
	\begin{Rem}
		Note that this claim depends very strongly on fact that the indices 1 come with the \emph{same} argument \( \alpha(\eps) \).  The asymptotic expansion of
		\[
		\Li_{n_1, \ldots, n_d, \{1\}^k}(x_1, \ldots, x_d, \{1\}^{k-1}, 1-\eps)
		\]
		is instead connected to the shuffle regularization.  For this we refer to Proposition 2.20 in \cite{goncharov} in particular, and Sections 2.9--2.10 in \cite{goncharov} for the broader context.
	\end{Rem}
	
	As an example of the effect that the change of argument \( \alpha(\eps) \) makes to the asymptotic expansion -- something with which we must contend later -- let us consider the distribution relations (cf. Lemma 2.21 \cite{goncharov}).  
	
	\begin{Eg}[Regularized distribution relations]
		The following holds for all \( \eps > 0 \) as every term is convergent
		\[
		\sum_{s_1,s_2,s_3 \in \{ \pm 1 \}} \Li_{n_1,1,1}(s_1, s_2(1-\eps), s_3(1-\eps)) = \frac{1}{2^{n_1-1}} \Li_{n_1,1,1}(1,(1-\eps)^2, (1-\eps)^2) \,.
		\]
		
		On the left-hand side, index 1 comes with argument \( 1-\eps \) in the case of non-convergent MPL's.  We then have \( \Li_1(1-\eps) = -\log\eps \).  The constant term of (the asymptotic expansion of) this is 0.  On the right-hand side however, index 1 comes with argument \( (1-\eps)^2 \), and instead we have
		\[
		\Li_1((1-\eps)^2) = -\log\eps - \log(2-\eps) \,.
		\]
		So the constant term in the asymptotic expansion of 
		\[
		\reg^\asym \Li_1((1-\eps)^2) = -\log2 -\log\eps
		\]
		is \( -\log2 \).
		
		The following regularized version of the distribution relation holds, with different regularization parameters on the left-hand and right-hand sides:
		\[
		\sum_{s_1,s_2,s_3 \in \{ \pm 1 \}} \reg_{T=0}^\ast \Li_{n_1,1,1}(s_1, s_2, s_3) = \frac{1}{2^{n_1-1}} \reg_{T=-\log2}^\ast \Li_{n_1,1,1}(1,1,1) \,.
		\]
	\end{Eg}
	
	\subsection{\texorpdfstring{Asymptotic expansions for zeta and $t$-values}{Asymptotic expansions for zeta and t values}}\label{sec:asymp}
	
	We now focus on multiple zeta values and multiple \( t \)-values for the remainder of the paper.
	
	\begin{Def}[MZV's and M$t$V's]
		The multiple zeta values, respectively multiple \( t \) values, with signs \( \eps_1,\ldots,\eps_m \in \{ z \in \C : \abs{z} = 1 \} \) are defined by
		\begin{align*}
		\zeta\sgnarg{\eps_1,\ldots,\eps_m}{n_1,\ldots,n_m} &\coloneqq \sum_{0 < k_1 < \cdots < k_m} \frac{\eps_1^{k_1} \cdots \eps_m^{k_m}}{k_1^{n_1} \cdots k_m^{n_m}} \,, \\
		t\sgnarg{\eps_1,\ldots,\eps_m}{n_1,\ldots,n_m} &\coloneqq \sum_{0 < k_1 < \cdots < k_m} \frac{\eps_1^{k_1} \cdots \eps_m^{k_m}}{(2k_1-1)^{n_1} \cdots (2k_m-1)^{n_m}} \,.
		\end{align*} 
	\end{Def}
	
	Essentially \( \zeta \) is given by just restricting  \( \Li \) to arguments on the unit circle, for notational emphasis, although \( t \) is a genuinely distinct object.  We will apply a similar prescription for the asymptotic expansion to the following setup involving these MZV's and M$t$V's,
	\[
	t\sgnarg{\exp(2\pi\ii n \phi_1), \ldots, \exp(2\pi\ii\phi_m)}{n_1,\ldots,n_m} \quad\quad \zeta\sgnarg{\exp(2\pi\ii n \phi_1), \ldots, \exp(2\pi\ii\phi_m)}{n_1,\ldots,n_m} \,,
	\]
	where \( \phi_i = \phi_i(\eps) \).  In particular we shall need compute the asymptotic expansions as \( \eps \to 0^+ \) (the direction from above is important!) of
	\[
	f\sgnarg{\exp(2 \pi \ii n \eps)}{1} \,,
	\]
	for \(f = \zeta, t \), as a polynomial in \( \log(1 - e^{2 \pi \ii \eps}) \).  Because we will use only this version of the asymptotic series for the remainder of the paper, there will be no confusion between this and the version in the previous \autoref{sec:asympolylog}; we will therefore also denote this by \( \reg^\asym \), with only a slight abuse of notation.  The same argument as in \autoref{lem:asympviastuff} allows us to relate the constant term of the asymptotic expansion (when each index 1 comes with the same sign \( \exp(2\pi\ii \alpha(\eps)) \)) and the stuffle regularizations of \( \zeta \) and \( t \). Likewise, the asymptotic expansion of an MZV or M$t$V which converges as \( \eps \to 0^- \) in this prescription is just the MZV or M$t$V evaluated at \( \eps = 0 \).  So we just need to deal with the basic depth 1 cases.
	
	\begin{Lem}[Asymptotic expansion of \( t \)]\label{lem:tasymp}
		The asymptotic expansion of \( t\sgnargsm{\exp(2 \pi \ii n \eps)}{1} \), for \( n \neq 0 \in \mathbb{Z} \), is given by
		\[
		\reg^\asym t\sgnarg{\exp(2 \pi \ii n \eps)}{1} = \log\Big( \frac{2}{\sqrt{n}} \Big) - \frac{1}{2} \log(1 - \exp(2 \pi \ii \eps)) \,.
		\]
		
		\begin{proof}
			One can evaluate the following series, for \( \eps \in (0,\frac{1}{n}) \), as
			\begin{align*}
			t\sgnarg{\exp(2 \pi \ii n \eps)}{1} = \sum_{k=1}^\infty \frac{\exp(2\pi \ii n \eps k)}{2k-1} &= \exp(\ii \pi n \eps) \arctanh(\exp(\ii \pi n \eps)) \\
			&= \exp(\ii \pi n \eps) \Big( \frac{1}{2} \log\cot\Big(\frac{\pi n\eps}{2}\Big) + \frac{\ii \pi}{4} \Big) \,.
			\end{align*}
			This is extended by periodicity to \( \R \setminus \frac{1}{n} \Z \).  Note there is a jump discontinuity in the imaginary part at each \( \frac{k}{n} \).  More precisely
			\begin{align*}
			& \lim_{\eps \to 0^+} \Im \exp(\ii \pi n \eps) \Big( \frac{1}{2} \log\cot\Big(\frac{\pi n\eps}{2}\Big) + \frac{\ii \pi}{4} \Big) = \frac{\pi}{4} \\
			& \lim_{\eps \to 0^-} \Im \exp(\ii \pi n \eps) \Big( \frac{1}{2} \log\cot\Big(\frac{\pi n\eps}{2}\Big) + \frac{\ii \pi}{4} \Big) = -\frac{\pi}{4} \,,
			\end{align*}
			where \( \lim_{\eps\to0^-} = \lim_{\eps\to\frac{1}{n}^-} \) returns us to the original range of definition of the series.
			
			We also have
			\[
			\log(1 - \exp(2 \pi\ii \eps)) = \frac{1}{2} \log(2 \cdot (1 - \cos(2\pi \eps))) + \Big( {-}\frac{\pi}{2} + \pi \eps \Big) \ii
			\]
			
			The following limit computations using the above establish the asymptotic series is a linear polynomial, and then compute for us the constant term, giving the claim.  First, assuming \( n > 0 \)
			\begin{align*}
			\lim_{\eps\to0^+} \frac{t\sgnarg{\exp(\pm2 \pi \ii n \eps)}{1}}{\log(1 - \exp(2 \pi \ii \eps)))} &= -\frac{1}{2} \,,
			\end{align*}
			Then
			\begin{align*}
			\lim_{\eps\to0^+} t\sgnarg{\exp(2 \pi \ii n \eps)}{1} + \frac{1}{2} \log(1 - \exp(2 \pi \ii \eps)) &= \log2 - \frac{1}{2} \log n \, \\
			\lim_{\eps\to0^+} t\sgnarg{\exp(-2 \pi \ii n \eps)}{1} + \frac{1}{2} \log(1 - \exp(2 \pi \ii \eps)) &= \log2 - \frac{1}{2} \log n - \frac{\ii\pi}{2} \,.
			\end{align*}
			The two cases \( \pm n \) combine to the expression involving the square root by taking \( \frac{1}{2} \log(-1) = \frac{\ii\pi}{2} \).
		\end{proof}
	\end{Lem}
	
	\begin{Lem}
		[Asymptotic expansion of \( \zeta \)]\label{lem:zetasymp}
		The asymptotic expansion of \(
		\zeta\!\sgnargsm{\!\exp(2 \pi \ii n \eps)\!}{1} 
		\), for \( n \neq 0 \in \mathbb{Z} \), is given by
		\[
		\reg^\asym \zeta\sgnarg{\exp(2 \pi \ii n \eps)}{1}  = - \log n - \log(1 - \exp(2 \pi \ii \eps)) \,.
		\]
		
		\begin{proof}
			One can evaluate the following series, for \( \eps \in (0,\frac{1}{n}) \), as
			\begin{align*}
			\zeta\sgnarg{\exp(2 \pi \ii n \eps)}{1}  = \sum_{k=1}^\infty \frac{\exp(2\pi \ii n \eps k)}{k} &= -\log(1 - \exp(2\pi \ii n \eps)) \,.
			\end{align*}
			This is extended by periodicity to \( \R \setminus \frac{1}{n} \Z \).  Similar computations to the above give the claimed asymptotic series.  
		\end{proof}
	\end{Lem}	
	
	\subsection{Truncated identity}

	We establish an identity here for the generating series of truncated multiple \( t \)-values with signs \(  \exp(2\pi\ii \phi_1), \ldots,  \exp(2\pi\ii \phi_m) \), which through careful analysis of the limit via the above asymptotic series will give the identity necessary for the proof of \autoref{thm:scon}. \medskip
	
	Let \( \phi_1,\ldots,\phi_m \in \mathbb{R} \), and define the following truncated series (cf. Goncharov \cite[\S{}2.6]{goncharov} as \( M \to \infty \), wherein he obtains instead a distribution on the \( m \)-torus).  We use variables \( y_1,\ldots,y_m \) to minimize the otherwise inevitable confusion between variables \( t_i \) and the multiple \( t \)-values themselves.  Let us also write \( \ee(x) \coloneqq \exp(2 \ii \pi x) \), and pre-emptively introduce the notation \( y_{i,j} = y_i - y_j \) for later convenience.
	
	\begin{Def}[$t$-Bernoulli series]\label{def:tbernoulli}
		The \emph{$t$-Bernoulli} series \( B^t_M \) of depth \( m \), truncated to order \( M \) is defined as:
		\begin{equation*}
		B^t_M(\phi_1,\ldots,\phi_m \mid y_1,\ldots,y_m) \coloneqq \sum_{-M \leq k_1 < \cdots < k_m \leq M} \frac{\ee(\phi_1 k_1 + \cdots + \phi_m k_m))}{(2k_1-1 - y_1) \cdots (2k_m-1 - y_m)} \,.
		\end{equation*}
		In the limit, the \emph{$t$-Bernoulli} series \( B^t_M \) of depth \( m \) is given as
		\[
		B^t(\phi_1,\ldots,\phi_m \mid y_1,\ldots,y_m) = \lim_{M\to\infty} B^t_M(\phi_1,\ldots,\phi_m \mid y_1,\ldots,y_m) \,.
		\]
	\end{Def}
	The motivation for this name comes from the fact that in the limit \( M \to \infty \), the depth 1 case can be evaluated via depth 1 multiple \( t \)-values, i.e. partially via Bernoulli numbers, as shown in \autoref{prop:bernoulli_depth1} below.  (See \cite[\S{}2.6]{goncharov} for the corresponding Bernoulli series in the case of MZV's.)
	
	\begin{Def}[Truncated MZV's and M$t$V's]
		The multiple zeta values, respectively multiple \( t \) values, with signs \( \eps_1,\ldots,\eps_m \in \{ z \in \C : \abs{z} = 1 \} \), truncated to order \( M \) are defined by
		\begin{align*}
		\zeta_M\sgnarg{\eps_1,\ldots,\eps_m}{n_1,\ldots,n_m} &\coloneqq \sum_{0 < k_1 < \cdots < k_m \leq M} \frac{\eps_1^{k_1} \cdots \eps_m^{k_m}}{k_1^{n_1} \cdots k_m^{n_m}} \,, \\
		t_M\sgnarg{\eps_1,\ldots,\eps_m}{n_1,\ldots,n_m} &\coloneqq \sum_{0 < k_1 < \cdots < k_m \leq M} \frac{\eps_1^{k_1} \cdots \eps_m^{k_m}}{(2k_1-1)^{n_1} \cdots (2k_m-1)^{n_m}} \,.
		\end{align*} 
	\end{Def}
	
	For \( (n_m,\eps_m) \neq (1,1) \), the both infinite series converge.  Corollary 2.3.10 in \cite{zhao} gives a proof that, for \( \abs{x_i} \leq 1 \), the multiple polylogarithm \( \Li_{n_1,\ldots,n_m}(x_1,\ldots,x_m) \) converges if (and only if) \( (n_m,x_m) \neq (1,1) \), which implies the claim for \( \zeta_M \).  Using the representation
	\[
	t_M\sgnarg{\eps_1,\ldots,\eps_m}{n_1,\ldots,n_m} = \frac{\sqrt{\eps_1} \cdots \sqrt{\eps_m}}{2^m} \sum_{s_1,\ldots,s_m \in \{ \pm1 \}}s_1 s_2\cdots s_m \zeta_{2M}\sgnarg{s_1 \sqrt{\eps_1}, \ldots, s_m \sqrt{\eps_m}}{n_1,\ldots,n_m} \,,
	\]
	we get the result for M$t$V's, as \( (\eps_m, n_m) \neq (1,1) \) implies \( (\pm \sqrt{\eps_m}, n_m) \neq (1,1) \), and so each of the \( 2^m \) MZV sums also converges.
	The (non-truncated) multiple zeta values and multiple \( t \)-values with signs \( \eps_1,\ldots,\eps_m \) are obtained via the limits as \( M \to \infty \), i.e.
	\begin{align*}
	\zeta\sgnarg{\eps_1,\ldots,\eps_m}{n_1,\ldots,n_m} \coloneqq \lim_{M\to\infty} 	\zeta_M\sgnarg{\eps_1,\ldots,\eps_m}{n_1,\ldots,n_m} & \quad\quad
	t\sgnarg{\eps_1,\ldots,\eps_m}{n_1,\ldots,n_m} =\coloneqq \lim_{M\to\infty} 	t_M\sgnarg{\eps_1,\ldots,\eps_m}{n_1,\ldots,n_m} \,.
	\end{align*}
	
	We now assemble the truncated MZV's and M$t$V's of fixed depth \( m \) into a generating series, as follows.
	\begin{Def}[Generating series for truncated MZV's and M$t$V's]
		The generating series of depth \( m \) truncated MZV's, respectively M$t$V's, with \emph{phases} \( \phi_1,\ldots,\phi_m \) (or equivalently, with signs \( \ee(\phi_1) \coloneqq \exp(2\pi\ii \phi_1), \ldots, \ee(\phi_m) \coloneqq \exp(2\pi\ii \phi_m) \)) are defined by
		\begin{align*}
		\Lic_M(\phi_1,\ldots,\phi_m \mid y_1,\ldots,y_m) &\coloneqq \sum_{n_1, \ldots, n_m \geq 1} \zeta_M\sgnarg{\ee(\phi_1),\ldots,\ee(\phi_m)}{n_1,\ldots,n_m} y_1^{n_1-1} \cdots y_m^{n_m-1} \\
		\Tic_M(\phi_1,\ldots,\phi_m \mid y_1,\ldots,y_m) &\coloneqq \sum_{n_1, \ldots, n_m \geq 1} t_M\sgnarg{\ee(\phi_1),\ldots,\ee(\phi_m)}{n_1,\ldots,n_m} y_1^{n_1-1} \cdots y_m^{n_m-1} \,
		\end{align*}
	\end{Def}
	In the case \( m = 1 \), we can evaluate the series \( B^t_M \) as follows.
	\begin{Prop}\label{prop:bernoulli_depth1}
		The following identity holds
		\begin{align*}
		B^t_M(\phi_1 \mid y_1) &\coloneqq \sum_{-M \leq k_1 \leq M} \frac{\ee(\phi_1 k_1)}{2k_1 - 1 - y_1} \\
		&=	\Tic_M(\phi_1 \mid y_1) - \ee(\phi_1) \Tic_{M+1}(-\phi_1 \mid -y_1)
		\end{align*}
		
		\begin{proof}
			Apply the geometric series
			\[
			\frac{1}{2k_1 - 1 - y_1} = \sum_{n=0}^\infty \frac{y_1^n}{(2k_1 - 1)^{n+1}} \,,
			\]
			to obtain
			\begin{align*}
			B^t_M(\phi_1 \mid y_1) &= \sum_{-M \leq k_1 \leq M} \sum_{n=0}^\infty \frac{\ee(k_1 \phi_1)}{(2k_1-1)^{n+1}} y_1^n \\
			&= \sum_{n=0}^\infty \bigg(\sum_{-M \leq k_1 \leq M} \frac{\ee(k_1 \phi_1)}{(2k_1-1)^{n+1}} \bigg) y_1^n
			\end{align*}
			Now observe
			\begin{align*}
			\sum_{-M \leq k_1 \leq M} \frac{\ee(k_1 \phi_1)}{(2k_1-1)^{n+1}} &= \sum_{k_1=1}^M \frac{\ee(k_1 \phi_1)}{(2k_1-1)^{n+1}} + \sum_{k_1=0}^M \frac{\ee(-k_1 \phi_1)}{(-2k_1-1)^{n+1}} \\
			& = \sum_{k_1=1}^M \frac{\ee(k_1 \phi_1)}{(2k_1-1)^{n+1}} - (-1)^n \ee( \phi_1) \sum_{k_1=1}^{M+1} \frac{\ee(-k_1 \phi_1)}{(2k_1-1)^{n+1}} \\
			& = t_M\sgnarg{\ee(\phi_1)}{n+1} - (-1)^n \ee(\phi_1) t_{M+1}\sgnarg{\ee(-\phi_1)}{n+1} \,
			\end{align*}
			which leads to the claimed generating series identity.
		\end{proof}
	\end{Prop}
	In particular, we obtain
	\[
	B^t(\phi_1 \mid y_1) = \Tic(\phi_1 \mid y_1) - \ee(\phi_1) \Tic(-\phi_1 \mid -y_1) \,,
	\]
	as \( M \to \infty \) (with \( \phi_1 \in \R \setminus \Z \)), reducing \( B^t(\phi_1 \mid y_1) \) to a generating series of depth 1 M$t$V's with signs \( \ee(\pm \phi_1) \). \medskip
	
	Now we seek to evaluate the \( t \)-Bernoulli series  \( B_M^t(\phi_1,\ldots,\phi_m \mid y_1,\ldots,y_m) \) of depth \( m \) in two different ways, via the MZV and M$t$V generating series introduced above. \medskip
	
	\paragraph{\em Evaluation 1:} Firstly, decompose the indexing set
	\[
	\{ -M \leq k_1 < \cdots < k_m \leq M \} = \bigcup_{j=0}^m \{ -M \leq k_1 < \cdots < k_j \leq 0 < k_{j+1} < \cdots < k_m \leq M \} \,,
	\]
	appearing in the series \( B_M^t(\phi_1,\ldots,\phi_m \mid y_1,\ldots,y_m) \).  We find
	\begin{align*}
	& \sum_{-M \leq \cdots < k_j \leq 0 < k_{j+1} < \cdots  \leq M} \frac{\ee(\phi_1 k_1 + \cdots + \phi_m k_m)}{(2k_1-1 - y_1) \cdots (2k_m - 1 - y_m)} \\
	& = \begin{aligned}[t] \sum_{0 < k_{j+1} < \cdots < k_m  \leq M} & \frac{\ee(\phi_{j+1} k_{j+1} + \cdots + \phi_m k_m)}{(2k_1-1 - y_1) \cdots (2k_m - 1 - y_m)} \cdot {} \\[-1ex]
	&  {} \cdot \sum_{0 < k_j' < \cdots < k_1' \leq M+1} \!\! \frac{\ee(\phi_1 + \cdots + \phi_j) \cdot \ee(-\phi_j k_j' - \cdots - \phi_1 k_1')}{(-2k_j'+1 - y_j) \cdots (-2k_1' + 1 - y_1)} \,,
	\end{aligned}
	\end{align*}
	where \( k_\ell' = 1-k_\ell \), for \( \ell = 1,\ldots,j \).  Notice the truncation bound of the right-hand factor is \( M+1 \) instead of \( M \).
	So directly, we have
	\begin{equation}\label{texp}
	\begin{aligned}[c]
	& B^t_M(\phi_1,\ldots,\phi_m \mid y_1,\ldots,y_m) = \\
	& \sum_{j=0}^m \begin{aligned}[t] (-1)^j \ee(\phi_1 + \cdots + \phi_j)
	\Tic_M(&\phi_{j+1}, \ldots, \phi_m \mid y_{j+1}, \ldots, y_m) \cdot {} \\
	& {} \cdot \Tic_{M+1}(-\phi_j, \ldots, -\phi_1 \mid -y_j, \ldots, -y_1) \,. \end{aligned}
	\end{aligned}
	\end{equation}
	
	\begin{Rem}Note that all terms in \autoref{texp}, except for \( j = 0, m \) are products.  In particular there is no term like
		\[
		\frac{1}{y_i} \Tic_M(\phi_2,\ldots,\phi_m \mid y_2, \ldots, y_m)
		\]
		which could introduce lower depth irreducibles to the result.  This is unlike the corresponding case for the usual polylgoarithms and MZV's as given in \cite[\S{}2.6]{goncharov}.  The lack of such a term for M$t$V's is ultimately the reason for the symmetry without lower depth irreducibles, as given in \autoref{thm:scon}.
	\end{Rem}
	
	\paragraph{\em Evaluation 2:} On the other hand, consider the decomposition (proven in \cite[Lemma 2.8]{goncharov})
	\[
	\frac{1}{(k_1 - y_1) \cdots (k_m - y_m)} = \sum_{j=1}^m \frac{1}{(k_j - y_j) \prod_{i \neq j} (k_{i,j} - y_{i,j})} \,,
	\]
	where \( k_{i,j} = k_i - k_j \), and \( y_{i,j} = y_i - y_j \).  Replacing \( k_i \) by \( 2k_i - 1 \), we obtain
	\begin{align*}
	& B^t_M(\phi_1,\ldots,\phi_m \mid y_1,\ldots,y_m)  = \sum_{j=1}^m \sum_{-M \leq k_1 < \cdots < k_m \leq M} \frac{\ee(\phi_1 k_1 + \cdots + \phi_m k_m)}{(2k_j-1 - y_j) \cdot \prod_{i \neq j} (2k_{i,j} - y_{i,j})}
	\end{align*}
	In the inner sum, where \( j \) is fixed, set \( k'_{i} = k_{i,j} \) for \( i > j \) and \( k'_{i} = -k_{i,j} \) for \( i < j \).  For \( i < j \), we see \( k'_i = k_j - k_i > 0 \), and for \( i > j \), \( k'_i = k_i - k_j > 0 \) also.  Moreover \( k_{j+1}' < k_{j+2}' < \cdots < k_m' \leq M - k_j \) and \( M + k_j \geq k_1' > k_2' > \cdots > k_{j-1}' \), since \( k_m \leq M \) and \( k_1 \geq -M \).  Note also that
	\[
	\phi_1 k_1 + \cdots + \phi_m k_m = \sum_{i < j} -\phi_i k'_i + (\phi_1 + \cdots + \phi_m) \cdot k_j + \sum_{i > j} \phi_i k'_i \,.
	\]
	This means
	\begin{align*}
	& B^t_M(\phi_1,\ldots,\phi_m \mid y_1,\ldots,y_m) = \\
	&	 \sum_{j=1}^m \sum_{-M \leq k_j \leq M} \begin{aligned}[t]  \Bigg( & \sum_{0 < k_{j-1}' < \cdots < k_1' \leq M + k_j}  \!\! \frac{\ee(-\phi_1 k_1' - \cdots - \phi_{j-1} k_{j-1}')}{\prod_{i < j} (- 2k_i' - y_{i,j})} \cdot {} \\[-1ex]
	& \quad {} \cdot \frac{\ee( (\phi_1 + \cdots+ \phi_m) k_j)}{(2k_j - 1 - y_j)} \! \sum_{0 < k_{j+1}' < \cdots < k_m' \leq M - k_j}  \! \frac{\ee(\phi_{j+1} k_{j+1}' + \cdots + \phi_{m} k_{m}')}{\prod_{i > j} (2k_i' - y_{i,j})} \Bigg) \end{aligned} \,.
	\end{align*}
	In terms of the \( \Lic_M \) generating series, we obtain
	\begin{equation}
	\label{pexp}
	\begin{aligned}
	& B^t_M(\phi_1,\ldots,\phi_m \mid y_1,\ldots,y_m) = \\
	& \frac{1}{2^{m-1}} \sum_{j=1}^m (-1)^{j-1} \!\!\!  \begin{aligned}[t] &  \sum_{-M \leq k_j \leq M} \bigg( \Lic_{M + k_j}(-\phi_{j-1}, \ldots, -\phi_1 \mid \tfrac{1}{2} y_{j,j-1} \,, \ldots, \tfrac{1}{2} y_{j,1}) \cdot {} \\
	& {} \cdot \frac{\ee( (\phi_1 + \cdots+ \phi_m) k_j)}{(2k_j - 1 - y_j)} \Lic_{M-k_j}(\phi_{j+1}, \ldots, \phi_m \mid \tfrac{1}{2} y_{j+1,j} \,, \ldots, \tfrac{1}{2} y_{m,j} ) \bigg) \,. \end{aligned}
	\end{aligned}
	\end{equation}\medskip
	
	Overall, by equating the representations \autoref{texp} and
	\autoref{pexp}, we have obtained the following result.
	\begin{Prop}\label{prop:truncgs} With \( y_{i,j} \coloneqq y_i - y_j \), and \( \ee(x) \coloneqq \exp(2\pi\ii x) \), the following generating series identity holds for all \( \phi_1, \ldots, \phi_m \), and all \( M \),
		\begin{align*}
		& \sum_{j=0}^m (-1)^j \ee(\phi_1 + \cdots + \phi_j)
		\begin{aligned}[t]
		\Tic_M(& \phi_{j+1}, \ldots, \phi_m \mid y_{j+1}, \ldots, y_m) \cdot {}  \\
		& {} \cdot \Tic_{M+1}(-\phi_j, \ldots, -\phi_1 \mid -y_j, \ldots, -y_1) 
		\end{aligned} \\
		& = \begin{aligned}[t]
		\frac{1}{2^{m-1}}  \sum_{j=1}^m (-1)^{j-1} \!\!\! & \sum_{-M \leq k_j \leq M} \bigg( \Lic_{M + k_j}(-\phi_{j-1}, \ldots, -\phi_1 \mid \tfrac{1}{2} y_{j,j-1} \,, \ldots, \tfrac{1}{2}y_{j,1}) \cdot {} \\
		& {} \cdot \frac{\ee(( \phi_1 + \cdots+ \phi_m) k_j)}{(2k_j - 1 - y_j)} \Lic_{M-k_j}(\phi_{j+1}, \ldots, \phi_m \mid \tfrac{1}{2} y_{j+1,j} \,, \ldots, \tfrac{1}{2} y_{m,j} ) \bigg) \,.
		\end{aligned}
		\end{align*}
	\end{Prop}
	
	We now wish to take \( \lim_{M\to\infty} \), but this requires some careful analysis to justify the result.  In particular, we should take \( \phi_1,\phi_m \in \R \setminus \Z \), so every MZV and M$t$V appearing in the above generating series \emph{does not} end in the pair \( (\eps, n) = (\exp(2\pi \ii \cdot \text{integer}), 1) = (1,1) \).  That is to say, so that every involved MZV and M$t$V is convergent, although perhaps only conditionally.  We return to the analytic issues momentarily. \medskip
	
	We introduce the generating series of (non-truncated) MZV's and M$t$V's as follows.
	\begin{Def}[Generating series for MZV's and M$t$V's]
		The generating series of all MZV's, respectively M$t$V's, of depth \( m \)  with \emph{phases} \( \phi_1,\ldots,\phi_m \) (or equivalently, with signs \( \ee(\phi_1) \coloneqq \exp(2\pi\ii \phi_1), \ldots, \ee(\phi_m) \coloneqq \exp(2\pi\ii \phi_m) \)) are defined by
		\begin{align*}
		\Lic(\phi_1,\ldots,\phi_m \mid y_1,\ldots,y_m) &= \sum_{n_1,\ldots,n_m\geq1} \zeta\sgnarg{\ee(\phi_1), \ldots, \ee(\phi_m)}{n_1,\ldots,n_m} y_1^{n_1-1} \cdots y_m^{n_m-1} \,. \\
		\Tic(\phi_1,\ldots,\phi_m \mid y_1,\ldots,y_m) &= \sum_{n_1,\ldots,n_m\geq1} t\sgnarg{\ee(\phi_1), \ldots, \ee(\phi_m)}{n_1,\ldots,n_m} y_1^{n_1-1} \cdots y_m^{n_m-1} \,.
		\end{align*}
	\end{Def}
	If \( \phi_m \in \R \setminus \Z \), every MZV in \( \Lic_M(\phi_1,\ldots,\phi_m \mid y_1, \ldots, y_m) \) is convergent, and we have for example:
	\[
	\lim_{M\to\infty}  \Lic_M(\phi_{1},\ldots,\phi_m \mid y_{1}, \ldots,y_m) = \Lic(\phi_{1},\ldots,\phi_m, \mid y_{1}, \ldots, y_m) \,,
	\]
	Likewise for the corresponding (truncated) M$t$V generating series.  \medskip
	
	Using \autoref{prop:specialsumconvergence}, we can rigorously justify passage to the limit \( \lim_{M\to\infty} \).  This proposition states essentially that
	\[
	\lim_{M\to\infty} \sum_{k=1}^M f_{M+k} g_{M-k} s_k = F G S \,.
	\]
	for sufficiently nice sequences \( f_k \to F \), \( g_k \to G \) and convergent series \( s_k \) with \( \sum_{k=1}^\infty s_k = S \).  The truncated MZV's and M$t$V's are sufficiently nice, as per \autoref{prop:zetatail}, so we can apply \autoref{prop:specialsumconvergence} term-by-term to the generating series in \autoref{prop:truncgs} (splitting the bidirectional sum into \( k < 0 \), \( k = 0  \) and \( k > 0 \) as necessary).  We have also from \autoref{prop:bernoulli_depth1} that
	\[
	\sum_{-\infty < k_j < \infty} \frac{\ee(\phi_1 + \cdots+ \phi_m) k_j)}{(2k_j - 1 - y_j)} = B^t(\phi_1 + \cdots + \phi_m \mid y_j)
	\]
	as per the definition of \( B^t(\phi \mid y) \), at least if \( \phi_1 + \cdots + \phi_m \in \mathbb{R} \setminus \mathbb{Z} \).  Therefore we have obtained the following result.
	
	\begin{Thm}\label{thm:symgsrestricted} With \( y_{i,j} \coloneqq y_i - y_j \), and \( \ee(x) \coloneqq \exp(2\pi\ii x) \), the following generating series identity holds for all \( \phi_1, \ldots, \phi_m \), whenever \( \phi_1, \phi_m, \phi_1 + \cdots + \phi_m \in \mathbb{R} \setminus \mathbb{Z} \):
		\begin{align*}
		& \sum_{j=0}^m (-1)^j \ee(\phi_1 + \cdots + \phi_j)
		\begin{aligned}[t]
		\Tic(& \phi_{j+1}, \ldots, \phi_m \mid y_{j+1}, \ldots, y_m) \cdot  \\
		& \cdot \Tic(-\phi_j, \ldots, -\phi_1 \mid -y_j, \ldots, -y_1) 
		\end{aligned} \\
		& = \frac{1}{2^{m-1}} \sum_{j=1}^m (-1)^{j-1}  \begin{aligned}[t]
		& \Lic(-\phi_{j-1}, \ldots, -\phi_1 \mid \tfrac{1}{2} y_{j,j-1} \,, \ldots, \tfrac{1}{2} y_{j,1} ) \cdot {} \\
		&  \cdot  B^t(\phi_1 + \cdots + \phi_m \mid y_j) \cdot \Lic(\phi_{j+1}, \ldots, \phi_m \mid \tfrac{1}{2} y_{j+1,j} \,, \ldots, \tfrac{1}{2} y_{m,j} ) \,.
		\end{aligned}
		\end{align*}
	\end{Thm}
	
	\subsection{Regularization statement}
	
	We turn the above generating series identity, which holds for \( \phi_1,\phi_m,\phi_1 + \cdots + \phi_m \in \mathbb{R}\setminus\mathbb{Z}\), into the corresponding identity for asymptotic series to obtain a result in the case where some \( \phi_i \to 0 \).
	
	Firstly, we establish a relation when changing variables between two different regularization parameters.  Extend the stuffle regularization \( \reg_T = \reg_T^\ast \) by linearity to the coefficients of the generating \( \Lic(\phi_1,\ldots,\phi_r \mid y_1,\ldots,y_r) \) of MZV's (correspondingly M$t$V's), and we have the following.
	
	\begin{Prop}\label{prop:reg}
		The following relation between regularized MZVs (respectively M$t$Vs) holds, where we assume \( \phi_m \neq 0 \):
		\begin{align*}
		&\reg_{T} \Lic(\phi_1, \ldots, \phi_m, \{0\}^\alpha \mid y_1,\ldots,y_m,y_{m+1},\ldots,y_{m+\alpha}) \\
		& = \sum_{i=0}^{\alpha} \reg_{S}  \Lic(\phi_1,\ldots,\phi_m,\{0\}^i \mid y_1,\ldots,y_m,y_{m+1},\ldots,y_{m+i}) \frac{(T-S)^{\alpha-i}}{(\alpha-i)!} \,.
		\end{align*}
		
		\begin{proof}
			On the level of zeta values, this is equivalent to the claim
			\[
			\reg_{T} \zeta\sgnarg{\phi_1, \ldots, \phi_m, \{1\}^\alpha}{n_1,\ldots,n_m, \{1\}^\alpha} = 
			\sum_{i=0}^\alpha \reg_{S} \zeta\sgnarg{\phi_1, \ldots, \phi_m, \{1\}^i}{n_1,\ldots,n_m, \{1\}^i} \frac{(T-S)^{\alpha-i}}{(\alpha-i)!} \,.
			\]
			This was show in \cite[Lemma 2.5]{charltont2212}; the therein proof relies only on the properties of the Hopf algebra, and so holds unchanged for M$t$V's.  A version holding for classical multiple zeta values (with no signs) is essentially given in \cite{ikz}, whose proof is generalized in \cite[Lemma 2.5]{charltont2212}.
		\end{proof}
	\end{Prop}
	
	It will be notationally convenient to write \[
	\Lic_{T=t_0}(\phi_1,\ldots,\phi_m \mid y_1,\ldots,y_m) \coloneqq \reg_{T=t_0} \Lic(\phi_1,\ldots,\phi_m \mid y_1,\ldots,y_m) \,,
	\]
	and correspondingly \( \Tic_{T=t_0} \) for \( \reg_{T=t_0} \Tic \).  At this point we have left behind any need for truncated MZV's (where the truncation parameter was always \( M \)), moreover we shall always be writing \( \Lic_{T=t_0} \) to indicate how the regularization parameter is specialised.  Therefore there should be no confusion.  Similarly we shall also write
	\[
	B^t_{T=t_0}(\phi \mid y) \coloneqq \reg_{T=t_0} B^t(\phi \mid y) \,,
	\]
	to denote the regularization with parameter \( T=t_0 \) of the $t$-Bernoulli series (of depth 1) from \autoref{def:tbernoulli} (expressed via \( \Tic \) using \autoref{prop:bernoulli_depth1}). \medskip
	
	Let \(\vec{\phi} = (\phi_i)_{i=1}^m \) be given, and without loss of generality, assume each \( \phi_i \in [0, 1) \).  We split into three cases at this point: \( \vec{\phi}\neq\vec{0}\) and \(\phi_1+\dots+\phi_m\neq 0\);
	\(\vec{\phi}\neq\vec{0}\) and \(\phi_1+\dots+\phi_m=0\);
	and \(\vec{\phi}=\vec{0}\), where \( \vec{0} \) denotes the zero vector of appropriate length. \medskip
	
	\paragraph{{\em Case 1, \(\ \vec{\phi} \neq \vec{0} \) and \( \phi_1 + \cdots + \phi_m \neq 0 \)}:}  Then necessarily we can find some \( \phi_n \neq 0 \) with \( 1 \leq n \leq m \).  We consider the small perturbation \( \phi'_i = \phi_i + c_i \) where 
	\[
	c_i = \begin{cases}
	- \eps & \text{\( 1 \leq i < n \)} \\
	\eps(2n- 1 - m) & \text{\( i = n \)} \\
	\eps & \text{\( n < i \leq m \).}
	\end{cases}
	\]
	That is, with the condition \( \phi'_1 + \cdots + \phi'_m = \phi_1 + \cdots + \phi_m \neq 0 \), and with \( \phi'_1, \phi'_m \neq 0 \).  By choosing \( \eps \) sufficiently small, we land in the case where \( \phi'_1, \phi'_m, \phi'_1 + \ldots + \phi'_m \neq 0 \), with the additional assumption \( \phi_n + c_n \neq 0 \).
	
	From \autoref{thm:symgsrestricted} we obtain the following identity, using the  \( \phi'_i \) parameters,
	\begin{equation}\label{eqn:case1pertgs}
	\begin{aligned}[c]
	& \sum_{j=0}^m (-1)^j \begin{aligned}[t]
	\ee(\phi_1 + \cdots + \phi_j +  (c_1 + \cdots + c_j)) &
	\Tic( \phi_{j+1} + c_{j+1}, \ldots, \phi_m + c_m \mid y_{j+1}, \ldots, y_m) \cdot {} \\
	& {} \cdot \Tic(-\phi_j - c_j, \ldots, -\phi_1 - c_1 \mid -y_j, \ldots, -y_1) 
	\end{aligned} \\
	& = \frac{1}{2^{m-1}} \sum_{j=1}^m \begin{aligned}[t] 
	(-1)^{j-1}  &\Lic(-\phi_{j-1} - c_{j-1}, \ldots, -\phi_1 - c_1 \mid \tfrac{1}{2} y_{j,j-1} \,, \ldots, \tfrac{1}{2} y_{j,1}) \cdot {} \\
	& {} \cdot  B^t(\phi_1 + \cdots + \phi_m \mid y_j) \cdot {} \\
	& {} \cdot \Lic(\phi_{j+1} + c_{j+1} , \ldots, \phi_m + c_m \mid \tfrac{1}{2} y_{j+1,j} \,, \ldots, \tfrac{1}{2} y_{m,j} )  \,.
	\end{aligned}
	\end{aligned}
	\end{equation}
	Now we shall compute the constant term of the asymptotic expansion, using the results from \autoref{sec:asymp} and the general theory from \autoref{sec:asympolylog}.  The key observation is that in each of these generating series, any signs of the form \( \ee(\phi_j + c_j) = \ee(\eps) \) (if \( j > n \) and \( \phi_j = 0 \)), at the end and \( \ee(\phi_i + c_i) = \ee(-\eps) \) (if \( i < n \) and \( \phi_i = 0 \)), at the start are separated by a term with sign \( \ee(\phi_n + c_n) \to \ee(\phi_n) \neq 1 \), as \( \eps\to0\), since \( \phi_n \neq 0 \) by assumption.  The constant term in the asymptotic expansion may then be computed by the corresponding stuffle regularization via \autoref{lem:asympviastuff}.  \medskip
	
	For example, when \( \vec{\phi} = (0, 0, 0, \frac{1}{10}, 0) \), we might encounter (in the perturbed series) the term
	\[
	Z = \zeta\Big(\begin{matrix} \ee(-\eps), & \ee(-\eps), & \ee(-\eps), & \ee(\frac{1}{10} + 2\eps), & \ee(\eps) \\ 1,&2,&1,&1,&1 \end{matrix}\Big) \,.
	\]
	We then have that
	\[
	[\log^0(1-e^{2\pi\ii\eps})] \reg^\asym Z = \reg_{T=0} \zeta\Big(\begin{matrix} 1, & 1, & 1, & \ee(\frac{1}{10}), & 1 \\ 1,&2,&1,&1,&1 \end{matrix}\Big) \,,
	\]
	since \( [\log^0(1-e^{2\pi\ii\eps})] \reg^\asym \zeta\sgnargsm{\ee(\eps)}{1} = 0 \) by \autoref{lem:zetasymp}.  The case of M$t$V's is analogous; we use \autoref{lem:tasymp}, to get \( [\log^0(1-e^{2\pi\ii\eps})] \reg^\asym t\sgnargsm{\ee(\eps)}{1} = \log{2} \) as the regularization parameter.  So applying this regularization prescription to the above generating series identity in \autoref{eqn:case1pertgs} gives, for \( \vec{\phi} \neq \vec{0} \) and \( \phi_1 + \cdots + \phi_m \neq 0 \),
	\begin{equation}\label{eqn:case1final}
	\begin{aligned}
	& \sum_{j=0}^m (-1)^j \ee(\phi_1 + \cdots + \phi_j) 
	\begin{aligned}[t]
	\Tic_{T=\log2}(& \phi_{j+1}, \ldots, \phi_m \mid y_{j+1}, \ldots, y_m) \cdot {} \\
	& {} \cdot \Tic_{T=\log2}(-\phi_j, \ldots, -\phi_1 \mid -y_j, \ldots, -y_1) 
	\end{aligned} \\
	& = 
	\frac{1}{2^{m-1}} \sum_{j=1}^m \begin{aligned}[t] & (-1)^{j-1}  \Lic_{T=0}(-\phi_{j-1}, \ldots, -\phi_1 \mid \tfrac{1}{2} y_{j,j-1} \, , \ldots, \tfrac{1}{2} y_{j,1}) \cdot {} \\
	& {} \cdot  B^t_{T=\log2}(\phi_1 + \cdots + \phi_m \mid y_j)  \Lic_{T=0}(\phi_{j+1}, \ldots, \phi_m \mid \tfrac{1}{2} y_{j+1,j} \,, \ldots, \tfrac{1}{2} y_{m,j} )  \,.
	\end{aligned}
	\end{aligned}
	\end{equation}

	\paragraph{{\em Case 2, \( \vec{\phi} \neq \vec{0} \) and \( \phi_1 + \cdots + \phi_m = 0 \)}:} Then necessarily we can find some \( \phi_n \neq 0 \) with \( 1 \leq n \leq m \).  We consider the small perturbation \( \phi'_i = \phi_i + c_i \) where 
	\[
	c_i = \begin{cases}
	- \eps & \text{\( 1 \leq i < n \)} \\
	\eps(2n - m) & \text{\( i = n \)} \\
	\eps & \text{\( n < i \leq m \)\,.}
	\end{cases}
	\]
	so that \( \phi'_1 + \cdots + \phi'_m = \phi_1 + \cdots + \phi_m + \eps = \eps \neq 0 \), and \( \phi'_1 \neq 0, \phi'_m \neq 0 \).  We note that the constant term of the asymptotic expansion of \( B^t(\eps \mid y) \) is given as follows
	\begin{align*}
	\reg^\asym_0 B^t(\eps \mid y) &= \reg^\asym_0 \big( \Tic(\eps \mid y) - \exp(2\pi \ii \eps) \Tic(-\eps \mid -y) \big) \\
	&= \Tic_{T=\log2}(0 \mid y) - \Tic_{T=\log2 - \frac{i\pi}{2}}(0 \mid -y) \\
	&= \frac{\ii\pi}{2} + B^t_{T=\log2}(0 \mid y) \,,
	\end{align*}
	using \( \reg^\asym_0 t\sgnargsm{\ee(\eps)}{1} = \log{2} \) and \( \reg^\asym_0 t\sgnargsm{\ee(-\eps)}{1} = \log{2} - \frac{\ii\pi}{2} \) from \autoref{lem:tasymp}, and the fact that the only divergent M$t$V in \( \Tic_{T=t_0}(0 \mid y) \) is \( \reg_{T=t_0} t(1) \).  (Alternatively, use \autoref{prop:reg}.) That is to say, we pick up an extra \( \frac{\ii\pi}{2} \) contribution in this case.
	
	From \autoref{thm:symgsrestricted}, we obtain the following identity using the \( \phi_i' \) parameters:
	\begin{align*}
	& \sum_{j=0}^m (-1)^j \ee(\phi_1 + \cdots + \phi_j + (c_1 + \cdots + c_j)\eps)
	\begin{aligned}[t]
	& \Tic( \phi_{j+1} + c_{j+1}, \ldots, \phi_m + c_m \mid y_{j+1}, \ldots, y_m) \cdot {} \\
	& {} \cdot \Tic(-\phi_j - c_j, \ldots, -\phi_1 - c_1 \mid -y_j, \ldots, -y_1) 
	\end{aligned} \\
	& = 
	\frac{1}{2^{m-1}} \sum_{j=1}^m (-1)^{j-1} \begin{aligned}[t] & \Lic(-\phi_{j-1} - c_{j-1}, \ldots, -\phi_1 - c_1 \mid \tfrac{1}{2} y_{j,j-1} \,, \ldots, \tfrac{1}{2}y_{j,1}) \cdot {} \\
	& \cdot  B^t(\eps \mid y_j) \Lic(\phi_{j+1} + c_{j+1}\,, \ldots, \phi_m + c_m \mid \tfrac{1}{2}  y_{j+1,j} \,, \ldots, \tfrac{1}{2} y_{m,j} )  \,.
	\end{aligned}
	\end{align*}
	Upon taking the constant term in the asymptotic expansion, we find the following identity:
	\begin{equation}\label{eqn:case2final}
	\begin{aligned}[c]
	& \sum_{j=0}^m (-1)^j \ee(\phi_1 + \cdots + \phi_j)
	\begin{aligned}[t]
	& \Tic_{T=\log2}( \phi_{j+1}, \ldots, \phi_m \mid y_{j+1}, \ldots, y_m) \cdot  \\
	& \cdot \Tic_{T=\log2}(-\phi_j, \ldots, -\phi_1 \mid -y_j, \ldots, -y_1) 
	\end{aligned} \\
	& = 
	\frac{1}{2^{m-1}} \sum_{j=1}^m \begin{aligned}[t] (-1&)^{j-1} \Lic_{T=0}(-\phi_{j-1}, \ldots, -\phi_1 \mid \tfrac{1}{2} y_{j,j-1} \,, \ldots, \tfrac{1}{2} y_{j,1}) \cdot \\
	&  \cdot  \Big(\frac{\ii\pi}{2} + B^t_{T=\log2}(0 \mid y_j) \! \Big)  \Lic_{T=0}(\phi_{j+1}, \ldots, \phi_m \mid \tfrac{1}{2}  y_{j+1,j} \,, \ldots, \tfrac{1}{2} y_{m,j} ) \,.
	\end{aligned}
	\end{aligned}
	\end{equation}
	With the following lemma, we can eliminate the additional \( \tfrac{\ii\pi}{2} \) term above, and hence obtain that the same identity as in \autoref{eqn:case1final} of Case 1 holds, also when \( \phi_1 + \cdots + \phi_m = 0 \).
	\begin{Lem}\label{lem:shuffle}
		The following holds for all \( \vec{\phi} \neq \vec{0} \) with \( \phi_1 + \ldots + \phi_m = 0 \).
		\begin{equation}\label{eqn:lishuffleasym}
		\sum_{j=1}^m (-1)^{j-1} \begin{aligned}[t] & \Lic_{T=0}(-\phi_{j-1}, \ldots, -\phi_1 \mid \tfrac{1}{2} y_{j,j-1} \, , \ldots, \tfrac{1}{2} y_{j,1}) \\
		& \cdot \Lic_{T=0}(\phi_{j+1}, \ldots, \phi_m \mid \tfrac{1}{2} y_{j+1,j} \,, \ldots, \tfrac{1}{2} y_{m,j} ) = 0
		\end{aligned}
		\end{equation}
		
		\begin{proof}
			We consider a similar small perturbation \( \phi'_i = \phi_i + c_i \), where \( n \) is such that \( \phi_n \neq 0 \) and 
			\[
			c_i = \begin{cases}
			- \eps & \text{\( 1 \leq i < n \)} \\
			\eps(2n - m) & \text{\( i = n \)} \\
			\eps & \text{\( n < i \leq m \)\,,}
			\end{cases}
			\]
			giving \( \vec{\phi}' \) with \( \phi'_1 + \cdots + \phi'_m = 0 \).  Expanding the above claim for \( \vec{\phi}' \), wherein no regularization is necessary, via the \emph{shuffle} product, one can show that the combination
			\begin{align}\label{eqn:lishuffle}
			\sum_{j=1}^m (-1)^{j-1} \begin{aligned}[t] \Lic(& -\phi'_{j-1}, \ldots, -\phi'_1 \mid \tfrac{1}{2} y_{j,j-1} \,, \ldots, \tfrac{1}{2} y_{j,1}) \\
			& \cdot \Lic(\phi'_{j+1}, \ldots, \phi'_m \mid \tfrac{1}{2} y_{j+1,j}\,, \ldots, \tfrac{1}{2} y_{m,j} ) = 0
			\end{aligned}
			\end{align}
			is identically 0. 
			
			To show this, we first translate into a generating series for iterated integrals.  Namely from \cite[Equation (26)]{goncharov}, we have
			\begin{align*}
			& \Lic(\phi_1,\ldots\phi_m \mid y_1,\ldots,y_m) = \\ 
			& (-1)^m \mathcal{I}(\ee(-\phi_1-\cdots-\phi_m), \ee(-\phi_2-\cdots-\phi_m), \cdots, \ee(-\phi_m) \mid y_1,\ldots,y_m) \,,
			\end{align*}
			where 
			\[
			\mathcal{I}(x_1,\ldots,x_m \mid y_1,\ldots, y_m) \coloneqq \sum_{n_1,\ldots,n_m\geq1} I_{n_1,\ldots,n_m}(x_1,\ldots,x_m) y_1^{n_1-1} \cdots y_m^{n_m-1} \,,
			\]
			is the generating series of iterated integrals with fixed depth \( m \), and arguments \( x_1,\ldots,x_m \).  Here \( I_{n_1,\ldots,n_m}(x_1,\cdots,x_m) \coloneqq I(0; x_1, \{0\}^{n_1-1}, \ldots, x_m, \{0\}^{n_m-1}, 1) \), with
			\[
			I(x_0; x_1,\ldots,x_N;x_{N+1}) = \int_{x_0 < t_1 < \cdots < t_N < x_{N+1}} \frac{\dd t_1}{t_1 - x_1} \wedge \cdots \wedge \frac{\dd t_N}{t_N - x_N} \,,
			\]
			the iterated integral over a family of differential forms.
			
			By making the change of power-series variables
			\[
			\mathcal{I}^!(\vec{x} \mid y_1,\ldots,y_m) \coloneqq \mathcal{I}(\vec{x} \mid y_1, y_1+y_2, \ldots, y_1 + y_2 + \cdots + y_m) \,,
			\]
			Theorem 2.9 in \cite{goncharov} shows that \( \mathcal{I}^! \) satisfies the shuffle product formula, namely
			\begin{align*}
			& \mathcal{I}^!(x_1,\ldots,x_m \mid y_1 \ldots, y_m) \mathcal{I}^!(x_{m+1},\ldots,x_{m+\ell} \mid y_{m+1} \,,\ldots,y_{m+\ell})  \\
			&= \sum_{\sigma \in \Sigma_{m,\ell}} \mathcal{I}^!(x_{\sigma(1)}, \ldots, x_{\sigma(m+\ell)} \mid y_{\sigma(1)},\ldots,y_{\sigma(m+\ell)}) \,,
			\end{align*}
			where \( \Sigma_{m,\ell} \) is the set of \( (m,\ell) \)-shuffles of \( \{1,\ldots,m+\ell\} \), i.e. the set of all permutations \( \sigma \) of \( \{1,\ldots,m+\ell\} \), such that \( \sigma(1) < \sigma(2) < \cdots < \sigma(m) \) and \( \sigma(m+1) < \sigma(m+2) < \cdots < \sigma(m+\ell) \).
			
			In particular, with \( y_{i,j} = y_i - y_j \), we find
			\begin{align*}
			& \Lic(-\phi'_{j-1}, \ldots, -\phi'_1 \mid \tfrac{1}{2} y_{j,j-1} \,, \ldots, \tfrac{1}{2} y_{j,1}) \\
			&= \mathcal{I}^!(\ee(\phi'_{j-1} + \cdots + \phi'_1), \ee(\phi'_{j-2} + \cdots + \phi'_1), \ldots, \ee(\phi'_1) \mid \tfrac{1}{2} y_{j,j-1}, \tfrac{1}{2} y_{j-2,j-1} \,, \ldots, \tfrac{1}{2} y_{3,2}, \tfrac{1}{2} y_{2,1}) \,,
			\end{align*}
			and
			\begin{align*}
			& \Lic(\phi'_{j+1}, \ldots, \phi'_m \mid \tfrac{1}{2} y_{j+1,j} \,, \ldots, \tfrac{1}{2} y_{m,j} ) \\
			&= \mathcal{I}^!(\ee(-\phi'_{j+1} - \cdots - \phi'_m), \ee(-\phi'_{j+2} - \cdots - \phi'_m), \ldots, \ee(-\phi'_m) \mid \tfrac{1}{2} y_{j+1,j} \,, \tfrac{1}{2} y_{j+2,j-1} \,, \ldots, \tfrac{1}{2} y_{m, m-1}) \,.
			\end{align*}
			Consider now the variables \( x_i = \ee(-\phi'_{i+1} - \cdots - \phi'_m) \) and \( t_i = \tfrac{1}{2} y_{i+1,i} \).  From the condition \( \phi'_1 + \cdots + \phi'_m = 0 \) it follows that \( 	\ee(\phi'_{j-1} + \ldots + \phi'_1) = x_{m+1-i} \).
			Upon translation of the claimed identity in \autoref{eqn:lishuffle} into these coordinates we find it is equivalent to
			\[
			\sum_{j=1}^m (-1)^{j-1} \mathcal{I}^!(x_1,\ldots,x_j \mid t_1,\ldots,t_j) \mathcal{I}^!(x_m,\ldots,x_{j+1} \mid t_m \ldots, t_{j+1}) = 0 \,.
			\]
			But this identity holds in any shuffle algebra (compare, e.g., \cite[Ex. (29)]{galois}).  Then \autoref{eqn:lishuffleasym} follows by taking the constant term in the asymptotic expansion, which (by the previous discussion) is equivalent to taking the stuffle regularization with \( T = 0 \), and so completes the proof.
		\end{proof}
	\end{Lem}

	\paragraph{{\em Case 3: \( \vec{\phi}=\vec{0} \)}.} 
	In this case, it is no longer possible to guarantee the indices 1 with phase \( \eps \) are separated from the indices 1 with phase \( -\eps \), since there is no non-zero phase at which to split.  That is, we would encounter  say \( \zeta\sgnargsm{\ee(-\eps), \ee(\eps), \ee(\eps)}{1,1,1} \) whose asymptotic series is not obtained by the stuffle regularization directly, as we separately regularize \( \reg_0^\asym \zeta\sgnargsm{\ee(-\eps)}{1} = -\frac{i\pi}{2} \), and \( \reg_0^\asym \zeta\sgnargsm{\exp(\eps)}{1} = 0 \), giving a `hybrid' form.
	
	Instead we take the small perturbation \( \phi'_i = \eps \).  From \autoref{thm:symgsrestricted} and taking the constant term in the asymptotic series (using \autoref{lem:tasymp} and \autoref{lem:zetasymp}), we obtain
	\begin{equation}\label{eqn:gscase3}
	\begin{aligned}
	& \sum_{j=0}^m (-1)^j 
	\begin{aligned}[t]
	\Tic_{T = \log2}(& \{0\}^{m-j} \mid y_{j+1}, \ldots, y_m) \Tic_{T = \log2 - \frac{\pi \ii}{2}}(\{0\}^j \mid -y_j, \ldots, -y_1) 
	\end{aligned} \\
	& =
	\frac{1}{2^{m-1}} \sum_{j=1}^m (-1)^{j-1}  \begin{aligned}[t] & \Lic_{T=-\ii\pi}(\{0\}^{j-1} \mid \tfrac{1}{2} y_{j,j-1} \, , \ldots, \tfrac{1}{2}y_{j,1}) \\
	& \cdot \Big( \frac{\ii\pi}{2} + B^t_{T=\log2}(0 \mid y_j) \Big) \Lic_{T=0}(\{0\}^{m-j}  \mid \tfrac{1}{2} y_{j+1,j} \,, \ldots, \tfrac{1}{2} y_{m,j} )   \,.
	\end{aligned}
	\end{aligned}
	\end{equation}
	We will now apply induction to show that all of the regularizations above can be replaced by the standard ones \( \Lic_{T=0} \) and \( \Tic_{T=\log2} \), up to an additional constant.  In particular, we make the following claim.
	\begin{Lem}\label{clm:id}
		The following identity holds
		\begin{equation}\label{eqn:claim}
		\begin{aligned}
		& \sum_{j=0}^m (-1)^j 
		\begin{aligned}[t]
		\Tic_{T = \log2}(& \{0\}^{m-j} \mid y_{j+1}, \ldots, y_m) \Tic_{T = \log2}(\{0\}^j \mid -y_j, \ldots, -y_1) 
		\end{aligned} \\
		& - 
		\frac{1}{2^{m-1}} \sum_{j=1}^m (-1)^{j-1} \begin{aligned}[t] & \Lic_{T=0}(\{0\}^{j-1} \mid \tfrac{1}{2} y_{j,j-1} \,, \ldots, \tfrac{1}{2} y_{j,1}) \\
		& \cdot B^t_{T=\log2}(0 \mid y_j)  \Lic_{T=0}(\{0\}^{m-j}  \mid \tfrac{1}{2} y_{j+1,j}\,, \ldots, \tfrac{1}{2} y_{m,j} )   \\[2ex]
		& = \begin{cases}
		\displaystyle\frac{1}{m!} \bigg(\displaystyle\frac{\ii\pi}{2} \bigg)^m & \text{\( m \) even} \\[1ex]
		\quad 0 & \text{\( m \) odd.}
		\end{cases}
		\end{aligned}
		\end{aligned}
		\end{equation}
	\end{Lem}
	
	\begin{proof} We prove this by induction on \( m \), first establishing the base cases. 
		
		\medskip
		\paragraph{{\em Case \( m = 1\)}:} In the case \( m = 1 \), we obtain from \autoref{eqn:gscase3} that
		\[
		\Lic_{T=\log2}(0\mid y_1) - \Lic_{T=\log2 - \frac{i\pi}{2}}(0 \mid -y_1)
		- \Big(\frac{\ii\pi}{2} + B^t_{T=\log2}(0 \mid y_1)\Big) = 0
		\]
		Expanding using the regularization change of variables in \autoref{prop:reg}, we obtain the case \( m = 1 \) of this lemma, as the terms \( \pm \frac{\ii\pi}{2} \) cancel.  
		
		\medskip
		\paragraph{{\em Case \( m = 2\)}:}  Likewise, when \( m = 2 \), we obtain from \autoref{eqn:gscase3} that
		\begin{align*}
		&\Lic_{T=\log2}(0,0 \mid y_1, y_2) - \Lic_{T=\log2}(0 \mid y_2) \Lic_{T=\log2 -\frac{i\pi}{2}}(0 \mid -y_1) + \Lic_{T=\log2 - \frac{\ii\pi}{2}}(0, 0 \mid -y_2, -y_1) \\
		& - \frac{1}{2}\Big( \frac{\ii\pi}{2} + B^t_{T=\log2}(0 \mid y_1)\!\Big)\Lic_{T=0}(0 \mid \tfrac{1}{2} y_{2,1})
		+ \frac{1}{2} \Lic_{T=-\ii\pi}(0 \mid \tfrac{1}{2} y_{2,1}) \Big( \frac{\ii\pi}{2} + B^t_{T=\log2}(0 \mid y_2) \! \Big)   = 0 .
		\end{align*}
		Expanding this out with the regularization change of variables in  \autoref{prop:reg}, we find
		\begin{align*}
		& \Tic_{T=\log2}(0, 0 \mid y_1, y_2) - \Tic_{T=\log2}(0 \mid -y_1) \Tic_{T=\log2}(0 \mid y_2) + \Tic_{T=\log2}(0, 0, \mid -y_2, -y_1) \\
		& - \frac{1}{2} \big( \Lic_{T=0}(0 \mid \tfrac{1}{2} y_{2,1}) B_{T=\log2}^t(0 \mid y_1) - \Lic_{T=0}(0 \mid \tfrac{1}{2} y_{2,1}) B_{T=\log2}^t(0 \mid y_2) \big) \\
		& + \frac{\ii\pi}{2} \Big(
		\Tic_{T=\log2}(0 \mid y_2) - \Tic_{T=\log2}(0 \mid -y_2) - B^t_{T=\log2}(0 \mid y_2)
		\Big) - \frac{1}{2!} \Big(\frac{\ii\pi}{2}\Big)^2 = 0
		\end{align*}
		The first term on line 3 is just case \( m = 1 \) of \autoref{clm:id}, for power series variables \( (y_2) \), which we already showed is equal to 0.  We therefore obtain the case \( m = 2 \) of this lemma, including the claimed constant.
		
		\medskip
		\paragraph{{\em Induction hypotheses}:} 	Now we assume the following identities hold for depth \( < m \), namely \autoref{eqn:claim}, and the following shuffle product formula:
		\begin{equation}\label{eqn:shuffleclaim}
		\begin{aligned}
		\sum_{j=1}^m (-1)^{j-1}   \begin{aligned}[t] & \Lic_{T=0}(\{0\}^{j-1} \mid \tfrac{1}{2} y_{j,j-1} \, , \ldots, \tfrac{1}{2}y_{j,1}) \\[-2ex]
		& \cdot \Lic_{T=0}(\{0\}^{m-j} \mid \tfrac{1}{2} y_{j+1,j}\,, \ldots, \tfrac{1}{2} y_{m,j} ) = \begin{cases}
		\displaystyle\frac{(\ii \pi)^{m-1}}{m!} & \text{\( m \) odd,} \\
		\quad 0 & \text{otherwise.}
		\end{cases}
		\end{aligned}
		\end{aligned}
		\end{equation}
		The shuffle product identity is quickly verified for \( m = 1 \), as series on the left hand side are `depth 0', and so the left hand side is trivially 1.  Whereas for \( m = 2 \), the left hand side is
		\[
		\Lic_{T=0}(0 \mid \tfrac{1}{2} y_{2,1}) + (-1) \Lic_{T=0}(0 \mid \tfrac{1}{2} y_{2,1}) = 0  \,,
		\]
		and so trivially vanishes.  We may proceed to the case \( m \geq 3 \).
		
		\medskip
		\paragraph{{\em Case \( m \geq 3\)}:} 	By \autoref{prop:reg}, we have the following change of regularization parameter results,
		\begin{align*}
		\Tic_{T=\log2 - \tfrac{\ii\pi}{2}}(\{0\}^j \mid -y_j,\ldots, -y_1) &= \sum_{\ell=0}^j \Tic_{T=\log2}(\{0\}^{j-\ell} \mid -y_j,\ldots,-y_{\ell+1}) \cdot \frac{(-\ii\pi/2)^\ell}{\ell!} \,, \\
		\Lic_{T=-\ii\pi}(\{0\}^{j-1} \mid \tfrac{1}{2} y_{j,j-1} \,, \ldots, \tfrac{1}{2} y_{j,1}) &= \sum_{\ell=0}^{j-1} \Lic_{T=0}(\{0\}^{j-1-\ell} \mid \tfrac{1}{2} y_{j,j-1}\,, \ldots, \tfrac{1}{2} y_{j,\ell+1}) \cdot \frac{(-\ii\pi)^{\ell}}{\ell!}  \,.
		\end{align*}
		Now consider the depth \( m \) case of \autoref{eqn:gscase3}, and with the above change of regularizations; we find the following combination is 0,
		\begin{align*}
		& \sum_{j=0}^m (-1)^j 
		\begin{aligned}[t]
		\Tic_{T = \log2}(& \{0\}^{m-j} \mid y_{j+1}, \ldots, y_m) \cdot \bigg(\!\sum_{\ell=0}^j \Tic_{T=\log2}(\{0\}^{j-\ell} \mid -y_j,\ldots,-y_{\ell+1}) \cdot \frac{(-\ii\pi/2)^\ell}{\ell!} \bigg)
		\end{aligned} \\[-1ex]
		& - \begin{aligned}[t]
		\frac{1}{2^{m-1}} \sum_{j=1}^m  (-1)^{j-1}  \Big( \sum_{\ell=0}^{j-1} & \Lic_{T=0}(\{0\}^{j-1-\ell} \mid \tfrac{1}{2} y_{j,j-1}, \tfrac{1}{2} y_{j,j-2}\,,\ldots, \tfrac{1}{2} y_{j,\ell+1}) \cdot \frac{(-\ii\pi)^{\ell}}{\ell!}  \Big) \\[-2ex]
		& \cdot  \Big( \frac{\ii\pi}{2} + B^t_{T=\log2}(0 \mid y_j) \Big) \Lic_{T=0}(\{0\}^{m-j}  \mid \tfrac{1}{2} y_{j+1,j}, \ldots, \tfrac{1}{2} y_{m,j} )  = 0 \,.
		\end{aligned}
		\end{align*}
		Now switch the summation order of \( j \) and \( m \) in both (double) sums, and re-index both \( j \) sums with \( j \mapsto j - \ell \) (keeping the same variable for notational ease with \( y \)-subscripts).  We find (simplifying some powers at the same time)
		\begin{align*}
		& \sum_{\ell = 0}^m \frac{(\ii\pi/2)^\ell}{\ell!} \bigg( \sum_{j = 0}^{m-\ell} (-1)^{j}  
		\begin{aligned}[t]
		\Tic_{T = \log2}(& \{0\}^{m-j-\ell} \mid y_{j+\ell+1}, \ldots, y_m) \Tic_{T=\log2}(\{0\}^{j} \mid -y_{j+\ell},\ldots,-y_{\ell+1}) \! \bigg)
		\end{aligned} \\
		& - 
		\sum_{\ell = 0}^{m-1}  \frac{(\ii\pi/2)^{\ell} }{\ell!} \frac{1}{2^{m-\ell-1}} \sum_{j = 1}^{m-\ell} \begin{aligned}[t] (-1)^{j-1}  \Lic_{T=0}( &\{0\}^{j-1} \mid \tfrac{1}{2} y_{j+\ell,j+\ell-1}, \tfrac{1}{2} y_{j+\ell,j+\ell-2},\ldots, \tfrac{1}{2} y_{j+\ell,\ell+1}) \cdot  \\
		& \cdot  \Big( \frac{\ii\pi}{2} + B^t_{T=\log2}(0 \mid y_{j+\ell}) \Big) \cdot \\
		& \cdot \Lic_{T=0}(\{0\}^{m-j-\ell}  \mid \tfrac{1}{2} y_{j+\ell+1,j+\ell} \,, \ldots, \tfrac{1}{2} y_{m,j+\ell} )  = 0 \,.
		\end{aligned}
		\end{align*}
		We notice that the \( \ell = m \) term of the first sum contributes only \( \frac{(\ii\pi/2)^m}{m!} \), as the \( \Tic_{T=\log2} \) series have `depth 0' and are simply 1.  Whereas, the \( \ell = 0 \) terms are the combination we wish to evaluate to prove \autoref{eqn:claim} in the case \( m \).  The remaining terms evaluate by the induction hypothesis; in particular we find:
		\begin{align*}
		&  \sum_{j = 0}^{m} (-1)^{j} 
		\Tic_{T = \log2}( \{0\}^{m-j} \mid y_{j+1}, \ldots, y_m) \Tic_{T=\log2}(\{0\}^{j} \mid -y_{j},\ldots,-y_{1}) \\[-1ex]
		& \quad\quad - \frac{1}{2^{m-1}} \sum_{j = 1}^{m} \begin{aligned}[t] (-1)^{j-1}  \Lic_{T=0}( &\{0\}^{j-1} \mid \tfrac{1}{2} y_{j,j-1}, \tfrac{1}{2} y_{j,j-2},\ldots, \tfrac{1}{2} y_{j,1}) \cdot  \\
		& \cdot  \Big( \frac{\ii\pi}{2} + B^t_{T=\log2}(0 \mid y_{j}) \Big) \Lic_{T=0}(\{0\}^{m-j}  \mid \tfrac{1}{2} y_{j+1,j} \,, \ldots, \tfrac{1}{2} y_{m,j} ) 
		\end{aligned} \\[1ex]
		& \, = \, - \frac{(\ii\pi/2)^m}{m!} - \sum_{\ell = 1}^{m-1} \frac{(\ii\pi/2)^\ell}{\ell!} \bigg\{ \sum_{j = 0}^{m-\ell} (-1)^{j}  
		\begin{aligned}[t]
		\Tic_{T = \log2}(& \{0\}^{m-j-\ell} \mid y_{j+\ell+1}, \ldots, y_m) \cdot \\
		& \cdot \Tic_{T=\log2}(\{0\}^{j} \mid -y_{j+\ell},\ldots,-y_{\ell+1}) 
		\end{aligned} \\
		& \quad\quad - \frac{1}{2^{m-\ell-1}} \sum_{j = 1}^{m-\ell} \begin{aligned}[t]& (-1)^{j-1}  \Lic_{T=0}( \{0\}^{j-1} \mid \tfrac{1}{2} y_{j+\ell,j+\ell-1}, \tfrac{1}{2} y_{j+\ell,j+\ell-2},\ldots, \tfrac{1}{2} y_{j+\ell,\ell+1}) \cdot  \\
		& \cdot  \Big( \frac{\ii\pi}{2} + B^t_{T=\log2}(0 \mid y_{j+\ell}) \Big)  \Lic_{T=0}(\{0\}^{m-j-\ell}  \mid \tfrac{1}{2} y_{j+\ell+1,j+\ell} \,, \ldots, \tfrac{1}{2} y_{m,j+\ell} ) \bigg\}
		\end{aligned}
		\end{align*}
		Apply the induction hypotheses (shuffle product, and \autoref{eqn:claim}, for cases \( <m \)), we find the right hand side is given as follows
		\begin{align*}
		& = - \frac{(\ii\pi/2)^m}{m!} - \sum_{l=1}^{m-1} \frac{(\ii\pi/2)^\ell}{\ell!} \Bigg( \delta_{\text{$m{-}\ell$ even}} \cdot \frac{ ( \ii\pi / 2 )^{m-\ell}}{(m-\ell)!} - \frac{1}{2^{m - \ell - 1}} \frac{\ii\pi}{2} \delta_\text{$m{-}\ell$ odd} \cdot \frac{(\ii\pi)^{m-\ell-1}}{(m-\ell)!}  \Bigg)  \\
		& = - \frac{(\ii\pi/2)^m}{m!} - \sum_{l=1}^{m-1} \frac{(\ii\pi/2)^\ell}{\ell!}  \cdot (-1)^{m-\ell} \frac{ ( \ii\pi / 2 )^{m-\ell}}{(m-\ell)!}   \\
		& = - \frac{(\ii\pi/2)^m}{m!} - \sum_{l=1}^{m-1} \frac{(\ii\pi/2)^m}{m!} \cdot (-1)^{m-\ell} \binom{m}{\ell} 
		=  (-1)^m \frac{(\ii\pi/2)^m}{m!} \,, 
		\end{align*}
		via the binomial theorem.  By comparing the real and imaginary parts, and considering the cases where \( m \) is even, or \( m \) is odd, we establish \autoref{eqn:claim} and the shuffle product identity \autoref{eqn:shuffleclaim} for the case \( m \geq 3 \).  This proves the lemma.
	\end{proof}
	
	We can therefore state the following identity, by combining the individual cases from \autoref{eqn:case1final}, \autoref{eqn:case2final} (after removing the \( \frac{\ii\pi}{2} \) term with \autoref{lem:shuffle}, and \autoref{eqn:claim}.
	
	\begin{Thm}[Symmetry Theorem, generating series form] \label{thm:symgsfull}
		Let \( y_{i,j} \coloneqq y_i - y_j \), and \( \ee(x) \coloneqq \exp(2\pi\ii x) \).  Then for every choice of \( m \) and every choice of \( \vec{\phi} = (\phi_1,\ldots,\phi_m) \), following generating series identity holds:
		\begin{align*}
		& \sum_{j=0}^m (-1)^j \ee(\phi_1 + \cdots + \phi_j)
		\begin{aligned}[t]
		\Tic_{T=\log2}(& \phi_{j+1}, \ldots, \phi_m \mid y_{j+1}, \ldots, y_m) \cdot {} \\
		& {} \cdot \Tic_{T=\log2}(-\phi_j, \ldots, -\phi_1 \mid -y_j, \ldots, -y_1) 
		\end{aligned} \\
		& -
		\frac{1}{2^{m-1}} \sum_{j=1}^m \begin{aligned}[t]  (-1)^{j-1} & \Lic_{T=0}(-\phi_{j-1}, \ldots, -\phi_1 \mid \tfrac{1}{2} y_{j,j-1} \,, \ldots, \tfrac{1}{2} y_{j,1}) \cdot {} \\
		& \cdot  B^t_{T=\log2}(\phi_1 + \cdots + \phi_m \mid y_j)  \Lic_{T=0}(\phi_{j+1}, \ldots, \phi_m \mid \tfrac{1}{2} y_{j+1,j}\,, \ldots, \tfrac{1}{2} y_{m,j} ) 
		\end{aligned} \\
		& = \begin{cases}
		\displaystyle\frac{1}{m!} \Big( \displaystyle\frac{\ii\pi}{2} \Big)^m & \text{\( \vec{\phi} = \vec{0} \) and \( m \) even,} \\[2ex]
		\quad 0 & \text{otherwise\,.}
		\end{cases}
		\end{align*}
	\end{Thm}
	
	This theorem provides a version of the Symmetry Theorem in which the product terms consist of products of multiple \( t \)-values, and products of multiple zeta values.  In order to establish the claim as stated in \autoref{thm:scon}, we must convert the multiple zeta products into multiple \( t \) products.
	
	To that end, we recall the following result from \cite{murakami} relating multiple \( t \)-values and multiple zeta values.
	
	\begin{Thm}[Murarkami, Theorem 8 in \cite{murakami}]
		The motivic \( t \)-values \( t^\mot(k_1,\ldots,k_d) \), \( k_i \in \{2,3\} \) are a basis for motivic multiple zeta values
	\end{Thm}
	
	This is a result about so-called `motivic' objects, the technical details of which are not so important here.  Upon applying the period map, the upshot of Murakami's result is that every multiple zeta value can be written as a linear combination of multiple \( t \)-values, in fact of the special form \( t(k_1,\ldots,k_d) \), \( k_i \in \{2,3\} \).
	
	Consider then the case \( \vec{\phi} = \vec{0} \) of \autoref{thm:symgsfull}.  We see that by extracting the coefficient of \( y^{n_1 - 1} \cdots y^{n_m - 1} \), and writing every multiple zeta value as a linear combination of multiple \( t \)-values via Murakami's Theorem, the following holds.  For \( n_1 + \cdots + n_m \geq 3 \),
	\[
	\reg_{T=\log2} \Big( \! t(n_1,\ldots,n_m) + (-1)^{n_1 + \cdots + n_m} t(n_m, \ldots, n_1) \! \Big) = 0 \pmod{\text{\rm products of \( t \)'s}}
	\]
	
	This proves the version of the Symmetry Theorem given in \autoref{thm:scon}.
	
	\begin{Rem}
		We point out that the only non-trivial case of weight \( \leq 2 \), namely \( (n_1,n_2) = (1,1) \), gives
		\[
		\reg_{T=\log2}^\ast \big( t(1,1) + (-1)^2 t(1,1) \big) = \reg_{T=\log2}^\ast \big( t(1)^2 - t(2) \big) \,.
		\]
		The Symmetry Theorem therefore does not hold in this case, since \( t(2) \) is irreducible, at least when viewed as a multiple \( t \)-value.  But \( t(2) = \frac{\pi^2}{8} \) could be viewed as decomposable relative to alternating multiple \( t \)-values; it is a product \( \pi \times \pi \), where \( \pi = -4t(\overline{1}) \) is a strictly alternating \( t \) value.  (Recall the notation \( \overline{k} \) means the sign of argument \( k \) is \( \eps = -1 \) in the framework of alternating MZV's and M$t$V's.)
	\end{Rem}
	
	In the case \( \vec{\phi} \in \{ 0, \frac{1}{2} \}^m \setminus \{ (0,\ldots,0) \} \), we have the following relationship between alternating \( t \)-values and alternating multiple zeta values.
	
	\begin{Thm}[Charlton, Corollary 8.26 in \cite{charltont2212}]
		The stuffle regularized motivic multiple \( t \)-values \( t^{\mot,\ast}(k_1,\ldots,k_d) \), \( k_i \in \{1,2\} \) (with regularization parameter \( t^{\mot,\ast}(1) = \log^\mot(2) \)) are a basis for motivic alternating multiple zeta values.
	\end{Thm}
	
	After application of the period map, we have that every alternating multiple zeta value can be written as a linear combination of multiple \( t \)-values of the form \( \reg^\ast_{T=\log2} t(k_1,\ldots,k_d) \), \( k_i \in \{1,2\}\), regularized with \( \reg^\ast_{T=\log2} t(1) = \log2 \).
	
	Generally in the case \( \vec{\phi} \in \{ 0, \frac{1}{2} \}^m \setminus \{(0, \ldots, 0)\} \) we have a corresponding symmetry result.   By extracting the coefficient of \( y^{n_1 - 1} \cdots y^{n_m - 1} \), and writing every alternating multiple zeta value as a linear combination of alternating \( t \)-values, the following symmetry result holds.
	
	\begin{Cor}
		The following symmetry result holds, for any choice of \( \eps_i = \pm 1 \):
		\begin{align*}
		\reg_{T=\log2}^\ast \bigg( \! t\sgnarg{\eps_1, \ldots, \eps_m}{n_1,\ldots,n_m} + (-1)^{n_1 +\cdots + n_m} (-1)^{\#\{ \eps_i = -1 \}} t\sgnarg{\eps_m, \ldots, \eps_1}{n_m,\ldots,n_1} \! \bigg) \\
		= 0 \pmod{\text{\rm products of alternating $t$'s}} \,.
		\end{align*}
	\end{Cor}
	
	\section{Applications}
	\label{sec:applications}
	In this section we apply the Symmetry Theorem to obtain a number
	of evaluations of multiple $t$-values, multiple $t$-star values,
	and multiple $t^{\half}$-values.
	\subsection{\texorpdfstring{Evaluations of $t(3,\{2\}^n,3)$ and $t^\star(3,\{2\}^n,3)$}{Evaluations of t(3,\{2\}\textasciicircum{}n,3) and t\textasciicircum{}*(3,\{2\}\textasciicircum{}n,3)}}\label{sec:t3223andstar}
	
	\subsubsection{\texorpdfstring{Evaluation of $t(3,\{2\}^n,3)$}{Evaluation of t(3,\{2\}\textasciicircum{}n,3)}}
	\label{sec:t3223}
	
	Here we will establish the following evaluation for \( t(3,\{2\}^n,3) \).  In the following section we shall also obtain a similar expression for the corresponding \( t^\star \)-values.
	
	\begin{Thm}[Evaluation of \( t(3,\{2\}^n,3) \)]\label{thm:t3223:eval}
		We have the following evaluation for \( t(3,\{2\}^n,3) \) as a polynomial in Riemann zeta values.
		\begin{align*}
		t(3,\{2\}^n,3) &= \Big({-}\frac{1}{4}\Big)^{n+2} \Big\{ {-}\frac{9+6n}{2} \zeta(2) \zeta(2n+4) \\
		& \quad\quad+ \sum_{\substack{q+r+s=n+2 \\ q,r,s\geq 1}} 2 r s (2 - 2^{-2r} - 2^{-2s}) \zeta(2r + 1)\zeta(2s+1) \cdot (-1)^q \frac{\pi^{2q}}{(2q)!} \\
		& \quad\quad + \sum_{\substack{r+s = n+2 \\ r,s \geq 1}}2 r s (2 - 2^{-2r})(2 - 2^{-2s}) \zeta(2r+1) \zeta(2s+1) \Big\} \,.
		\end{align*}
	\end{Thm}
	
	\begin{proof}
		We start by extracting the coefficient of \( y_1^2y_2\cdots y_{n+1} y_{n+2}^2 \) from the generating series expression in \autoref{thm:symgsfull}, in the case \( \phi = \vec{0} \).  In this case, the Symmetry Theorem reads
		\begin{equation}\label{eqn:tsymphi0gs:fort3223}
		\begin{aligned}
		& \sum_{j=0}^m (-1)^j \begin{aligned}[t]
		\Tic_{T=\log2}(& \vec{0} \mid y_{j+1}, \ldots, y_m) \Tic_{T=\log2}(\vec{0} \mid -y_j, \ldots, -y_1) 
		\end{aligned} \\
		& - 
		\frac{1}{2^{m-1}} \sum_{j=1}^m \begin{aligned}[t] (-1)^{j-1}  & \Lic_{T=0}(\vec{0} \mid \tfrac{1}{2} y_{j,j-1} \,, \ldots, \tfrac{1}{2} y_{j,1}) \cdot {} \\ & {} \cdot  B^t_{T=\log2}(0 \mid y_j) \Lic_{T=0}(\vec{0} \mid \tfrac{1}{2} y_{j+1,j} \,, \ldots, \tfrac{1}{2} y_{m,j} ) 
		\end{aligned} \\
		& = \begin{cases}
		\displaystyle\frac{1}{m!} \Big( \displaystyle\frac{\ii\pi}{2} \Big)^m & \text{\( m \) even,} \\[1ex]
		\quad 0 & \text{\( m \) odd\,.}
		\end{cases}
		\end{aligned}
		\end{equation}
		
		Combining all of the contributions and rearranging slightly gives the following identity as the coefficient of \( y_1^2 y_2 \cdots y_{n+1} y_{n+2}^2 \) in \autoref{eqn:tsymphi0gs:fort3223}:
		\begin{equation}\label{eqn:t3223:orig}
		\begin{aligned}
		2 t(3,\{2\}^n,3) = {} & 
		\sum_{i=0}^n t(\{2\}^i, 3) t(\{2\}^{n-i}, 3) - \frac{1}{2^{2n+3}} t(2) \sum_{i=0}^{n-1} \zeta(\{2\}^i,3) \zeta(\{2\}^{n-1-i},3) \\
		& - \frac{1}{2^{2n+2}} t(2) \Big\{ 3 \zeta(\{2\}^{n},4) + 2 \sum_{i=0}^{n-1} \zeta(\{2\}^i, 3, \{2\}^{n-1-i}, 3) \Big\} \,.
		\end{aligned}
		\end{equation}
		At this point we have a reduction of \( t(3,\{2\}^n,3) \) to more familiar objects, and can proceed to reduce the right-hand side.  It will be convenient to make use of the shuffle regularization to understand the MZV combination appearing in \autoref{eqn:t3223:orig}.
		
		\begin{Rem}[Shuffle regularization]
			We give a brief reminder of the shuffle regularization of MZV's.  One can write any MZV as an iterated integral in the following way
			\begin{equation}\label{eqn:zetaasint}
			\zeta(n_1,\ldots,n_\ell) = (-1)^\ell I(0; 1, \{0\}^{n_1 - 1}, \ldots, 1, \{0\}^{n_\ell-1}; 1) \,,
			\end{equation}
			where
			\[
			I(x_0; x_1,\ldots,x_N;x_{N+1}) = \int_{x_0 < t_1 < \cdots < t_N < x_{N+1}} \frac{\dd t_1}{t_1 - x_1} \wedge \cdots \wedge \frac{\dd t_N}{t_N - x_N} \,.
			\]
			Viewing \( x_1,\ldots,x_N \) as a word \( w = x_1 \cdots x_N \) over the alphabet \( \{0,1\} \) (in the case of MZV's), one obtains a product structure on these integrals via the shuffle product \( \shuffle \).  The shuffle product is inductively defined by
			\[
			a w_1 \shuffle b w_2 = a (w_1 \shuffle b w_2) + b (a w_1 \shuffle w_2) \,,
			\]
			and the integral (with fixed bounds) extended to formal linear combinations of words.  Then any integral \( I(0; x_1,\ldots,x_N; 1) \) with \( x_1 = 0 \) or \( x_N = 1 \) (therefore non-convergent) may be written as a polynomial in \( I(0; 0; 1) \) and \( I(0; 1; 1) \) with MZV coefficients.  Replacing \( I(0; 0; 1) \) and \( I(0; 1; 1) \) by the same indeterminate \( T \) (in order to preserve duality under \( t_i \mapsto 1-t_i \)), we obtain the shuffle-regularization polynomial \( \reg_{T}^\shuffle I(0; x_1, \ldots, x_N; 1) \).  Finally \( \reg_{T}^\shuffle \zeta(n_1,\ldots,n_\ell) \) is defined as \( \reg_{T}^\shuffle \) of the corresponding integral from \autoref{eqn:zetaasint}
		\end{Rem}
		
		We note now that the zeta combination in the second line of \autoref{eqn:t3223:orig} can be obtained from the following shuffle regularization, with \( \reg_{T=0}^\shuffle I(0; 0; 1) = \reg_{T=0}^\shuffle I(0; 1; 1) = 0 \), i.e.,
		\begin{align*}
		3 \zeta(\{2\}^{n},4) + 2 \sum_{i=0}^{n-1} \zeta(\{2\}^i, 3, \{2\}^{n-1-i}, 3) &= 
		-(-1)^{n+1}\reg^\shuffle_{T=0} I(0; 0, \{1,0\}^n, 1, 0, 0; 1) \\
		& = -\reg^\shuffle_{T=0}\zeta(1,\{2\}^{n+1},1) \,.
		\end{align*}
		The second equality follows by applying duality.  The first equality holds by direct calculation since
		\begin{align*}
		& -(-1)^{n+1} \reg_{T=0}^\shuffle I(0; 0,\{1,0\}^{n},1,0,0 ; 1)  \\
		&= (-1)^{n} \Big[ \begin{aligned}[t] 
		I(0; \{1,0\}^{n},1,0,0 ; 1) \overbrace{\reg_{T=0}^\shuffle I(0;0;1)}^{=0} - 3I(0; \{1,0\}^{n},1,0,0,0;1) \\[-1ex]
		- 2 \sum_{i=0}^{n-1} I(0; \{1,0\}^i ,1,0,0, \{1,0\}^{n-1-i}, 1,0,0; 1) \Big] \end{aligned} \\
		&= 3\zeta(\{2\}^{n}, 4) + 2\sum_{i=0}^{n-1} \zeta(\{2\}^i,3,\{2\}^{n-1-i},3)
		\end{align*}
		With this in mind, the reduction for \( t(3,\{2\}^n,3) \) from \autoref{eqn:t3223:orig} can be written as
		\begin{equation}\label{eqn:t3223:new:withreg}
		\begin{aligned}
		t(3,\{2\}^n,3) = {}	& \frac{1}{2} \sum_{i=0}^n t(\{2\}^i, 3) t(\{2\}^{n-i}, 3) - \frac{ t(2) }{2^{2n+4}}\sum_{i=0}^{n-1} \zeta(\{2\}^i,3) \zeta(\{2\}^{n-1-i},3) \\
		& + \frac{ t(2) }{2^{2n+3}} \reg^\shuffle_{T=0} \zeta(1,\{2\}^{n+1},1) \,.
		\end{aligned}
		\end{equation}
		
		We now apply the following shuffle-antipode identity (reversing the integral string, cf. Equation (28) in \cite{galois})
		\[
		\sum_{i=0}^m (-1)^i I(0; w_1,w_2,\ldots,w_i; 1) I(0; w_m, w_{m-1}, \ldots, w_{i+1}; 1) = 0 \,,
		\]
		which implies that \( I(0; w_1,w_2,\ldots,w_m; 1) + (-1)^m I(0; w_m, w_{m-1}, \ldots, w_1; 1) = 0 \pmod{\mathrm{products}} \).  This identity applied to \( \reg_{T=0}^\shuffle \zeta(1,\{2\}^n,1) \) shows that
		\begin{equation}\label{eqn:z211:rev}
		\reg_{T=0}^\shuffle  \Big( \sum_{i=0}^n \zeta(1,\{2\}^i,1) \zeta(\{2\}^{n-i}) + \zeta(\{2\}^n, 1,1) - \sum_{i=0}^n \zeta(\{2\}^i,1) \zeta(1,\{2\}^{n-i}) \Big) = 0 \,,
		\end{equation}
		the main term of which gives \( \reg_{T=0}^\shuffle  \big( \zeta(1,\{2\}^n,1) + \zeta(\{2\}^n,1,1) \big) = 0 \pmod{\mathrm{products}} \).
		
		At this point we will find it convenient to reformulate everything in terms of generating series.  To this end, we introduce the following general notation.
		\begin{Def}[General generating series for \( \zeta \) and \( t \)]\label{def:Gnotation}
			We define the following notation representing the generating functions of MZV's, respectively M$t$V's, whose arguments have the form \( (\alpha, \{\beta\}^n, \gamma) \), for arbitrary \( \alpha, \beta, \gamma \), as follows.
			\begin{align*}
			G_{\alpha\{\beta\}\gamma}(u) &\coloneqq \sum_{n=0}^\infty \reg_{T=0}^\shuffle \zeta(\alpha, \{\beta\}^n, \gamma) u^{\abs{\alpha} + n\abs{\beta} + \gamma} \,, \\
			G^{t,\bullet}_{\alpha\{\beta\}\gamma}(u) &\coloneqq \sum_{n=0}^\infty \reg_{T=\log2}^\ast t^\bullet(\alpha, \{\beta\}^n, \gamma) u^{\abs{\alpha} + n\abs{\beta} + \gamma} \,,
			\end{align*}
			where \( \bullet \in \{ \emptyset, \star, \half, r \} \).  That is, \( G_{\alpha\{\beta\}\gamma}(u) \) is the generating function of the \emph{shuffle}-regularized MZV's \( \reg_{T=0}^\shuffle \zeta(\alpha, \{\beta\}^n, \gamma) \) weighted by the MZV weight, and \( G^{t,\bullet}_{\alpha\{\beta\}\gamma}(u) \) is the corresponding generating function of \emph{stuffle}-regularized (interpolated) M$t$V's \( \reg_{T=\log2}^\ast t^\bullet(\alpha, \{\beta\}^n, \gamma) \).
		\end{Def}
		It is also convenient to introduce the following similar generating series for MZV's and M$t$V's of arguments with two repeating blocks of 2's:
		\begin{align*}
		F(u,v) &= \sum_{a,b=0}^\infty (-1)^{a+b} \zeta(\{2\}^a, 3, \{2\}^b) u^{2a} v^{2b} \,,\\
		F^t(u,v) &= \sum_{a,b=0}^\infty (-1)^{a+b} t(\{2\}^a, 3, \{2\}^b) u^{2a} v^{2b} \,.
		\end{align*}
		We know from Zagier \cite{zagier2232}, that \( F(u,v) \) is given as follows (after slightly rewriting Zagier's expression)
		\begin{align*}
		F(u,v) & = \sum_{a,b\geq0}^\infty (-1)^{a+b} \zeta(\{2\}^a,3,\{2\}^b) u^{2a} v^{2b} \\
		& = \frac{\sin(\pi v)}{\pi u^2 v} \Big[ A(u+v) + A(u-v) - 2 A(v) \Big] - \frac{\sin(\pi u)}{\pi u^2 v} \Big[ B(u+v) - B(u-v) \Big],
		\end{align*}
		where \( A(z) = \sum_{r=1}^\infty \zeta(2r + 1) z^{2r} \) and \( B(z) = \sum_{r = 1}^\infty (1 - 2^{-2r}) \zeta(2r + 1) z^{2r} \).  Likewise, from Murakami \cite{murakami}, we know --- with the same series \( A(z) \) and \( B(z) \) --- that
		\begin{align*}
		F^t(u,v) & = \sum_{a,b\geq0}^\infty (-1)^{a+b} t(\{2\}^a,3,\{2\}^b) u^{2a} v^{2b} \\
		& = \frac{1}{2 u v} \cos\Big(\frac{\pi v}{2}\Big) \Big[ A\Big(\frac{u+v}{2}\Big) - A\Big(\frac{u-v}{2}\Big) \Big] + \frac{1}{2 u v} \cos\Big(\frac{\pi u}{2}\Big)  \Big[ B\Big(\frac{u+v}{2}\Big) - B\Big(\frac{u-v}{2}\Big) \Big] \,.
		\end{align*}
		The series \( G^t_{3\{2\}3}(u) \)  encapsulates the M$t$V's that we want to evaluate.  We note that by \autoref{eqn:t3223:new:withreg} it can expressed as follows:
		\begin{equation}\label{eqn:gt3223:identity}
		G^t_{3\{2\}3}(u) = \frac{1}{2} G^t_{\{2\}3}(u)^2 - t(2) u^2 G_{\{2\}3}(\tfrac{u}{2})^2 + 2 t(2) u^2 G_{1\{2\}1}(\tfrac{u}{2}) \,.
		\end{equation}
		By \autoref{eqn:z211:rev}, we have
		\[
		G_{1\{2\}1}(u) G_{\{2\}}(u) + G_{\{2\}11}(u) - G_{\{2\}1}(u) G_{1\{2\}}(u) = 0 \,,
		\]
		or equivalently
		\begin{equation}\label{eqn:gz121:identity}
		G_{1\{2\}1}(u) = \Big( G_{\{2\}1}(u) G_{1\{2\}}(u) - G_{\{2\}11}(u) \Big) G_{\{2\}}(u)^{-1} \,.
		\end{equation}
		(We abuse the notation for \( G_{\{2\}11} \) slightly, as \( 11 \) is treated as the string \( 1, 1 \) in the generating series.)
		We now recall the following standard result.
		\[
		G_{\{2\}}(u) = \sum_{n=0}^\infty \zeta(\{2\}^n) u^{2n} = \sum_{n=0}^\infty \frac{\pi^{2n}}{(2n+1)!} u^{2n} = \frac{\sinh(\pi u)}{\pi u} \,.
		\]
		Then by taking \( \lim_{v\to0} F(u,v) \) we find \[
		\ii G_{\{2\}3}(\ii u) = u^3 F(u, 0) = -2 u A(u) + \frac{2}{\pi} u \sin(\pi u) B'(u) \,,
		\]
		so that (recall \( A \) and \( B \) are even, so \( B' \) is odd)
		\[
		G_{\{2\}3}(u) = 2 u A(\ii u) + \frac{2 \ii}{\pi} u B'(\ii u) \sinh(\pi u) \,.
		\]
		Since \( \zeta(\{2\}^n,3) = \zeta(1,\{2\}^{n+1}) \) by duality, and \( \reg_{T=0}^\shuffle \zeta(1) = 0 \) gives no additional contribution, we immediately also have that
		\[
		G_{1\{2\}}(u) = G_{\{2\}3}(u) = 2 u A(\ii u) + \frac{2 \ii}{\pi} u B'(\ii u) \sinh(\pi u) \,.
		\]
		Likewise, we find, by taking \( \lim_{v\to0} F^t(u,v) \), that
		\[
		\ii G^t_{\{2\}3}(\ii u) = u^3 F^t(u, 0) = \frac{1}{2} u^2 A'\Big(\frac{u}{2}\Big) + \frac{1}{2} u^2 B'\Big(\frac{u}{2}\Big) \cos\Big( \frac{\pi u}{2} \Big) \,,
		\]
		so that
		\[
		G^t_{\{2\}3}(u) = -\frac{\ii}{2} u^2 A'\Big(\frac{\ii u}{2}\Big) - \frac{\ii}{2} u^2 B'\Big(\frac{\ii u}{2}\Big) \cosh\Big(\frac{\pi u}{2}\Big).
		\]
		
		Therefore we need only consider \( G_{\{2\}1}(t) \) and \( G_{\{2\}11}(t) \), both of which can be handled in a similar way.  Namely by duality and shuffle regularization we find
		\begin{align*}
		\reg_{T=0}^\shuffle \zeta(\{2\}^n,1) &= \reg^\shuffle_{T=0} (-1)^{n+1} I(0; 0, \{1,0\}^n; 1) = -2 \sum_{i=0}^{n-1} \zeta(\{2\}^i, 3, \{2\}^{n-1-i}) \\
		\reg_{T=0}^\shuffle \zeta(\{2\}^n,1,1) &= \reg^\shuffle_{T=0} (-1)^{n+2} I(0; 0, 0, \{1,0\}^n; 1) \\
		& = -3 \sum_{i=0}^{n-1} \zeta(\{2\}^i, 4, \{2\}^{n-1-i}) - 4 \sum_{i+j+k=n-2} \zeta(\{2\}^i, 3, \{2\}^j, 3, \{2\}^k)
		\end{align*}
		The first one can be expressed as $-$2 times the sum of all MZV's of weight \( w = 2n + 1 \), depth \( d = n \) and height \( h = n \), and hence evaluated via the Ohno-Zagier Theorem.  The latter consists of all weight \( w = 2n + 2 \), depth \( d = n \) and height \( h = n \) MZV's, although the coefficients are not constant, so does not evaluate immediately with the Ohno-Zagier Theorem, but can be tweaked to allow this.  (In both cases, the sums obtained are symmetric sums, and hence are rational polynomials in Riemann zeta values, though these may be difficult to evaluate explicitly.)  To this end, recall the Ohno-Zagier theorem.	
		\begin{Thm}[Ohno-Zagier, \cite{ohno-zagier}]
			The generating series obtained by summing all MZV's of fixed weight \( w \), height \( h \) and depth \( d \) is given by
			\begin{align*}
			\phi_0(x,y,z) &{} \coloneqq \sum_{\substack{w \geq d + h \\ h \leq d}} \Big\{ \sum_{\substack{\wt(\vec{k}) = w \\ \operatorname{dp}(\vec{k}) = d \\ \operatorname{ht}(\vec{k}) = h}} \zeta(\vec{k}) \Big\} x^{w-d-h} y^{d-h} z^{h-1} \\
			&= \frac{1}{xy - z} \Big\{ 1 - \exp\Big( \sum_{m=2}^\infty \frac{1}{m} \zeta(m) (x^m + y^m - \alpha^m - \beta^m ) \Big) \Big\} \,
			\end{align*}
			where
			\[
			\alpha, \beta = \frac{x + y \pm \sqrt{(x + y)^2 - 4 z}}{2} \,.
			\]
		\end{Thm}
		For the sum corresponding to \( \reg_{T=0}^\shuffle\zeta(\{2\}^n, 1) \), we find
		\begin{align*}
		\sum_{n=0}^\infty \Big( \sum_{i=0}^{n} \zeta(\{2\}^i, 3, \{2\}^{n-i}) \Big) z^n & {} = \frac{\partial\phi_0}{\partial x}(0, 0, z) \\
		& = \frac{1}{z} \exp\Big(\sum_{n=1}^\infty \frac{(-1)^{n-1}}{n} z^n \zeta(2n) \Big) \cdot \sum_{n=1}^\infty (-1)^{n-1} z^n \zeta(2n+1) \,.
		\end{align*}
		Whereas, for the following sum, related to that for \( \reg_{T=0}^\shuffle\zeta(\{2\}^n, 1, 1) \), we have
		\begin{align*}		
		& \sum_{n=0}^\infty \Big( \sum_{i+j+k=n-1} \zeta(\{2\}^i,3,\{2\}^j,3,\{2\}^k) + \sum_{i=0}^n \zeta(\{2\}^i,4,\{2\}^{n-i}) \Big) z^n \\
		&= \frac{1}{2} \frac{\partial^2 \phi_0}{\partial x^2}(0, 0, z)  = \frac{1}{2z} \exp\Big(\sum_{n=1}^\infty \frac{(-1)^{n-1}}{n} z^n \zeta(2n) \Big) \cdot \Big( \sum_{n=1}^\infty (-1)^{n-1} z^n \zeta(2n+1) \Big)^2 + \\ &\quad\quad\quad \frac{1}{2z} \exp\Big(\sum_{n=1}^\infty \frac{(-1)^{n-1}}{n} z^n \zeta(2n) \Big) \cdot  \sum_{n=1}^\infty (-1)^{n-1} (n+1) z^n \zeta(2n+2)  \,.
		\end{align*}
		In order to deal with the series appearing above, we make the following observations.  Firstly
		\[
		\sum_{n=1}^\infty \frac{(-1)^{n-1}}{n} z^n \zeta(2n) = -\log\big(\Gamma(1 + \ii \sqrt{z})\Gamma(1 - \ii \sqrt{z})\big) = \log\Big(\frac{\sinh(\pi \sqrt{z})}{\pi \sqrt{z}}\Big)\,,
		\]
		using the Taylor expansion
		\begin{equation}\label{eqn:taylorLogGamma}
		\log \Gamma(1 - z) = \gamma z +  \sum_{k=2}^\infty \frac{1}{k} \zeta(k) z^k \,.
		\end{equation}
		along with the reflection formula \( \Gamma(1-z)\Gamma(1+z) = \pi z \csc(\pi z) \) for the gamma function.  Moreover
		\[
		\sum_{n=1}^\infty (-1)^{n-1} z^n \zeta(2n+1) = -A(\ii \sqrt{z}) \, ,
		\]
		so
		\begin{align}
		& \sum_{n=0}^\infty \Big( \sum_{i=0}^{n} \zeta(\{2\}^i, 3, \{2\}^{n-i}) \Big) z^n  = -\frac{A(\ii \sqrt{z}) \sinh(\pi \sqrt{z})}{\pi z \sqrt{z}} \,, \label{eqn:232sum} \\
		& \sum_{n=0}^\infty \Big( \sum_{i+j+k=n-1} \zeta(\{2\}^i,3,\{2\}^j,3,\{2\}^k) + \sum_{i=0}^n \zeta(\{2\}^i,4,\{2\}^{n-i}) \Big) z^n \notag \\
		& = \frac{\sinh(\pi \sqrt{z})}{2 \pi z \sqrt{z}} \Big( A(\ii \sqrt{z})^2 + \zeta(2) - 2 t(2) \frac{\coth(\pi \sqrt{z})}{\pi \sqrt{z}} + 2 t(2) \operatorname{csch}(\pi \sqrt{z})^2   \Big) \label{eqn:2323and242sum} .
		\end{align}
		The second through fourth terms in the second identity come from explicitly summing
		\[
		\sum_{n=1}^\infty (-1)^{n-1} (n+1) \zeta(2n+2) z^n \,.
		\]
		
		On the other hand, we first have the following identity (with \( \widetilde{\shuffle} \) denoting the shuffle of the argument strings, rather than the shuffle product of the associated MZV's):
		\[
		\zeta(\{2\}^a \widetilde{\shuffle} \{4\}) = \sum_{i=0}^{a} \zeta(\{2\}^i, 4, \{2\}^{a-i}) = \frac{8 \pi^{2a + 4}}{(2a + 6)!} \binom{a + 3}{a} \,.
		\]
		This follows by simplifying the result
		\[
		\zeta(\{2\}^a \widetilde{\shuffle} \{4\}) = \zeta(\{2\}^{a+1}) \ast \zeta(\{2\}) - (a + 2) \zeta(\{2\}^{a+2}) = \frac{\pi^{2a+2}}{(2a + 3)!} \cdot \frac{\pi^2}{6} - (a+2) \frac{\pi^{2a+4}}{(2a + 5)!} \,,
		\]
		where \( \ast \) denotes stuffle product, or alternatively by taking $k=n-1=a+1$ in \cite[Theorem 1.4]{hoffman17}.	Then
		\begin{equation} \label{eqn:242sum}
		\begin{aligned}
		\sum_{n=0}^\infty \Big( \sum_{i=0}^{n} \zeta(\{2\}^i, 4, \{2\}^{n-i}) \Big) z^n
		& = \sum_{n=0}^\infty \frac{8 \pi^{2n + 4}}{(2n + 6)!} \binom{n + 3}{n} z^n  \\
		& = -\frac{\cosh(\pi \sqrt{z})}{2 z^2} + \frac{\sinh(\pi \sqrt{z})}{2 \pi z^{5/2}} + \zeta(2) \frac{\sinh(\pi \sqrt{z})}{\pi z^{3/2}} .
		\end{aligned}
		\end{equation}
		With appropriate changes of variables, and combining the above series, we find the following.  From $-2$ times \autoref{eqn:232sum} we have
		\begin{align*}
		G_{\{2\}1}(u) = \frac{2}{\pi} A(\ii u) \sinh(\pi u) ,
		\end{align*}
		whereas from $-3$ times \autoref{eqn:2323and242sum} minus \autoref{eqn:242sum} we have
		\begin{align*}
		G_{\{2\}11}(u) = \frac{1}{2} \frac{\pi u}{\sinh(\pi u)} - \frac{1}{2} \frac{\sinh(\pi u)}{\pi u} + \zeta(2) u^2 \frac{\sinh(\pi u)}{\pi u} + 2u^2 A(\ii u)^2 \frac{\sinh(\pi u)}{\pi u} \,.
		\end{align*}
		Substituting these various results into \autoref{eqn:gz121:identity} we find (after some straightforward simplification) that
		\begin{align*}
		G_{1\{2\}1}(u) &= \Big( G_{\{2\}11}(u) - G_{\{2\}1}(u) G_{1\{2\}}(u) \Big) \cdot G_{\{2\}}(u)^{-1} \\
		&= \frac{1}{2} - \zeta(2) u^2 + 2 u^2 A(\ii u)^2 - 3\zeta(2) \frac{u^2}{\sinh^2(\pi u)} - 4 \ii u^3 A(\ii u) B'(\ii u) \frac{\sinh(\pi u)}{\pi u} \,.
		\end{align*}
		Now substitute this into \autoref{eqn:gt3223:identity}, and after some further straightforward simplification, we find
		\begin{align*}
		G^{t}_{3\{2\}3}(u) = {} & t(2)u^2 - t(4)u^4 - \frac{u^4}{8} A'\Big(\frac{\ii u}{2}\Big)^2 - \frac{u^4}{8} B'\Big(\frac{\ii u}{2}\Big)^2  \\
		& - \frac{u^4}{4} A'\Big(\frac{\ii u}{2}\Big) B'\Big(\frac{\ii u}{2}\Big) \cosh\Big(\frac{\pi u }{2}\Big) - 2t(2)^2 u^4 \csch^2\Big(\frac{\pi u}{2}\Big) \,.
		\end{align*}
		
		It is now a straightforward exercise to extract the evaluation given in the statement of \autoref{thm:t3223:eval} at the start of the section.  And so \autoref{thm:t3223:eval} is proven.  
	\end{proof}
	
	\subsubsection{\texorpdfstring{Evaluation of \( t^\star(3,\{2\}^n,3)\)}
		{Evaluation of t\textasciicircum{}*(3,\{2\}\textasciicircum{}n,3)}}
	\label{sec:tstar3223}
	By applying \( t \circ \Sigma^r \) to \autoref{eqn:stuffleantipode} from the proof of \autoref{thm:dep}, setting \( r = 1\), and recombining the last \( k-1 \) products via \autoref{eqn:stuffleantipode} in reverse, we find that
	\[
	t^\star(3,\{2\}^n,3) = -(-1)^n t(3,\{2\}^n,3) + \sum_{j=0}^n (-1)^j t(\{2\}^j,3) t^\star(\{2\}^{n-j},3).
	\]
	Hence, in terms of generating series, we find
	\[
	G^{t,\star}_{3\{2\}3}(u) = G^{t}_{3\{2\}3}(\ii u ) + \ii G^t_{\{2\}3}(\ii u) G^{t,\star}_{\{2\}3}(u) \,.
	\]
	We have the results that
	\begin{align*}
	G^t_{\{2\}3}(u) &= -\frac{\ii}{2} u^2 A'\Big(\frac{\ii u}{2}\Big) - \frac{\ii}{2} u^2 B'\Big(\frac{\ii u}{2}\Big) \cosh\Big(\frac{\pi u}{2}\Big) \\
	G^{t,\star}_{\{2\}3}(u) &= \ii G^{t}_{3\{2\}}(\ii u) G^{t,\star}_{\{2\}}(u) =  \frac{1}{2} u^2 A'\Big(\frac{u }{2}\Big) - \frac{1}{2} u^2 B'\Big(\frac{\ii u}{2}\Big) \sec\Big(\frac{\pi u}{2}\Big) \,,
	\end{align*}
	the second of which follows from similar consideration, relating \( G^{t,\star}_{\{2\}3}(u), G^t_{3\{2\}}(u) \) and \( G^{t,\star}_{\{2\}}(u) \) or \( G^{t}_{\{2\}}(u) \) via Theorem \ref{thm:dep}; a formula for \( G^{t}_{3\{2\}}(u) \) follows from \( \lim_{u\to0} F^t(u,v) \), while one sees that \( G^{t,\star}_{\{2\}}(u) G^t_{\{2\}}(\ii u) = 1 \) and \( G^t_{\{2\}}(u) = \cosh(\tfrac{\pi u}{2}) \) according to Equation 3.5 in \cite{hoffman19}.  Using these results we obtain the following generating series expression, after some simplification using the fact that \( A' \) and \( B' \) are odd functions.
	\begin{align*}
	G^{t,\star}_{3\{2\}3}(u) = & -t(2) u^2  - t(4) u^4 + \frac{u^4}{8} A'\Big(\frac{u}{2}\Big)^2 + \frac{u^4}{8} B'\Big(\frac{u}{2}\Big)^2 \\
	& \quad\quad + \frac{u^4}{4} A'\Big(\frac{u}{2}\Big) B'\Big(\frac{u}{2}\Big) \sec\Big(\frac{\pi u }{2}\Big) + 2 t(2)^2 u^4 \csc^2\Big(\frac{\pi u}{2}\Big).
	\end{align*}
	
	Recall \( E_n \) is the Euler number, defined via
	\begin{equation}
	\frac{1}{\cosh(z)} = \frac{2}{e^z + e^{-z}} = \sum_{n=0}^\infty \frac{E_n}{n!} z^n \,.
	\label{eqn:eulerno}
	\end{equation}
	It is then, as before, a routine exercise to extract the following explicit evaluation from the above generating series.  
	
	\begin{Thm}[Evaluation of \( t^\star(3,\{2\}^n,3) \)]
		We have following evaluation of \( t^\star(3,\{2\}^n,3) \).
		\begin{align*}
		t^\star(3,\{2\}^n,3) = {} & \Big(\frac{1}{4}\Big)^{n+2} \Big\{ \frac{9+6n}{2} \zeta(2) \zeta(2n+4) \\ 
		& \quad\quad + \sum_{\substack{r+s+q=n+2 \\ r,s,q\geq 1}} 2 r s (2 - 2^{-2r} - 2^{-2s}) \zeta(2r + 1)\zeta(2s+1) \cdot (-1)^q \frac{E_{2q} \pi^{2q}}{(2q)!} \\
		& \quad\quad + \sum_{\substack{r + s = n+2 \\ r,s \geq 1}} 2 r s (2 - 2^{-2r})(2 - 2^{-2s}) \zeta(2r+1) \zeta(2s+1) \Big\} \,,
		\end{align*}
		where \( E_n \) is defined via \autoref{eqn:eulerno}.
	\end{Thm}
	
	\subsection{\texorpdfstring{Evaluations of \( t(1,\{2\}^n,1) \), \( t(1, \{-1\}^n, 1) \), and their \( t^\star \) counterparts}
		{Evaluations of t(1,\{2\}\textasciicircum{}n,1), t(1, \{-1\}\textasciicircum{}n, 1), and their t\textasciicircum{}* counterparts}}
	\label{sec:t1221andt1mm1}
	
	\subsubsection{\texorpdfstring{Evaluation of \( t(1,\{2\}^n,1) \)}
		{Evaluation of t(1,\{2\}\textasciicircum{}n,1)}}\label{sec:t1221}
	
	Extracting the coefficient of \( y_2\cdots y_{n+1} \) in \autoref{eqn:tsymphi0gs:fort3223} (i.e., the Symmetry Theorem \autoref{thm:symgsfull}, with \( \vec{\phi} = \vec{0} \)), we find the following
	\[
	\reg^\ast_{T=\log2} 2t(1, \{2\}^n, 1) = \begin{aligned}[t]
	& \frac{\delta_{n=0}}{2!}\Big(\frac{\ii\pi}{2}\Big)^2 + \sum_{i=0}^n \reg^\ast_{T=\log2} \big( t(\{2\}^i, 1) t(\{2\}^{n-i}, 1) \big) \\
	& - \frac{1}{2^{2n-1}} t(2) \sum_{i=0}^{n-1} \reg^\ast_{T=0} \big( \zeta(\{2\}^i, 1) \zeta(\{2\}^{n-1-i},1) \big) \,.
	\end{aligned}
	\]
	Recall from \autoref{def:Gnotation} that \( G^{t,\bullet}_{\alpha\{\beta\}\gamma}(u) \) is our general notation for the generating function of \emph{stuffle}-regularized M$t$V's of the form \( \reg_{T=\log2}^\ast t^\bullet(\alpha, \{\beta\}^n, \gamma) u^{\abs{\alpha} + n\abs{\beta} + \abs{\gamma}} \), whereas with the zeta generating series \( G_{\alpha\{\beta\}\gamma}(u) \) it was more convenient to utilize the shuffle regularization.  Since there is a single trailing 1, we know (see \cite[Theorem 1]{ikz}) that 
	\[
	\reg_{T=0}^\ast \zeta(\{2\}^n,1) = \reg_{T=0}^\shuffle \zeta(\{2\}^n, 1) \,,
	\]
	so we can evaluate this using the generating series 
	\[
	G_{\{2\}1}(u) = \frac{2}{\pi} A(\ii u) \sinh(\pi u) \,,
	\] from above.   
	Hence in terms of generating series, we have
	\[
	G^t_{1\{2\}1}(u) = -\frac{\pi^2}{8} +  \big( G^t_{\{2\}1}(u)\big)^2 - 2u^2 t(2) \big( G_{\{2\}1}(\tfrac{u}{2}) \big)^2 \,.
	\]
	From \cite[Theorem 3.3]{charltont2212}, we have that
	\begin{equation}\label{eqn:t2212:eval}
	\begin{aligned}
	\sum_{a,b\geq0} (-1)^{a+b} & \reg_{T=\log2} t(\{2\}^a,1,\{2\}^b) \cdot (2x)^{2a} (2y)^{2b} = {} \\
	& \frac{1}{2} \cos(\pi x) (A(x-y) + A(x+y)) 
	+ \frac{1}{2} \cos(\pi y) (B(x-y) + B(x+y) + 2 \log2) \,,
	\end{aligned}
	\end{equation}
	so that by taking \( y \to 0 \) we find
	\begin{align*}
	G^t_{\{2\}1}(u) = u \cosh\Big( \frac{\pi u}{2} \Big) A\Big( \frac{\ii u}{2} \Big) + u B\Big(\frac{\ii u}2\Big)+u\log2 .
	\end{align*}
	Hence we have the following identity for the generating series of \( \reg_{T=\log2} t(1,\{2\}^n,1) \):
	\begin{align*}
	& G_{1\{2\}1}^t(u) =  -\frac{\pi^2 u^2}{16} + \frac{1}{2} \Big(\log2 + A\Big( \frac{\ii u}{2} \Big) + B\Big( \frac{\ii u}{2} \Big) \cosh\Big( \frac{\pi u}{2} \Big)\Big)^2 - \frac{1}{2} A\Big( \frac{\ii u}{2} \Big)^2 \sinh\Big( \frac{\pi u}{2} \Big) \,.
	\end{align*}
	This establishes that \( \reg_{T=\log2} t(1,\{2\}^n,1) \) is a polynomial in single Riemann zeta values and \( \log2 \), and an explicit formula for it can be extracted easily from this generating series.  The general regularization (for \( n > 0 \)) can be recovered as
	\[
	\reg_{T} t(1,\{2\}^n,1) = \reg_{T=\log2} t(1,\{2\}^n) (T-\log2) + \reg_{T=\log2} t(1,\{2\}^n,1) \,,
	\]
	which is now a polynomial in Riemann zeta values, \( \log2 \), and the regularization parameter \( T \), on account of the evaluation
	\[
	\reg_{T=\log2} G_{1\{2\}}^t(u) = u A\Big( \frac{\ii u}{2} \Big) + u \cosh\Big( \frac{\pi u}{2} \Big) \Big( B\Big( \frac{\ii u}{2} \Big) + \log2 \Big) \,,
	\]
	which follows from \autoref{eqn:t2212:eval} by taking \( x \to 0 \).
	
	\subsubsection{\texorpdfstring{Evaluation of \( t^\star(1,\{2\}^n,1) \)}
		{Evaluation of t\textasciicircum{}*(1,\{2\}\textasciicircum{}n,1)}}\label{sec:ts1221}
	Now using the same argument as in \autoref{sec:tstar3223}, we find
	\[
	\reg_{T=\log2} t^\star(1,\{2\}^n,1) = \reg_{T=\log2} \big( -(-1)^n t(1,\{2\}^n,1) + \sum_{j=0}^n (-1)^j t(\{2\}^j,1) t^\star(\{2\}^{n-j},1) \big) \,,
	\]
	from the shuffle-antipode in \autoref{thm:dep}.  We therefore have the generating series identity
	\[
	G^{t,\star}_{1\{2\}1}(u) =  \big( G^{t}_{1\{2\}1}(\ii u ) - \ii G^t_{\{2\}1}(\ii u) G^{t,\star}_{\{2\}1}(u) \big) \,.
	\]
	Likewise
	\begin{align*}
	G^{t,\star}_{\{2\}1}(u) & {} = -  \ii  G^{t}_{1\{2\}}(\ii u) G^{t,\star}_{\{2\}}(u) \\
	& = u B\Big( \frac{u}{2} \Big) + u \log2 + u  A\Big(\frac{u}{2}\Big) \sec\Big(\frac{u\pi}{2}\Big) \,.
	\end{align*}
	So we find
	\begin{align*}
	G^{t,\star}_{1\{2\}1}(u) = {}
	& \frac{\pi^2u^2}{16} + \frac{u^2}{2} \log^2 2 + \frac{u^2}{2} B\Big( \frac{u}{2} \Big)  \Big( B\Big( \frac{u}{2} \Big) + 2 \log2  \Big) + \\
	& {} + \frac{u^2}{2} A\Big( \frac{u}{2} \Big)  \Big( A\Big( \frac{u}{2} \Big)  + 2 B\Big( \frac{u}{2} \Big)  \sec\Big( \frac{u\pi}{2} \Big)  + 2 \log2 \sec\Big( \frac{u\pi}{2} \Big) \Big) \,.
	\end{align*}
	Hence \( \reg_{T=\log2} t^\star(1,\{2\}^n,1) \) is a polynomial in single zeta values and \( \log2 \).  The general regularization can be recovered using the stuffle product of \( t^\star \)-values, giving (for \( n > 0 \))
	\[
	\reg_{T} t^\star(1,\{2\}^n,1) = t^\star(1,\{2\}^n) (T - \log2) + \reg_{T=\log2} t^\star(1,\{2\}^n,1) \,,
	\]
	which is then a polynomial in Riemann zeta values, \( \log2 \), and \( T \).  (That \( t^\star(1,\{2\}^n) \), and generally \( t^\star(\{2\}^a, 1, \{2\}^b) \), is a polynomial in single zeta values and \( \log2 \) follows from the stuffle antipode, as in \autoref{sec:tstar3223}.)
	
	\subsubsection{\texorpdfstring{Evaluation of \( t(1, \{\overline{1}\}^n, 1) \)}
		{Evaluation of t(1, \{-1\}\textasciicircum{}n, 1)}}
	\label{sec:t1mm1}
	
	Recall the notation \( \overline{k} \) means the sign of argument \( k \) is \( \eps = -1 \) in the framework of alternating MZV's and M$t$V's.  From the case \( \vec{\phi} = (0, \{\frac{1}{2}\}^n, 0) \) of \autoref{thm:symgsfull}, we can extract the following identity:
	\begin{equation}\label{eqn:t1mm1}
	\begin{aligned}
	& \reg_{T=\log2} 2t(1,\{\overline{1}\}^{n},1) =  -\delta_{n=0} t(2) + \sum_{i=0}^{n} \reg_{T=\log2} t(\{\overline{1}\}^i, 1) t(\{\overline{1}\}^{n-i},1) \\
	& \quad - \reg_{T=0} \frac{\pi}{2^{n+1}} \Big( (1 - (-1)^n ) \zeta(\{\overline{1}\}^{n}, 1) - \sum_{i=0}^{n-1} (-1)^i \zeta(\{\overline{1}\}^i, 1) \zeta(\{\overline{1}\}^{n-1 - i}, 1) \Big) \,.
	\end{aligned}
	\end{equation}
	From \cite{charltontmm1m}, we know that \( t(\{\overline{1}\}^n, 1) \) is a polynomial in Riemann zeta values, Dirichlet beta values (recall \( \beta(n) \coloneqq \sum_{k=0}^\infty \frac{(-1)^k}{(2k+1)^n} = -t(\overline{n}) \)), and \( \log2 \).  This follows from the more general generating series identity for \( t(\{\overline{1}\}^a, 1, \{\overline{1}\}^b) \) given in Theorem 1.1 \cite{charltontmm1m}.  Setting \( y = 0 \), \( x = -u \) therein leads to the following explicit result verifying this,
	\begin{equation*}
	\begin{aligned}
	G^t_{\{\overline{1}\}1}(u) = {} & \frac{u}{2} \Big( \cos\Big(\frac{\pi u}{4}\Big) - \sin\Big(\frac{\pi u}{4}\Big) \Big)  \Big( 2 A\Big(\frac{u}{2}\Big) + \log2 \Big) \\
	& +  \frac{u}{2} \Big( {-} A\Big(\frac{u}{8}\Big) + A\Big(\frac{u}{4}\Big) + 2  C\Big(\frac{u}{2}\Big) + \log2 \Big) \,,
	\end{aligned}
	\end{equation*}
	where
	\begin{align*}
	C(z) &\coloneqq \frac{1}{8} \big( \psi(\tfrac{1}{4} + \tfrac{z}{4}) - \psi(\tfrac{1}{4} - \tfrac{z}{4}) - \psi(\tfrac{3}{4} + \tfrac{z}{4}) + \psi(\tfrac{3}{4} - \tfrac{z}{4}) \big) = \sum_{r=1}^\infty \beta(2r) z^{2r-1} \,.
	\end{align*}
	
	One further ingredient we need is an evaluation for \( \zeta(\{\overline{1}\}^n, 1) \) as a polynomial in single zeta values and \( \log2 \). 
	
	\begin{Prop}[\( \reg_{T=0} \zeta(\{\overline{1}\}^m, 1) \) evaluation]\label{prop:zetamm1}
		The following regularized generating series evaluation holds.
		\begin{equation}
		\begin{aligned}
		& L(x) \coloneqq \sum_{r=0}^\infty \reg_{T=0} (-1)^r \zeta(\{\overline{1}\}^r, 1) x^r \\
		& {} = \frac{1}{x} - \frac{\Gamma(\tfrac{1}{2})}{\Gamma(1 - \tfrac{x}{2}) \Gamma(\frac{1+x}{2})} \Big( {-} \log2 - 2A(x) + \pi \Big( \cot\Big(\frac{\pi x}{2} \Big) - \cot(\pi x) \Big) \Big) \,,
		\end{aligned}
		\end{equation}
		where
		\[
		A(z) = \psi(1) - \frac{1}{2} (\psi(1+z) + \psi(1-z)) = \sum_{r=1}^\infty \zeta(2r+1) z^{2r} \,,
		\]
		with \( \psi(x) = \frac{\dd }{\dd x} \log\Gamma(x) \) the logarithmic derivative of the gamma function \( \Gamma(x) \).
		\begin{proof}
			We begin by considering the multiple polylog generating series
			\begin{align*}
			K(x;z) &= \sum_{r=0}^\infty (-1)^r \Li_{\{1\}^{r+1}}(\{-1\}^r,z) x^r
			= \sum_{r=0}^\infty \prod_{k < r} \Big( 1 - \frac{(-1)^k x}{k} \Big) \frac{z^r}{r}
			\end{align*}
			Splitting into the odd- and even-indexed terms, and rewriting each via Pochhammer symbols gives
			\begin{align*}
			(r = 2m+2) \quad\quad \prod_{k < r} \Big( 1 - \frac{(-1)^k x}{k} \Big) \frac{z^r}{r} &=  -\frac{1}{x} \cdot \frac{
				\poch{-\tfrac{x}{2}}{m+1}
				\poch{\tfrac{1+x}{2}}{m+1}
			}{
				\poch{\tfrac{1}{2}}{m+1}
			} 
			\frac{z^{2m+2}}{(m+1)!} \,, \\
			(r = 2m+1) \quad\quad \prod_{k < r} \Big( 1 - \frac{(-1)^k x}{k} \Big) \frac{z^r}{r} & = z \cdot \frac{
				\poch{1-\tfrac{x}{2}}{m}
				\poch{\tfrac{1+x}{2}}{m}
			}{
				\poch{\tfrac{3}{2}}{m}
			}
			\frac{z^{2m}}{m!}
			\end{align*}
			Each summation runs from \( m = 0 \), and gives a \( {}_2F_1 \) hypergeometric series (up to an additive constant in the former case), so that
			\[
			K(x;z) = \frac{1}{x} - \frac{1}{x} \cdot  \pFq{2}{1}{-\tfrac{x}{2},\tfrac{1+x}{2}}{\tfrac{1}{2}}{z^2} + z \cdot \pFq{2}{1}{1-\tfrac{x}{2},\tfrac{1+x}{2}}{\tfrac{3}{2}}{z^2} \,.
			\]
			Now notice that
			\begin{align*}
			\Li_{\{1\}^{r+1}}(\{-1\}^r, z) = {} &  \Li_1(z) \Li_{\{1\}^r}(\{-1\}^r) - \sum_{i=0}^{r-1} \Li_{\{1\}^{r+1}}(\{-1\}^i, z, \{-1\}^{r-i}) \\[-1.5ex]
			& - \sum_{i=0}^{r-1} \Li_{\{1\}^{i},2,\{1\}^{r-1-i}}(\{-1\}^i, -z, \{-1\}^{r-1-i}) \,,
			\end{align*}
			so that on rearranging and taking the generating series of both sides, we find
			\begin{align*}
			& \lim_{z\to1^-} \Bigg( \sum_{r=0}^\infty (-1)^r \Li_{\{1\}^{r+1}}(\{-1\}^r, z) x^r + \log(1-z) \bigg( \frac{\Gamma(\tfrac{1}{2})}{\Gamma(1 - \tfrac{x}{2})\Gamma(\tfrac{1+x}{2})} \bigg) \! \Bigg) \\
			& = \sum_{r=0}^\infty (-1)^r \reg_{T=0} \zeta(\{\overline{1}\}^r,1) x^r
			\end{align*}
			as \( z \to 1^- \).  Here we have used both that \( \Li_1(z) = -\log(1-z) \), and that
			\[
			\sum_{r=0}^\infty (-1)^r \zeta(\{\overline{1}\}^r) x^r = \frac{\Gamma(\tfrac{1}{2})}{\Gamma(1 - \tfrac{x}{2})\Gamma(\tfrac{1+x}{2})} \,.
			\]
			(For the latter, see Equation 13 in \cite{kfold} or Equation 12 in \cite{resolution}.)  That is to say
			\begin{align*}
			L(x) = \lim_{z\to1^-} K(x;z) + \log(1-z) \bigg( \frac{\Gamma(\tfrac{1}{2})}{\Gamma(1 - \tfrac{x}{2})\Gamma(\tfrac{1+x}{2})} \bigg)
			\end{align*}
			so once we evaluate the limit, we will find an expression for the desired generating series.
			
			The Ramanujan asymptotic for 0-balanced \( {}_2F_1 \)'s says that
			\[
			\frac{\Gamma(a)\Gamma(b)}{\Gamma(a+b)} \pFq{2}{1}{a,b}{a+b}{z} = -\log(1-z) - 2\gamma - \psi(a) - \psi(b) + O\big((1-z)\log(1-z)\big) \,,
			\]
			as \( z \to 1^- \) (see Corollary 20 \cite{ramanujan}).  We use this to evaluate the limit for \( L(x) \), and after some simplification we find
			\[
			L(x) = \frac{1}{x} - \frac{\Gamma(\tfrac{1}{2})}{\Gamma(1 - \tfrac{x}{2}) \Gamma(\frac{1+x}{2})} \Big( {-}\log2 - 2A(x) + \pi \Big( \cot\Big(\frac{\pi x}{2} \Big) - \cot(\pi x) \Big) \! \Big) \,,
			\]
			as claimed.
		\end{proof}
	\end{Prop}
	
	Along with \autoref{eqn:t1mm1}, this Proposition establishes that \( \reg_{T=\log2} t(1,\{\overline{1}\}^n,1) \) is a polynomial in Riemann zeta values, Dirichlet beta values and \( \log2 \).  More precisely, from \autoref{eqn:t1mm1} we have the following generating series identity
	\begin{align*}
	& G^{t}_{1\{\overline{1}\}1}(u) =   -\frac{1}{2} t(2) u^2 + \frac{1}{2} G^t_{\{\overline{1}\}1}(u)^2 + \frac{\pi u^2}{16} \Big( 2\Big( \! L\Big(\frac{u}{2}\Big) - L\Big({-}\frac{u}{2}\Big) \! \Big) + u L\Big(\frac{u}{2}\Big) L\Big({-}\frac{u}{2}\Big) \! \Big)
	\end{align*}
	The combination involving \( L \) simplifies in a significant way (to the first bracketed term below), so that no gamma functions survive.  In particular, we find
	\begin{align*}
	& G^{t}_{1\{\overline{1}\}1}(u) = \\
	& -\frac{u^2\pi^2}{16} + \Big\{ \frac{\pi u}{4} - \frac{\pi^2 u^2}{16} \cot\Big(\frac{u\pi}{4}\Big) + \frac{u^2}{8}\Big( 2 A\Big( \frac{u}{2} \Big) + \log2 \Big)^2 \sin\Big(\frac{u\pi}{2}\Big) - \frac{\pi^2 x^2}{16} \tan\Big(\frac{u\pi}{4}\Big) \Big\} \\
	& + \frac{u^2}{8} \Big\{ 
	\! \Big( \! \cos\Big(\frac{\pi u}{4}\Big) + \sin\Big(\frac{\pi u}{4}\Big) \! \Big)\!  \Big( 2 A\Big(\frac{u}{2}\Big) + \log2  \Big)
	{-} A\Big(\frac{u}{8}\Big) + A\Big(\frac{u}{4}\Big) - 2  C\Big(\frac{u}{2}\Big) + \log2 \! \Big\}^2 \,.	 
	\end{align*}
	An explicit formula for \( \reg_{T=\log2} t(1,\{\overline{1}\}^n,1) \) can then be extracted from this generating series.  As before, the general regularization \( \reg_{T} t(1, \{\overline{1}\}^n, 1) \) can be recovered using the stuffle product.

	\subsubsection{\texorpdfstring{Evaluation of \( t(1,\{\overline{1}\}^n,1) \)}
		{Evaluation of t\textasciicircum{}*(1,\{-1\}\textasciicircum{}n,1)}}\label{sec:ts1mm1}
	
	Likewise, a corresponding identity for \( \reg_{T} t^\star(1, \{\overline{1}\}^m, 1) \) can be derived using the stuffle antipode as in \autoref{sec:tstar3223}.  Specifically
	\[
	G^{t,\star}_{1\{\overline{1}\}1}(u) =  - G^t_{1\{\overline{1}\}1}(-u) - G^t_{\{\overline{1}\}1}(-u) G^{t,\star}_{\{\overline{1}\}1}(u) \,,
	\]
	where \(  G^{t,\star}_{\{\overline{1}\}1}(u) = - G^t_{1\{\overline{1}\}}(-u) G^{t,\star}_{\{\overline{1}\}}(u) \), and \( G^{t,\star}_{\{\overline{1}\}}(u) G^t_{\{\overline{1}\}}(-u) = 1 \) using similar considerations.  Since 
	\[
	G^t_{\{\overline{1}\}}(u) = \cos\Big(\frac{u\pi}{4}\Big) - \sin\Big( \frac{u\pi}{4}\Big) \,
	\]
	from \cite[Corollary 6.1, proof]{hoffman19}, we can unwind these generating series relations to obtain explicitly
	\begin{align*}
	&  G^{t,\star}_{1\{\overline{1}\}1}(u) = \\
	& \frac{u\pi}{4} + \frac{u^2\pi^2}{16} \Big( \! 1 - 2 \csc\!\Big( \frac{u\pi}{2} \Big) \! \Big) + \frac{u^2}{8} \Big(2 A\Big(\frac{u}{2}\Big) + \log2 \Big)^2 + \frac{u^2}{8} \Big( A\Big(\frac{u}{8}\Big)  - A\Big(\frac{u}{4}\Big)  + 2 C\Big(\frac{u}{2}\Big)  - \log2  \Big)^2 \\
	& - \frac{u^2}{4} \Big(2 A\Big(\frac{u}{2}\Big) + \log2 \Big) \Big( A\Big(\frac{u}{8}\Big)  - A\Big(\frac{u}{4}\Big)  + 2 C\Big(\frac{u}{2}\Big)  - \log2 \Big) \Big(\cos\!\Big(\frac{u\pi}{4}\Big) - \sin\!\Big(\frac{u\pi}{4}\Big) \Big) \sec\!\Big(\frac{u\pi}{2}\Big).
	\end{align*}

	\subsection{\texorpdfstring{Evaluations of \( t^\half(\{1\}^n, 2\ell+2) \) and \( t^\half(2\ell+2, \{1\}^{2n}, 2\ell+2) \)}
		{Evaluations of t\textasciicircum{}½(\{1\}\textasciicircum{}n, 2l+2) and t\textasciicircum{}½(2l+2, \{1\}\textasciicircum{}2n, 2l+2)}}
	\label{sec:th111ev}
	
	From \autoref{thm:dep}, in the case \( r = \frac{1}{2} \) we observe the following:
	\begin{align*}
	\reg_{T=\log2} t^\half(\{1\}^n, 2\ell+2) & \overset{\mathrm{stuffle}}{=} \reg_{T=\log2} (-1)^n \, t^\half(2\ell+2,\{1\}^n)  \pmod{\mathrm{products}} \,.
	\end{align*}
	On the other hand, from \autoref{cor:symtr}, we have
	\begin{align*}
	\reg_{T=\log2} t^\half(\{1\}^n, 2\ell+2) & \overset{\mathrm{sym}}{=} \reg_{T=\log2} (-1)^{n+1} t^\half(2\ell+2,\{1\}^n) \pmod{\mathrm{products}} \,.
	\end{align*}
	Because of the opposite signs in each case, both of these \( t^\half \)-values must be reducible individually.  In fact, an argument based on the hypothetical `derivation with respect to \( \log2 \)' \cite[Conjecture 2.1 and thereafter]{hoffman19} (formalized somewhat in \cite[Remark 5.9]{charltont2212}) suggests that \( t^\half(\{1\}^n, 2) \) should not contain any terms of the form \( \log^k 2 \cdot u \), \( k < n \), with \( u \) indecomposable.  This does not preclude terms of the form \( \log^k 2 \cdot u \cdot v \), but does force \( t^\half(\{1\}^n, 2\ell+2) \) to be especially simple. \medskip
	
	In fact, from \autoref{thm:symgsfull} and \autoref{thm:dep} we can extract suitable generating series identities which give the following evaluation of \( t^\half(\{1\}^n, 2\ell+2) \) as a polynomial in \( \log2 \) and Riemann zeta values.
	\begin{Thm}\label{thm:thalf111ev}
		The following generating series identity holds:
		\begin{equation}\label{eqn:thalf111ev}
		R(u, \lambda) \coloneqq \sum_{n,\ell\geq0} t^\half(\{1\}^n, 2\ell+2) u^n \lambda^{2\ell+2} = \frac{\lambda^2 \pi^2 \sec\big(\frac{\lambda \pi}{2} \big) \sec\big(\frac{\pi u}{4}\big)}{8 \Gamma\big(1 - \frac{\lambda}{2} - \frac{u}{4}\big)\Gamma\big(1 + \frac{\lambda}{2} - \frac{u}{4}\big)}\frac{\Gamma(\frac{1}{2} - \frac{u}{4}\big)}{\Gamma(\frac{1}{2} + \frac{u}{4}\big)} \,.
		\end{equation}
		In particular, \( t^\half(\{1\}^n, 2\ell+2) \) is always a polynomial in \( \log{2} \), and Riemann zeta values.
	\end{Thm}
	
	We postpone the proof until the end of the section; we immediately have the following corollaries.
	
	\begin{Cor}
		The following M$t$V 
		\[
		t^\half(2\ell+2, \{1\}^{2n}, 2\ell+2)
		\]
		is always a polynomial in \( \log{2} \) and Riemann zeta values.
	\end{Cor}
	
	{\noindent Again recall from \autoref{def:Gnotation} that \( G^{t,\bullet}_{\alpha\{\beta\}\gamma}(u) \) is the general notation for the generating function of \emph{stuffle}-regularized M$t$V's of the form \( \reg_{T=\log2}^\ast t^\bullet(\alpha, \{\beta\}^n, \gamma) u^{\abs{\alpha} + n\abs{\beta} + \abs{\gamma}} \).}
	
	\begin{proof}
		This follows immediately with the stuffle antipode \autoref{thm:dep}, since \( t^\half(\{1\}^i, 2\ell+2) \) is always a polynomial in \(\log{2} \) and Riemann zeta values.  In particular
		\[
		G^{t,\half}_{2\ell+2,\{1\},2\ell+2}(u) + 	G^{t,\half}_{2\ell+2,\{1\},2\ell+2}(-u) = 	G^{t,\half}_{\{1\},2\ell+2}(u) G^{t,\half}_{\{1\},2\ell+2}(-u) \,.  \qedhere
		\]
	\end{proof}
	
	A particularly interesting case occurs for \( \ell = 0 \), wherein \( t^\half(2,\{1\}^n,2) \) appears to be a rational multiple of \( t(4+n) \) in each weight.  From the previous Corollary we have the following special case.
	
	\begin{Cor}\label{cor:thalf2112}
		The following evaluation holds:
		\[
		t^\half(2,\{1\}^{2n},2) = \frac{3+2n}{2^{2+2n}} t(4+2n) \,.
		\]
		
		\begin{proof}
			We have (after some straightforward simplification) that
			\begin{equation}\label{eqn:gsthalf1112}
			\begin{aligned}
			G^{t,\half}_{\{1\}2}(u) & {} \coloneqq \sum_{i=0}^\infty t^{\half}(\{1\}^i, 2)u^{i+2} \\
			& = \frac{u^2}{2!} \frac{\partial^2}{\partial \lambda^2}\bigg\rvert_{\lambda=0} R(u,\lambda) = 
			e^{u \log2} \frac{\Gamma\big(1 - \frac{u}{2}\big)\Gamma\big(1+\frac{u}{4}\big)^2}{\Gamma\big(1 +\frac{u}{2}\big)\Gamma\big(1-\frac{u}{4}\big)^2} \cdot \frac{\pi u}{2}  \tan\Big(\frac{\pi u}{4}\Big) \,.
			\end{aligned}
			\end{equation}
			From the stuffle antipode in \autoref{thm:dep}, we have the generating series identity
			\[
			G^{t,\half}_{2\{1\}2}(u) + G^{t,\half}_{2\{1\}2}(-u) =  G^{t,\half}_{\{1\}2}(u) G^{t,\half}_{\{1\}2}(-u) \,.
			\]
			Using \autoref{eqn:gsthalf1112}, we find
			\[
			G^{t,\half}_{2\{1\}2}(u) + G^{t,\half}_{2\{1\}2}(-u) = \frac{\pi^2 y^2}{4} \tan^2\Big(\frac{\pi y}{4}\Big) \,
			\]
			which is equivalent to the claimed evaluation.
		\end{proof}
	\end{Cor}
	
	This evaluation only holds for even weight; in odd weight \( t^\half(2,\{1\}^{2n+1},2) \) appears to evaluate in a similar straightforward way. Since \( t(2n+1) \) is (conjecturally) irreducible the Symmetry Theorem likely cannot be applied to establish the evaluation, and some new technique will be necessary.  We therefore leave the odd weight case as the following conjecture.
	
	\begin{Conj}\label{conj:thalf21od2}
		The following evaluation holds:
		\[
		\		t^\half(2,\{1\}^{2n+1},2) = \frac{4+2n}{2^{3+2n}} t(5+2n) \,.
		\]
	\end{Conj}
	
	\begin{Rem}
		We have checked this conjecture using the Data Mine \cite{datamine} up to weight 11 (the relevant limit of the Data Mine for alternating MZV's).  Thereafter, we have also verified numerically to 1000 decimal places in weights 13 and 15.
		
		The reader using the tables in \cite[Appendix A]{hoffman19} to check this conjecture in weight 7 should be aware of a misprint.  The formula for \( t(2,1,2,2) \) there in \cite{hoffman19} (corresponding to \( t(2,2,1,2) \) in the convention of the present paper) should read
		\begin{align*}
		t(2,1,2,2) =  -\frac{15}{32} t(7) - \frac{1}{14} t(3)t(4)+ \frac{111}{248} t(2)t(5) \,;
		\end{align*}
		the formula as printed has an incorrect coefficient for \( t(2)t(5) \).
	\end{Rem}
	
	\begin{Rem}
		A two-variable generating series expression for \( t^\half(2\ell+2, \{1\}^{2n}, 2\ell+2) \) is not as straightforward to find.  It requires taking the Hadamard (coefficient-wise) product of \( R(u, \lambda) \) with \( R(-u, \lambda) \), viewed as power series in \( \lambda \).  (This Hadamard product can be given implicitly through an integral representation.)  However, by repeated differentiation, one can extract any particular \( G^{t,\half}_{\{1\},2\ell+2}(u) \) and the corresponding series for \( t^\half(2\ell+2,\{1\}^{2n},2\ell+2) \).
		
		For example,
		\begin{align*}
		G^{t,\half}_{\{1\}4}(u) & {} \coloneqq \sum_{i=0}^\infty t^{\half}(\{1\}^i, 4)u^{i+4} \\
		& = \frac{u^4}{4!} \frac{\partial^4}{\partial \lambda^4}\bigg\rvert_{\lambda=0} R(u,\lambda) = \frac{\pi ^2 u^4 \sec \big(\frac{\pi  u}{4}\big)}{64 \Gamma \big(1-\frac{u}{4}\big)^2 } \frac{\Gamma \big(\frac{1}{2}-\frac{u}{4}\big)}{\Gamma \big(\frac{1}{2}+\frac{u}{4}\big)} \Big(\pi ^2-2 \psi ^{(1)}\Big(1-\frac{u}{4}\Big)\!\Big) \,,
		\end{align*}
		where \( \psi^{(n)}(z) = \frac{\dd^{n+1}}{\dd z^{n+1}} \log\Gamma(z) \) is the order \( n \) polygamma function.  Therefore
		\begin{align*}
		\sum_{i=0}^\infty t^\half(4, \{1\}^{2i}, 4) u^{8+2i} &= \frac{1}{2} G^{t,\half}_{\{1\}4}(u)G^{t,\half}_{\{1\}4}(-u) \\
		&=  \frac{\pi ^2 u^6 }{512} \tan ^2\Big(\frac{\pi  u}{4}\Big) \Big(\pi ^2-2 \psi ^{(1)}\Big(1-\frac{u}{4}\Big)\!\Big) \Big(\pi ^2-2 \psi
		^{(1)}\Big(1+\frac{u}{4}\Big)\!\Big) \,.
		\end{align*}
	\end{Rem}
	
	\begin{Rem}
		In contrast to the case \( t^\half(2, \{1\}^n, 2) \) in \autoref{conj:thalf21od2}, where a similarly simple evaluation seems to hold in odd weight, no such simple evaluation holds seems to hold for the higher \( t^\half(2\ell+2, \{1\}^\text{odd}, 2\ell+2) \) analogues, viz:
		\[
		t^\half(4, 1, 1, 1, 4)  	
		\]
		involves \( \zeta(1, 1, \overline{9}) \) and 4 other irreducible alternating MZV's of weight 11.
		
		Similarly, \( t^\half(\{1\}^i, 2\ell+1) \), with an odd final argument \( >1 \), does not appear to evaluate nicely, as already
		\[
		t^\half(1,3) = -\frac{1}{2} \zeta(1, \overline{3})
		\]
		involves an irreducible weight 4 alternating MZV.  Likewise in weight 5 with \( t^\half(1,1,3) \) we already involve \( \zeta(1, 1, \overline{3}) \).
	\end{Rem}
	
	\begin{proof}[Proof of \autoref{thm:thalf111ev}]
		For simplicity, let us write \( \Tic_{T=\log2}(\vec{0} \mid y_1,\ldots,y_m) \eqqcolon \Tic_{T=\log2}(y_1,\ldots,y_m) \), and likewise for \( \Lic_{T=0}(y_1,\ldots,y_m) \), as the first tuple will here always be the zero tuple.  We note then that
		\begin{align*}
		&\Tic_{T=\log2}(\{ur\}^d) - \sum_{n=0}^{k-2} \frac{(ru)^n}{n!} \frac{\partial^n}{\partial W^n} \bigg\rvert_{W=0} \Tic_{T=\log2}(\{ur\}^{d-1}, W) \\
		& = \sum_{n=d}^\infty \sum_{\substack{I= (i_1,\ldots,i_d) \\ \abs{I} = n \\ i_d \geq k}} t(i_1,\ldots,i_{d-1},\underbrace{i_d}_{\geq k}) (ru)^{n - d} \,.
		\end{align*}
		The sum of derivatives step-by-step eliminates those indices \( (i_1,\ldots,i_{d-1},i_{d}) \) which end with \( i_d = 1, 2, \ldots, k-1 \).  (Since \( k \geq 2 \), we do not need to explicitly regularize above.)
		
		We therefore want to evaluate the following (for \( k \) even, and already specializing to \( r = 1/2 \)):
		\begin{equation}
		G^{t,\half}_{\{1\}k}(u) = \sum_{d=1}^\infty \frac{u^d}{2^{1-k}} \Big( \Tic_{T=\log2}(\{\tfrac{u}{2}\}^{d}) - \sum_{n=0}^{k-2} \frac{u^n}{2^n \cdot n!}  \frac{\partial^n}{\partial W^n} \bigg\rvert_{W=0} \Tic_{T=\log2}(\{\tfrac{u}{2}\}^{d-1}, W) \Big).
		\end{equation}
		Take the following generating series with respect to \( k \), and we find after switching the order of summation and (formally) summing the infinite series of differential operators, that
		\begin{equation}
		\label{eqn:gthalf_11k}
		\begin{aligned}[c]
		R(u, \lambda) & = \sum_{\substack{k = 2 \\ \text{\( k \) even}}} G^{t,\half}_{\{1\}k}(u) \Big(\frac{\lambda}{u} \Big)^k \\
		& =  \sum_{d=1}^\infty \frac{2 \lambda^2 u^d}{4 \lambda^2 - u^2} \Big( {-}\Tic_{T=\log2}(\{\tfrac{u}{2}\}^d) + \mathcal{D} \Tic_{T=\log2}(\{\tfrac{u}{2}\}^{d-1}, W) \Big) \,,
		\end{aligned}
		\end{equation}
		where \( \mathcal{D} \) denotes the (formal) differential operator \[
		\mathcal{D} \coloneqq \cosh\big(\lambda \tfrac{\partial}{\partial W}\big\rvert_{W=0}\big) + \tfrac{2\lambda}{u} \sinh\big(\tfrac{\partial}{\partial W}\big\rvert_{W=0}\big) \,.
		\]
		We interpret and understand the application of an infinite series of differentials \( \tfrac{\partial}{\partial W} \) to a function \( f \) via the Taylor series of \( f \).  Specifically \( \exp(\lambda \tfrac{\partial}{\partial W}\big\rvert_{W=0}) f(W) = f(\lambda) \), via its Taylor series (assuming convergence at the relevant points), which we then then extend (formally) to the hyperbolic trigonometric combination above.
		
		Likewise,
		\begin{equation}
		\label{eqn:gthalf_k11}
		\begin{aligned}[c]
		S(u,\lambda) & {} \coloneqq \sum_{\substack{k = 2 \\ \text{\( k \) even}}} G^{t,\half}_{k\{1\}}(u) \Big(\frac{\lambda}{u} \Big)^k \\
		& {} =  \sum_{d=1}^\infty \frac{2 \lambda^2 u^d}{4 \lambda^2 - u^2} \Big( {-}\Tic_{T=\log2}(\{\tfrac{u}{2}\}^d) + \mathcal{D} \Tic_{T=\log2}(W,\{\tfrac{u}{2}\}^{d-1}) \Big) \,.
		\end{aligned}
		\end{equation}
		
		From the stuffle antipode \autoref{thm:dep}, we find
		\begin{align}\label{eqn:r111minuss}
		R(u,\lambda) \cdot G^{t,\half}_{\{1\}}(-u) - S(-u, \lambda) = 0 \,.
		\end{align}
		Now write down the case \( \vec{y} = (\{\tfrac{u}{2}\}^{d-1}, W) \) of the Symmetry Theorem \autoref{thm:symgsfull} to obtain
		\begin{align*}
		& \sum_{i=0}^{d-1} (-1)^{d-1-i} \Tic_{T=\log2}(\{\tfrac{u}{2}\}^i, W) \Tic_{T=\log2}(\{-\tfrac{u}{2}\}^{d-1-i}) - (-1)^{d-1} \Tic_{T=\log2}(-W, \{\tfrac{u}{2}\}^{d-1}) \\
		& -\frac{1}{2^{d-1}} \sum_{i=0}^{d-2} (-1)^i \big( \Tic_{T=\log2}(\tfrac{u}{2}) - \Tic_{T=\log2}(-\tfrac{u}{2}) \big) \Lic_{T=0}(\{0\}^i) \Li_{T=0}(\{0\}^{d-2-i}, \tfrac{W-u/2}{2}) \\
		& - \big({-}\tfrac{1}{2} \big)^{d-1} \big(\Tic_{T=\log2}(W) - \Tic_{T=\log2}(-W) \big) \Lic_{T=0}\big( \{\tfrac{W-u/2}{2}\}^{d-1}\big) 
		= \delta_{\text{\( d \) even}} \frac{1}{d!} \Big( \frac{\ii\pi}{2}  \Big)^d \,.
		\end{align*}
		After taking \( \sum_{d=1}^\infty \bullet  u^d \), using the following results, we find that
		\begin{equation}\label{eqn:symgs:111ev}
		\begin{aligned}
		& \sum_{i=0}^\infty \Tic_{T=\log2}(\{\tfrac{u}{2}\}, W) u^{i+1} \cdot  e^{\gamma u / 2} \frac{\Gamma\big( \frac{1}{2} + \frac{u}{4} \big)}{\Gamma\big( \frac{1}{2} + \frac{u}{4} \big)} + \sum_{i=0}^\infty (-1)^i \Tic_{T=\log2}(-W, \{-\tfrac{u}{2}\}) u^{i+1} \\
		& = \frac{\pi u}{2} e^{\gamma u/2} \frac{\Gamma\big( 1 + \frac{u}{4} - \frac{W}{2} \big)}{\Gamma(1 - \frac{u}{4} - \frac{W}{2} \big)}  \bigg( \! \tan\!\bigg(\frac{\pi u}{4} \bigg) - \tan\!\Big( \frac{\pi W}{2}\bigg) \! \bigg) \,.
		\end{aligned}
		\end{equation} 
		This requires the following results obtainable via standard arguments and evaluations, as we indicate below.
		\begin{align}
		& \label{eqn:Tdepth1} \Tic_{T=\log2}(y) - \Tic_{T=\log2}(-y) = \frac{\pi}{2} \tan\Big(\frac{\pi y}{2}\Big) \,, \\
		& \label{eqn:Tuuu} \sum_{i=0}^\infty \Tic_{T=\log2}(\{\tfrac{u}{2}\}^i) (-u)^{i} = G^{t,\half}_{\{1\}}(-u)  = e^{\gamma u / 2} \frac{\Gamma\big( \frac{1}{2} + \frac{u}{4} \big)}{\Gamma\big( \frac{1}{2} - \frac{u}{4} \big)} \,, \\
		& \label{eqn:T000} \sum_{i=0}^\infty \Lic_{T=0}(\{0\}^i) y^i = \frac{e^{-\gamma y}}{\Gamma(1+y)} \,, \\
		& \label{eqn:T00X} \sum_{i=1}^\infty \Lic_{T=0}(\{0\}^{i-1}, x) y^i = \frac{e^{-\gamma y}}{\Gamma(1+y)} - \frac{\Gamma(1-x)\Gamma(1-y)}{\Gamma(1-x-y)}  \,.
		\end{align}
		Firstly, \autoref{eqn:Tdepth1} comes from just explicitly evaluating the generating series, using known formulas for \( t(2n) \) in terms of \( \zeta(2n) \).  Secondly, \autoref{eqn:Tuuu} requires recognizing the sum as an expression for \( G^{t,\half}_{\{1\}}(-u) \), which is then evaluated using the results of Section 6.2 \cite{hoffman-ihara}, and a variant of Equation (40) therein.  Then \autoref{eqn:T000} involves the generating series for \( \zeta(\{n\}^k) \) \cite[Equation 11]{kfold}, or rather the extension to the Hopf algebra \cite[Equation 32]{hoffman-ihara} in the case \( n = 1 \), and well-known Taylor series \( \log\Gamma(1+z) \) already mentioned (a variant given in \autoref{eqn:taylorLogGamma}).  Finally \autoref{eqn:T00X} involves the generating series for \( \zeta(\{1\}^j, d+1) \), \( d \geq 1 \), given in \cite[Equation 10]{kfold}, and the result in \autoref{eqn:T000}.
		
		Now write
		\(
		E(t) \coloneqq \sum_{j=0}^\infty \reg_{T=0} \zeta(\{1\}^j) t^j = \frac{e^{-\gamma t}}{\Gamma(1 + t)} 
		\).
		Using \cite[Lemma 1]{hoffman17} (after applying the homomorphism taking elementary symmetric functions to MZV's), we find
		\begin{align*}
		\sum_{i=0}^\infty \Lic_{T=0}(\{x\}^i) y^i &= 1 + \sum_{n \geq k \geq1} \sum_{\substack{I = (i_1, \ldots, i_k) \\ \abs{I} = n}} \reg_{T=0} \zeta(i_1,\ldots,i_n) x^{n-k} y^k \\
		&= E\Big(\!\Big(\frac{y}{x} - 1\Big) x\Big) E(-x)^{-1} = \frac{e^{\gamma y}\Gamma(1 -x )}{\Gamma(1 - x + y)} \,.
		\end{align*}
		These results suffice to obtain \autoref{eqn:symgs:111ev}. \medskip
		
		Now apply \( \mathcal{D} \) to both sides of \autoref{eqn:symgs:111ev}.  The left-hand side can be rewritten using \autoref{eqn:gthalf_11k} and \autoref{eqn:gthalf_k11}; the right-hand side is evaluated via the formal interpretation of \( \exp(\lambda \tfrac{\partial}{\partial W}\big\rvert_{W=0}) f(W) = f(\lambda) \) as discussed above.  We obtain after some simplification
		\begin{align}\label{eqn:r111pluss}
		R(u,\lambda) \cdot G^{t,\half}_{\{1\}}(-u) + S(-u, \lambda) = \frac{4 \lambda^2 \pi^2}{u^2 -4 \lambda^2} e^{\gamma u  /2} \frac{\sec\big(\frac{\lambda \pi}{2}\big) \sec\big(\frac{\pi u}{4}\big)}{\Gamma\big(-\frac{\lambda}{2} - \frac{u}{4}\big)\Gamma\big(\frac{\lambda}{2} - \frac{u}{4}\big)} \,.
		\end{align}
		Solving \autoref{eqn:r111minuss} and \autoref{eqn:r111pluss} simultaneously leads to the claimed result for \( R(u,\lambda) \) and, as a  by-product, a formula for \( S(u,\lambda) \), namely
		\begin{align*}
		R(u, \lambda) \coloneqq \sum_{n,\ell\geq0} t^\half(\{1\}^n, 2\ell+2) u^n \lambda^{2\ell+2} & = \frac{\lambda^2 \pi^2 \sec\big(\frac{\lambda \pi}{2} \big) \sec\big(\frac{\pi u}{4}\big)}{8 \Gamma\big(1 - \frac{\lambda}{2} - \frac{u}{4}\big)\Gamma\big(1 + \frac{\lambda}{2} - \frac{u}{4}\big)}\frac{\Gamma(\frac{1}{2} - \frac{u}{4}\big)}{\Gamma(\frac{1}{2} + \frac{u}{4}\big)} \,,	\\
		S(u, \lambda) = \sum_{n,\ell\geq0} \reg_{T=\log2} t^\half(2\ell+2,\{1\}^n) u^n\lambda^{2\ell+2} & = e^{-\gamma u/2} \frac{\lambda^2 \pi^2 \sec\big(\frac{\lambda \pi}{2} \big) \sec\big(\frac{\pi u}{4}\big)}{8 \Gamma\big(1 - \frac{\lambda}{2} + \frac{u}{4}\big)\Gamma\big(1 + \frac{\lambda}{2} + \frac{u}{4}\big)} \,.
		\end{align*}
		As before, \( S(u,\lambda) \) for the general regularization parameter \( T \) can also be recovered, via the stuffle antipode \autoref{thm:dep}, as \( e^{(T-\log{2})u} S(u,\lambda) \).  This completes the proof.
	\end{proof}

	\bibliographystyle{habbrv}
	\bibliography{bib}
	
	\appendix
	\section{Tails of multiple zeta values, and an analytic result}
	\label{sec:appendix}
	
	In this Appendix, we gather some results on the growth and convergence rate of the truncated MZV's, and a useful analytic result  \autoref{prop:specialsumconvergence} which we need when taking the limit as \( M \to \infty \) in the truncated generating series identity \autoref{prop:truncgs}.  Although the following is probably well known in the literature, we include it for the sake of completeness.
	
	\begin{Def}[Tails of multiple zeta values]
		Let \( \eps_1,\ldots,\eps_m \in \C \), with \( \abs{\eps_i} = 1 \), and \( n_1,\ldots,n_m \in \Z_{>0} \) be given.  Let \( M \in \mathbb{Z}_{>0} \),  we then define the \emph{\( > \)-tail \( \zeta_{>M} \)} and the \emph{\( \gg \)-tail \( \zeta_{\gg M} \)} of an MZV as follows,
		\begin{align*}
		\zeta_{>M}\sgnarg{\eps_1,\ldots,\eps_m}{n_1,\ldots,n_m} &\coloneqq \sum_{\substack{k_1 < \ldots < k_m \\ k_m > M}} \frac{\eps_1^{k_1} \cdots \eps_m^{k_m}}{k_1^{n_1} \cdots k_m^{n_m}} \,, \quad 
		\zeta_{\gg M}\sgnarg{\eps_1,\ldots,\eps_m}{n_1,\ldots,n_m} \coloneqq \sum_{\substack{M < k_1 < \ldots < k_m}} \frac{\eps_1^{k_1} \cdots \eps_m^{k_m}}{k_1^{n_1} \cdots k_m^{n_m}} \,.
		\end{align*}
	\end{Def}
	Note that (in the convergent case)
	\[
	\zeta\sgnarg{\eps_1,\ldots,\eps_m}{n_1,\ldots,n_m} - \zeta_M\sgnarg{\eps_1,\ldots,\eps_m}{n_1,\ldots,n_m} = \zeta_{>M}\sgnarg{\eps_1,\ldots,\eps_m}{n_1,\ldots,n_m} \,,
	\]
	so \( \zeta_{>M} \) represents the usual tail of the series.  The behaviour of \( \zeta_{>M} \) as \( M \to \infty \) tells us the rate of convergence of \( \zeta_M \) to \( \zeta \).
	
	\begin{Prop}\label{prop:zetatail}
		Let \( \eps_1,\ldots,\eps_m \in \C \), with \( \abs{\eps_i} = 1 \), and let \( n_1,\ldots,n_m \in \Z_{>0} \) with \( (\eps_m, n_m) \neq (1,1) \) be given.  Then there exist \( J, J' \geq 0 \in \R \) such that the following asymptotics hold as \( M \to \infty \),
		\begin{align*}
		\zeta_{>M}\sgnarg{\eps_1,\ldots,\eps_m}{n_1,\ldots,n_m} &= O\bigg( \frac{\log^J M}{M} \bigg) \,, \quad 
		\zeta_{\gg M}\sgnarg{\eps_1,\ldots,\eps_m}{n_1,\ldots,n_m} = O\bigg( \frac{\log^{J'} M}{M} \bigg)
		\end{align*}
		
		\begin{proof}
			We treat first the depth 1 cases \(	\zeta_{>M}\sgnarg{\eps}{n} = \zeta_{\gg M}\sgnarg{\eps}{n} \), for any \( (\eps,n) \neq (1,1) \).
			
			\paragraph{{\em Case \( (\eps, 1) \), with \( \eps \neq 1 \)}:}  We have that
			\begin{align*}
			\sum_{k=M+1}^\infty \frac{\eps^k}{k} &= \eps^{-1} \int_0^1 \sum_{k=M+1}^\infty (\eps t)^{k-1} \dd t = \eps^{-1} \int_{t=0}^1 \frac{(\eps t)^{M}}{1 - \eps t} \dd t
			\end{align*}
			Since \( \eps \neq 1 \), the denominator \( 1 - (\eps t) \) never attains the value 0 on the interval \( [0,1] \).  Let \( C = \big( \min\{\abs{1 - \eps t} \mid t \in [0,1] \} \big)^{-1} > 0 \), then we estimate
			\begin{align*}
			\abs{\zeta_{>M}\sgnarg{\eps}{1}} & 
			\leq \int_{t=0}^1 \abs{\frac{(\eps t)^M}{1 - (\eps t)}} \dd t 
			= C \int_0^1 t^M \dd t 
			= \frac{C}{M+1} \,.
			\end{align*}
			
			\begin{Rem}  We can estimate \( C^{-1} \) as 
				\[
				C^{-1} = \begin{cases}
				\abs{\sin(\arg\eps)} & \arg\eps \in (0, \pi/2) \cup (3\pi/2, 2\pi) \\
				1 & \arg\eps \in [\pi/2, 3\pi/2] \,.
				\end{cases}
				\]
				This follows by computing the turning point, where \( \eps = \exp(i \theta) \), of
				\[
				\abs{1 - \eps t} = \sqrt{(1 - \cos(\theta) t)^2 + \sin(\theta) t} \,,
				\]
				which occurs at \( t = \cos(\theta) \).  For \( \theta = \arg\eps \) in the range \( [\pi/2, 3\pi/2] \), cosine is negative, so this turning point does not occur in the interval \( t \in [0,1] \).  One sees then that the minimal value occurs at the left end point \( t = 0 \), which gives value \( C^{-1} = 1 \).  Otherwise, \( C^{-1} \) is given by \( \abs{1 - \eps \cos(\theta)} = \abs{\sin(\theta)} \). \end{Rem}
			
			\medskip
			\paragraph{{\em Case \( (\eps, n) \), with \( n \geq 2 \)}:}  In this case, we know the series \( \zeta\sgnargsm{\eps}{n} \) converges absolutely, so we find that \( n-1\geq1\), and
			\[
			\abs{\sum_{k>M} \frac{\eps^k}{k^n} } \leq \sum_{k=M+1}^\infty \frac{1}{k^n} \leq \int_{t=M}^\infty \frac{\dd {t}}{t^n}  = \frac{-t^{-n+1}}{n-1} \bigg\lvert_{t=M}^\infty = \frac{1}{n-1} \frac{1}{M^{n-1}} = O(M^{-1}) \,.
			\]
			This establishes the depth 1 result for \( \zeta_{>M}\sgnargsm{\eps}{n} = \zeta_{\gg M}\sgnargsm{\eps}{n} \), whenever \( (\eps, n) \neq (1,1) \).
			
			\medskip
			\noindent Now we inductively show that same result for higher depth.\medskip
			
			\paragraph{{\em Case higher depth, \( \zeta_{\gg M} \)}:}	Firstly consider the \( \zeta_{\gg M} \) case, wherein
			\begin{align*}
			\zeta_{\gg M}\sgnarg{\eps_1, \ldots, \eps_m}{n_1, \ldots, n_m} &= \sum_{\substack{M < k_1 < \ldots < k_m}} \frac{\eps_1^{k_1} \cdots \eps_m^{k_m}}{k_1^{n_1} \cdots k_m^{n_m}}  = \sum_{\substack{k_1 = M + 1}}^\infty \frac{\eps_1^{k_1}}{k_1^{n_1}} \zeta_{\gg k_1}\sgnarg{\eps_2,\ldots\,\eps_m}{n_2,\ldots,n_m}
			\end{align*}
			So we estimate
			\[
			\abs{\zeta_{\gg M}\sgnarg{\eps_1, \ldots, \eps_m}{n_1, \ldots, n_m}} \leq \sum_{\substack{k_1 = M+1}}^\infty \frac{1}{k_1^{n_1}} \abs{\zeta_{\gg k_1}\sgnarg{\eps_2,\ldots\,\eps_m}{n_2,\ldots,n_m}} \,,
			\]
			and can apply the induction assumption to obtain that for some \( J \), and \( C' \), this is 
			\[
			\leq C' \sum_{k_1 = M+1}^\infty \frac{1}{k_1^{n_1}} \frac{\log^J k_1}{k_1} \,.
			\]
			With the integral test, we further estimate that this is 
			\[
			\leq C' \int_{M}^\infty \frac{\log^J x}{x^{k+1}} \dd x \,.
			\]
			By partial integration, one sees that
			\[
			\int \frac{\log(x)^J}{x^{k+1}} \dd x = \frac{P_J(\log(x))}{x^k} \,,
			\]
			for some polynomial \( P_J \) of degree \( J \).  So the estimate on \( \zeta_{\gg M} \) is
			\[
			\leq C' \frac{P_J(\log(M))}{M^k} = O\bigg( \frac{\log^J M}{M} \bigg) \,,
			\]
			proving the \( \zeta_{\gg M} \) case.
			\medskip
			\paragraph{{\em Case higher depth, \( \zeta_{>M} \)}:} Now consider the \( \zeta_{>M} \) case, wherein
			\[
			\zeta_{>M}\sgnarg{\eps_1, \ldots, \eps_m}{n_1, \ldots, n_m} = \sum_{\substack{k_1 < \ldots < k_m \\ k_m > M}} \frac{\eps_1^{k_1} \cdots \eps_m^{k_m}}{k_1^{n_1} \cdots k_m^{n_m}} \,.
			\]
			By inserting the condition \( k_i \leq M < k_{i+1} \) in all possible compatible ways throughout the summation index, we see
			\begin{align*}
			\zeta_{>M}\sgnarg{\eps_1, \ldots, \eps_m}{n_1, \ldots,  n_m} &= \sum_{i=1}^m \sum_{\substack{0 < k_1 < \ldots k_{i-1} \leq M \\
					< k_i < \ldots < k_m}} \frac{\eps_1^{k_1} \cdots \eps_m^{k_m}}{k_1^{n_1} \cdots k_m^{n_m}} \\
			&= \sum_{i=1}^m \zeta_M\sgnarg{\eps_1,\ldots,\eps_{i-1}}{n_1,\ldots,n_{i-1}} \zeta_{\gg M}\sgnarg{\eps_{i}, \ldots,\eps_m}{n_i,\ldots,n_m} \\
			&= \sum_{i=1}^m \bigg( Q_i(\log(M)) + O\bigg( \frac{\log^{J_i} M}{M} \bigg) \bigg) \cdot O\bigg( \frac{\log^{{J_i}'} M}{M} \bigg)
			\end{align*}
			for some polynomials \( Q_i \), using standard results about the growth of truncated MZVs, and the induction assumption in the case \( \zeta_{\gg M} \).  One then sees this sum has order
			\[
			= O\bigg( \frac{\log^{J''}{M}}{M} \bigg)
			\]
			for some \( J'' \), as claimed. This completes the proof of \( \zeta_{>M} \) case, and the proposition.
		\end{proof}
	\end{Prop}
	
	We need the following result about the convergence of infinite sums of a particular form, in order to pass from the generating series identity involving truncated MZV's and M$t$V's to a generating series identity which holds in the limit.  It holds when \( (f_k)_{k=0}^\infty \) and \( (g_k)_{k=0}^\infty \) are the series of truncated MZV's and M$t$V's, by the previous result \autoref{prop:zetatail}, and is used to deduce \autoref{thm:symgsrestricted} from \autoref{prop:truncgs}.
	
	\begin{Prop}\label{prop:specialsumconvergence}
		Let \( (f_k)_{k=0}^\infty \) and \( (g_k)_{k=0}^\infty \) be convergent sequences with limits \( F \) and \( G \), respectively.  Assume that the tails of \( F \) and \( G \) satisfy the following convergence rate condition
		\begin{align*}
		F - f_M = O\bigg( \frac{\log^J M}{M}\bigg) \,, \quad\quad	G - g_M = O\bigg( \frac{\log^{J'} M}{M}\bigg) \,,
		\end{align*}
		for some \( J, J' \geq 0 \).  Furthermore, assume that
		\(
		\sum_{k=1}^\infty s_k = S
		\)
		is a convergent series, with \( s_k = O(k^{-\eps}) \), for some \( \eps > 0 \).  (That is, the sequence is perhaps only conditionally convergent.)
		Then
		\[
		\lim_{M\to\infty} \sum_{k=1}^M f_{M+k} g_{M-k} s_k = F G S \,.
		\]
		
		\begin{proof}
			Write
			\begin{align}
			\sum_{k=1}^M f_{M+k} g_{M-k} s_k & = \sum_{k=1}^M (F - (F - f_{M+k})) (G - (G - G_{M-k})) s_k \nonumber \\
			& \begin{aligned}[m]
			= & \sum_{k=1}^M F G s_k - \sum_{k=1}^M (F - f_{M+k}) G s_k \\
			&- \sum_{k=1}^M F (G - G_{M-k}) s_k+ \sum_{k=1}^M (F - f_{M+k})(G - G_{M-k}) s_k .	\end{aligned} \label{sum:decomp}
			\end{align}
			
			The big-$O$ condition implies that (except for \( M = 1 \)), we can find the following \emph{absolute} bound on \( F - f_M \), and likewise \( G - g_M \): there exists \( C \) such that for all \( M > 1 \),
			\[
			\abs{F - f_M} \leq C \frac{\log^J M}{M} \,.
			\]
			This follows since there exists \( k, n_0 \) such that for all \( M \geq n_0 \), \( 	\abs{F - f_M} \leq k \frac{\log^J M}{M} \), as per the definition of big-$O$.  Then since \( \frac{\log^J M}{M} > 0 \) for \( M > 1 \), one can choose larger \( k \) to ensure the first terms \( M = 2, \ldots, n_0 \) terms also satisfy this inequality.  We now analyse each summand of \autoref{sum:decomp} in turn.  
			\medskip
			\paragraph{\em First summand:} The first summand gives
			\[
			\sum_{k=1}^M F G s_k \to F G \sum_{k=1}^\infty s_k = F G S \,,
			\]
			which is the main contribution to our claimed result.
			Now we show all other summands of \autoref{sum:decomp} tend to 0 as \( M \to \infty \).  
			\medskip
			\paragraph{\em Second summand:} The second summand is
			\begin{align*}
			\abs{\sum_{k=1}^M (F - f_{M+k}) G s_k} &\leq G \sum_{k=1}^M C \frac{\log^J(M+k)}{M+k} \cdot \frac{C'}{k^\eps} \\
			& \leq G C C' \sum_{k=1}^M \frac{\log^J(2M)}{M} \frac{1}{k^\eps} \\
			& \leq G C C' \frac{\log^J(2M)}{M} \sum_{k=1}^M \frac{1}{k^\eps}
			\end{align*}
			By considering Riemann sums for the integral $\int_0^M k^{-\eps}dk$ (assuming wlog that \( 0 < \eps < 1 \)),
			\[
			\sum_{k=1}^M \frac{1}{k^\eps} \leq  \frac{M^{1-\eps}}{1-\eps} ,
			\]
			hence
			\[
			\abs{\sum_{k=1}^M (F - f_{M+k}) G s_k} \leq G C C' \frac{\log^J(2M)}{M} \cdot \frac{M^{1-\eps}}{1-\eps} \to 0 
			\]
			as \( M \to \infty \).
			\medskip
			\paragraph{\em Third summand:} Writing the third summand in \autoref{sum:decomp} as
			\[
			\sum_{k=1}^M F (G - G_{M-k}) s_k = F \sum_{k=1}^{M-1} (G - G_{M-k}) s_k + s_M F (G - G_0) \,
			\]
			we need only consider the behaviour of the sum \( \sum_{k=1}^{M-1} \) since \( s_M \to 0 \) as \( M \to \infty \). We have
			\[
			\abs{\sum_{k=1}^{M-1} (G - G_{M-k}) s_k} \leq \sum_{k=1}^{M-1} C'' \frac{\log^{J'}(M-k)}{M-k} \cdot \frac{C'}{k^\eps}
			\]
			Now (ignoring the constant $C'C''$) decompose the sum on the right-hand side into terms with \( k \leq \lfloor M/2 \rfloor \) and those with \( k > \lfloor M/2 \rfloor \), namely
			\begin{equation}\label{split}
			\sum_{k=1}^{M-1}  \frac{\log^{J'}(M-k)}{(M-k) k^\eps} = \sum_{k=1}^{\lfloor M/2 \rfloor}  \frac{\log^{J'}(M-k)}{(M-k) k^\eps}  + \sum_{k=\lfloor M/2 \rfloor + 1}^{M-1} \frac{\log^{J'}(M-k)}{(M-k) k^\eps} .
			\end{equation}
			In the first sum, we can make the estimate \[
			\frac{\log^{J'}(M-k)}{M-k} \leq \frac{\log^{J'}{M}}{\lfloor M/2 \rfloor}
			\]
			to obtain the upper bound
			\[
			\sum_{k=1}^{\lfloor M/2 \rfloor}  \frac{\log^{J'}(M-k)}{(M-k) k^\eps}	\leq \frac{\log^{J'}{M}}{\lfloor M/2 \rfloor} \sum_{k=1}^{\lfloor M/2 \rfloor}  \frac{1}{k^\eps} \leq \frac{\log^{J'}{M}}{\lfloor M/2 \rfloor}\cdot \frac{\lfloor M/2 \rfloor ^{1-\eps}}{1-\eps}
			\]
			as before, which we see goes to 0 as \( M \to \infty \).  In the second sum, we reverse the summation via \( k' = M - k \), to obtain
			\[
			\sum_{k=\lfloor M/2 \rfloor + 1}^{M-1} \frac{\log^{J'}(M-k)}{(M-k) k^\eps} = \sum_{k'=1}^{\lfloor (M-1)/2 \rfloor} \frac{\log^{J'}{k'}}{k' \, (M-k')^\eps} .
			\]
			With the bound \( k' \leq \lfloor (M-1)/2 \rfloor \leq \lfloor M/2 \rfloor \leq M/2 < M \), we can bound the sum above by
			\[
			\frac{\log^{J'}{\lfloor M/2 \rfloor}}{{\lfloor M/2 \rfloor}^\eps} \sum_{k'=1}^{\lfloor (M-1)/2 \rfloor} \frac{1}{k'} \leq \frac{\log^{J'} \lfloor M/2 \rfloor}{{\lfloor M/2 \rfloor}^\eps} \bigg(1 + \log\bigg(\frac{M}{2}\bigg)\!\bigg) \,,
			\]
			and see it also goes to 0 as \( M \to \infty \).  So by \autoref{split} the third summand in \autoref{sum:decomp} goes to 0 as \( M \to \infty \).
			\medskip
			\paragraph{\em Fourth summand:} Finally, we bound the fourth summand in \autoref{sum:decomp} as follows:
			\begin{align*}
			&\abs{\sum_{k=1}^M (F - f_{M+k})(G - G_{M-k}) s_k} \\
			& \leq C C' C'' \sum_{k=1}^{M-1} \frac{\log^J(M+k)}{M+k} \frac{\log^{J'}(M-k)}{M-k} \frac{1}{k^\eps} + (F - f_M)(G - G_0) s_k  \\
			& \leq C C' C'' \frac{\log(2M)^J}{M} \sum_{k=1}^{M-1} \frac{\log^{J'}(M-k)}{M-k} \frac{1}{k^\eps} + (F - f_M)(G - G_0) s_k ,
			\end{align*}
			but the argument above for the third summand, together with the facts \( F - f_M \to 0 \) and \( s_k \to 0 \) as \( M \to \infty \), shows that the fourth summand in \autoref{sum:decomp} also goes to 0. \medskip
			
			\paragraph{\em Conclusion:} The first summand in \autoref{sum:decomp} is the only one which survives, so we have proved the proposition.
		\end{proof}
	\end{Prop}

\end{document}